\documentclass[a4,11pt]{amsart}
\usepackage[top=3cm, bottom=3cm, left=2.5cm, right=2.5cm]{geometry}

\usepackage{amssymb,amsmath,amsthm,txfonts, amsfonts}
\usepackage{amscd}
\usepackage[all]{xy} 
\usepackage[colorlinks=true,citecolor=blue,linkcolor=blue]{hyperref}

\usepackage{graphicx}
\usepackage{amsmath,amsfonts}
\usepackage{bm}
\usepackage{physics}
\usepackage{enumerate}
\usepackage{color,soul}
\usepackage{hyperref}

\usepackage{mathrsfs}
\usepackage{color}
\usepackage{cite}
\usepackage{nicefrac}
\usepackage{hyperref}
\usepackage{textcomp}
\usepackage{protosem}

\usepackage{graphicx}
\usepackage{lscape}
\usepackage[hang,small,bf]{caption}
\usepackage[subrefformat=parens]{subcaption}
\newtheorem{thm}{Theorem}[section]
\newtheorem{lem}[thm]{Lemma}
\newtheorem{cor}[thm]{Corollary}
\newtheorem{prop}[thm]{Proposition}
\theoremstyle{definition}
\newtheorem{defn}[thm]{Definition}
\theoremstyle{remark}
\newtheorem{rem}[thm]{\textbf{Remark}}
\newtheorem{rems}[thm]{\textbf{Remarks}}

 \makeatletter
    
    \@addtoreset{equation}{section}
  \makeatother

\makeatletter
      \def\@makefnmark{%
         \leavevmode
            \raise.9ex\hbox{\check@mathfonts
                \fontsize\sf@size\z@\normalfont%
                            \@thefnmark}%
       }
      \makeatother

\makeatletter
\@namedef{subjclassname@2020}{\textup{2020} Mathematics Subject Classification}
\makeatother

\begin{document}

\title[]{MHS equilibria in the non-resistive limit to the randomly forced resistive magnetic relaxation equations}

\author[]{Ken Abe}
\author[]{In-Jee Jeong}
\author[]{Federico Pasqualotto}
\author[]{Naoki Sato}
\date{}
\address[Ken Abe]{Department of Mathematics, Graduate School of Science, Osaka Metropolitan University, 3-3-138 Sugimoto, Sumiyoshi-ku Osaka, 558-8585, Japan}
\email{kabe@omu.ac.jp}
\address[In-Jee Jeong]{Department of Mathematical Sciences and RIM, Seoul National University, Seoul 08826, Korea}
\email{injee\_j@snu.ac.kr}
\address[Federico Pasqualotto]{Department of Mathematics,
University of California San Diego, La Jolla, CA, 92093, USA \newline \
and June E Huh Center for Mathematical Challenges, Korea Institute of Advanced Study, 
Seoul 02455, Korea}
\email{fpasqualotto@ucsd.edu}
\address[Naoki Sato]{National Institute for Fusion Science, 322-6 Oroshi-cho Toki-city, Gifu 509-5292, Japan	}
\email{sato.naoki@nifs.ac.jp}

\date{\today}

\begin{abstract}
We consider randomly forced resistive magnetic relaxation equations (MRE) with resistivity $\kappa>0$ and a force proportional to $\sqrt{\kappa}$ on the flat $d$-torus $\mathbb{T}^{d}$ for $d\geq 2$. We show the path-wise global well-posedness of the system and the existence of the invariant measures, and construct a random magnetohydrostatic (MHS) equilibrium $B(x)$ in $H^{1}(\mathbb{T}^{d})$ with law $\mathcal{D}(B)=\mu_0$ as a non-resistive limit $\kappa\to 0$ of statistically stationary solutions $B_{\kappa}(x,t)$. For $d=2$, the measure $\mu_0$ does not concentrate on any compact sets in $H^{1}(\mathbb{T}^{2})$ with finite Hausdorff dimension. In particular, all realizations of the random MHS equilibrium $B(x)$ are almost surely not finite Fourier mode solutions. 
\end{abstract}

\maketitle

%\tableofcontents

\section{Introduction}

\subsection{Magnetic relaxation equations}

Moffatt \cite{Moffatt85}, cf. \cite{Moffatt21}, introduced the concept of magnetic relaxation equations (MRE), aiming at constructing stationary solutions to the Euler equations (MHS equilibria). We consider the magnetic relaxation equations with hyper-viscosity for $\gamma\geq 0$ \cite{BFV21}, cf. \cite{Brenier14}:  
\begin{equation}
\begin{aligned}
\partial_t B+u\cdot \nabla B-B\cdot \nabla u&=0,\\
\nabla p&=B\cdot \nabla B-(-\Delta)^{\gamma}u,\\
\nabla \cdot u=\nabla \cdot B&=0,
\end{aligned}
\end{equation}
for evolutions of the magnetic field $B$, the pressure function $p$, and the velocity field $u$ of electrically conducting fluids on the flat $d$-torus $\mathbb{T}^{d}=[-\pi,\pi]^{d}$ for $d\geq 2$ with periodic boundary conditions. 

Magnetic relaxation is the idea of constructing MHS equilibria 
\begin{equation}
\begin{aligned}
\nabla p&=B\cdot \nabla B,\\
\nabla \cdot B&=0,
\end{aligned}
\end{equation}
by long-time limits of magnetic fields with vanishing velocity fields, i.e., (1.1) for $u=0$. Magnetic field lines of smooth solutions to (1.1) keep their topology and do not reconnect in finite time by the frozen-in law, e.g., \cite[I.5]{AK21}, \cite[1.4.3]{BV22}. Indeed, the magnetic energy $\mathscr{E}(B)=\frac{1}{2}\int_{\mathbb{T}^{d}}|B|^{2}dx$ of (1.1) decreases by the energy equality (with velocity dissipation): 
\begin{align*}
\mathscr{E}(B(t))+ \int_{0}^{t}||u||_{\dot{H}^{\gamma}}^{2}(s)ds=\mathscr{E}(B(0)),\quad t\geq 0.
\end{align*}
On the other hand, conservation of mean-square potential ($d=2$) and magnetic helicity ($d=3$) 
\begin{align*}
\mathscr{M}(B)=\int_{\mathbb{T}^{2}}|\textrm{curl}^{-1}B|^{2}dx,\quad 
\mathscr{H}(B)=\int_{\mathbb{T}^{3}} \textrm{curl}^{-1}B\cdot B dx,
\end{align*}
bound $\mathscr{E}(B)$ from below and prevent the magnetic field from the trivial state at $t=\infty$, where $\nabla^{\perp}={}^{t}(-\partial_2,\partial_1)$ and 
\begin{align*}
\textrm{curl}^{-1}B=
\begin{cases}
\ -\nabla^{\perp} \cdot (-\Delta)^{-1}B,\quad &d=2, \\
\ \nabla \times (-\Delta)^{-1}B,\quad &d=3.
\end{cases}
\end{align*}
For $d=2$, Casimir invariants $\mathscr{C}(B)=\int_{\mathbb{T}^{2}}f(\textrm{curl}^{-1}B)dx$ are conserved for arbitrary functions $f$. MHS equilibria obtained from initial data $B_0$ via MRE may lose topological equivalence with $B_0$ due to the appearance of tangential discontinuities of magnetic field lines at $t=\infty$. The MHS equilibria obtained are alternatively called \textit{topologically accessible} from $B_0$ \cite[8.2.1]{Moffatt21}, cf. \cite{EPS23}.

Beekie et al. \cite[Theorems 3.1 and 4.1]{BFV21} demonstrated the global well-posedness of (1.1) for arbitrary divergence-free initial data $B_0\in H^{s}(\mathbb{T}^{d})$ and $s, \gamma>d/2+1$ and the convergence of the velocity field $\lim_{t\to\infty}||\nabla u||_{L^{\infty}}=0$; see also \cite{BKS23}. The work \cite{BFV21} also shows the asymptotic stability of the specific 2D MHS equilibrium $B={}^{t}(1,0)$; see also \cite[Corollary 1.6]{JinTan}. These works, however, stop short of showing convergence (in a general setting) of the relaxation procedure to an MHS equilibrium. In addition, to the authors' knowledge, there do not exist many general algorithms to construct MHS equilibria whose convergence is rigorously proved (see, for instance, \cite{CP22} and \cite{Bhatt}).

\subsection{The statement of the main results}

\subsubsection{A convergence to a random MHS equilibrium}

This study establishes a general construction of MHS equilibria by means of a randomly forced version of MRE. The main theorem shows that, for general initial data, the system converges towards MHS equilibria in a random setting, thereby constructing infinitely many such equilibria arising from a limiting probability distribution. More specifically, we consider the randomly forced resistive MRE with resistivity $\kappa>0$ and a force proportional to $\sqrt{\kappa}$: 
\begin{equation}
\begin{aligned}
\partial_t B+u\cdot \nabla B-B\cdot \nabla u&=\kappa \Delta B +\sqrt{\kappa}\partial_t \zeta,\\
\nabla p&=B\cdot \nabla B-(-\Delta)^{\gamma}u,\\
\nabla \cdot u=\nabla \cdot B&=0,
\end{aligned}
\end{equation}
for $\gamma>d/2$ and average-zero magnetic fields in the Hilbert space 
\begin{align}
H=\left\{B\in L^{2}(\mathbb{T}^{d})\ \middle|\ \nabla \cdot B=0,\ <B>=\int_{\mathbb{T}^{d}}Bdx=0\ \right\},
\end{align}
endowed with the inner product $(\cdot,\cdot)_{H}$. The random field $\zeta$ is a Wiener process,
\begin{equation}
\begin{aligned}
\zeta&=\sum_{j=1}^{\infty}b_j e_j(x)\beta_j(t),
\end{aligned}
\end{equation}
consisting of a sequence of independent Brownian motions $\{{\beta}_{j}(t)\}_{j=1}^{\infty}$, and a complete orthonormal system $\{e_{j}\}_{j=1}^{\infty}$ on $H$ consisting of eigenfunctions of the Stokes operators with eigenvalues $\{\lambda_j\}_{j=1}^{\infty}$, $0<\lambda_1\leq \lambda_2\leq \cdots$ and $\lambda_j\to\infty$. For a sequence $\{b_{j}\}_{j=1}^{\infty}$, we set the non-negative constants 
\begin{align}
\mathsf{C}_s=\sum_{j=1}^{\infty}\lambda_j^{s}b_j^{2},\quad s\in \mathbb{R}.
\end{align}\\
For $d=3$, we choose $\{e_{j}\}_{j=1}^{\infty}$ by eigenfunctions of the rotation operator with the eigenvalues $\tau_j\in \mathbb{R}$, i.e., $\nabla \times e_j=\tau_j e_j$ and $\tau_j^{2}=\lambda_j$, and set the constant (smaller than $\mathsf{C}_{-1/2}$),
\begin{align}
\mathcal{C}_{-\frac{1}{2}}=\sum_{j=1}^{\infty}\frac{1}{\tau_j}b_j^{2}.
\end{align}
We construct such a complete orthonormal system (with explicit forms) using complex plane Beltrami waves \cite{DeS13}, \cite{BV19b}. 

The introduction of resistivity and random force in (1.3) is motivated by statistical studies on two-dimensional hydrodynamic turbulence \cite{Kuk04}, \cite{Kuk12}. The resistive term $\kappa \Delta B$, making magnetic fields simpler, and the random force $\partial_t \zeta$ stirring them up, are balanced with power $\sqrt{\kappa}$ for small $0<\kappa<<1$. In contrast to two-dimensional hydrodynamic turbulence, we will see that statistical equilibria to (1.3) approximate the MHS equilibria (1.2) for small $0<\kappa<<1$ with order $u_{\kappa}(x,t)=O(\sqrt{\kappa})$ and $B_{\kappa}(x,t)=O(1)$ (by taking means). 

More specifically, we consider limits of the system (1.3) in the order: (i) $t\to\infty$ and (ii) $\kappa\to 0$ (The order (i) $\kappa\to 0$ and (ii) $t\to\infty$ may reduce the issue to the deterministic problem). Namely,  
\begin{itemize}
\item[(i)] We first show the path-wise global well-posedness of (1.3) for random initial data and the convergence of the random solution $B(t): (\Omega, \mathcal{F},\mathbb{P})\to (H, \mathcal{B}(H))$ in law, i.e., $\mathcal{D}(B(t))\to \mu_{\kappa}$ on $H$ as $t\to\infty$ (taking the Ces\'aro mean), where $\mathcal{B}(H)$ denotes the Borel $\sigma$-algebra on $H$. The convergence in law refers to the weak convergence of measures; see \S 4.3. The limit measure $\mu_{\kappa}$ is time-independent (invariant measure) and yields global-in-time random solutions $B_{\kappa}(t)$ to (1.3) with time-independent law  $\mathcal{D}(B_{\kappa}(t))=\mu_{\kappa}$ (\textit{statistically stationary solutions}); see \S 3.6 for more detailed explanations. The random initial data $B_{\kappa}(0)$ is arbitrary with the prescribed law $\mathcal{D}(B_{\kappa}(0))=\mu_{\kappa}$. 
\item[(ii)] We then construct a probability measure and a random MHS equilibrium by taking a non-resistive limit to the invariant measures $\mu_{\kappa}$ and the statistically stationary solutions $B_{\kappa}$ as $\kappa\to0$.
\end{itemize}

The first main result of this study is the following convergence result to a random MHS equilibrium in $H^{1}(\mathbb{T}^{d})$ constructed in this order, cf. \cite[Q2]{BFV21}. We single out this result from a general convergence result to a random MHS equilibrium in $H^{\alpha}(\mathbb{T}^{d})$ for $\alpha\geq 1$ because of the physical importance of discontinuous magnetic fields; see Theorem 5.16 for a general result. The convergence in $C([0,\infty); H^{-\varepsilon}(\mathbb{T}^{d}))$ is a uniform convergence in $C([0,T]; H^{-\varepsilon}(\mathbb{T}^{d}))$ for arbitrary $T>0$.

\begin{thm}[Convergence to random MHS equilibria]
Let $d\geq 2$, $\gamma> d/2$, and $\mathsf{C}_0>0$. The system (1.3) admits an invariant measure $\mu_{\kappa}$ on $H^{1}(\mathbb{T}^{d})$. For $\varepsilon>0$, there exists a sequence of invariant measures $\{\mu_{\kappa}\}$ weakly converging to a measure $\mu_0$ on $H^{1-\varepsilon}(\mathbb{T}^{d})$ as $\kappa\to 0$ such that there exist statistically stationary solutions to (1.3), $B_{\kappa}(t): (\Omega, \mathcal{F},\mathbb{P})\to (H, \mathcal{B}(H))$ with law $\mathcal{D}(B_{\kappa})=\mu_{\kappa}$, and a random variable $B: (\Omega, \mathcal{F},\mathbb{P})\to (H, \mathcal{B}(H))$ on some probability space $(\Omega, \mathcal{F},\mathbb{P})$ such that for $\varepsilon>0$,
\begin{align}
B_{\kappa}(x,t)\to B(x)\quad \textrm{in}\ C([0,\infty); H^{-\varepsilon}(\mathbb{T}^{d}))\cap L^{2}_{\textrm{loc}}([0,\infty); H^{1-\varepsilon}(\mathbb{T}^{d})),\quad a.s.\quad \textrm{as}\ \kappa\to 0.
\end{align}
The limit $B\in H\cap H^{1}(\mathbb{T}^{d})$ is an MHS equilibrium (1.2) with some pressure function almost surely and $\mathcal{D}(B)=\mu_0$.
\end{thm}

\begin{rem}
For $d=2$ and $d=3$, the measure $\mu_0$ satisfies the equalities: 
\begin{align}
\mathbb{E}||B ||_{H}^{2}
&=\int_{H}||B ||_{H}^{2}\mu_0 (dB)=\frac{\mathsf{C}_{-1}}{2}, \\
\mathbb{E}\left(\nabla\times  B, B   \right)_H
&=\int_{H}\left( \nabla\times B, B   \right)_H\mu_0 (dB)
=\frac{\mathcal{C}_{-\frac{1}{2}}}{2}.
\end{align}
See Theorem 5.15. The left-hand sides denote the means of the random variables $||B ||_{H}^{2}$ and $\left(\nabla\times B, B   \right)_H$ for the random MHS equilibrium $B$ in Theorem 1.1. For $d=2$, $\mathsf{C}_{-1}\neq 0$ and $B\neq 0$ almost surely, i.e., $\mu_0\neq \delta_0$ for the Dirac mass $\delta_0$ concentrating at the origin. For $d=3$ and $\mathcal{C}_{-\frac{1}{2}}\neq 0$, $\mu_0\neq \delta_0$. The quantity $\left( \nabla\times B, B   \right)_H=\left((-\Delta)\textrm{curl}^{-1} B, B   \right)_H$ coincides with the modified magnetic helicity \cite{CP22}. For $d\geq 4$, similar equalities are unknown.
\end{rem}

\subsubsection{Finite Fourier mode solutions and the support of the measure}

The realizations of the random MHS equilibrium $B$ in Theorem 1.1 take states in the support of the measure $\mu_0$ (with probability one) 
\begin{align*}
\textrm{spt}\ \mu_0=\bigcap \left\{F\subset H\cap H^{1}(\mathbb{T}^{d})\ \middle|\  \mu_0(F)=1,\ F: \textrm{closed}\right\}.
\end{align*}
The measure $\mu_0$ is supported on MHS equilibria, i.e., $\textrm{spt}\ \mu_0\subset F_{MHS}$ for 
\begin{align*}
F_{\textrm{MHS}}=\left\{B\in H\cap H^{1}(\mathbb{T}^{d})\ \middle|\ B\ \textrm{is an MHS equilibrium (1.2)\ }\right\}.
\end{align*}

\clearpage

 \begin{figure}[h]
\vspace{5pt} 
 \begin{minipage}[b]{0.44\linewidth}
\hspace{-40pt}
\includegraphics[scale=0.14]{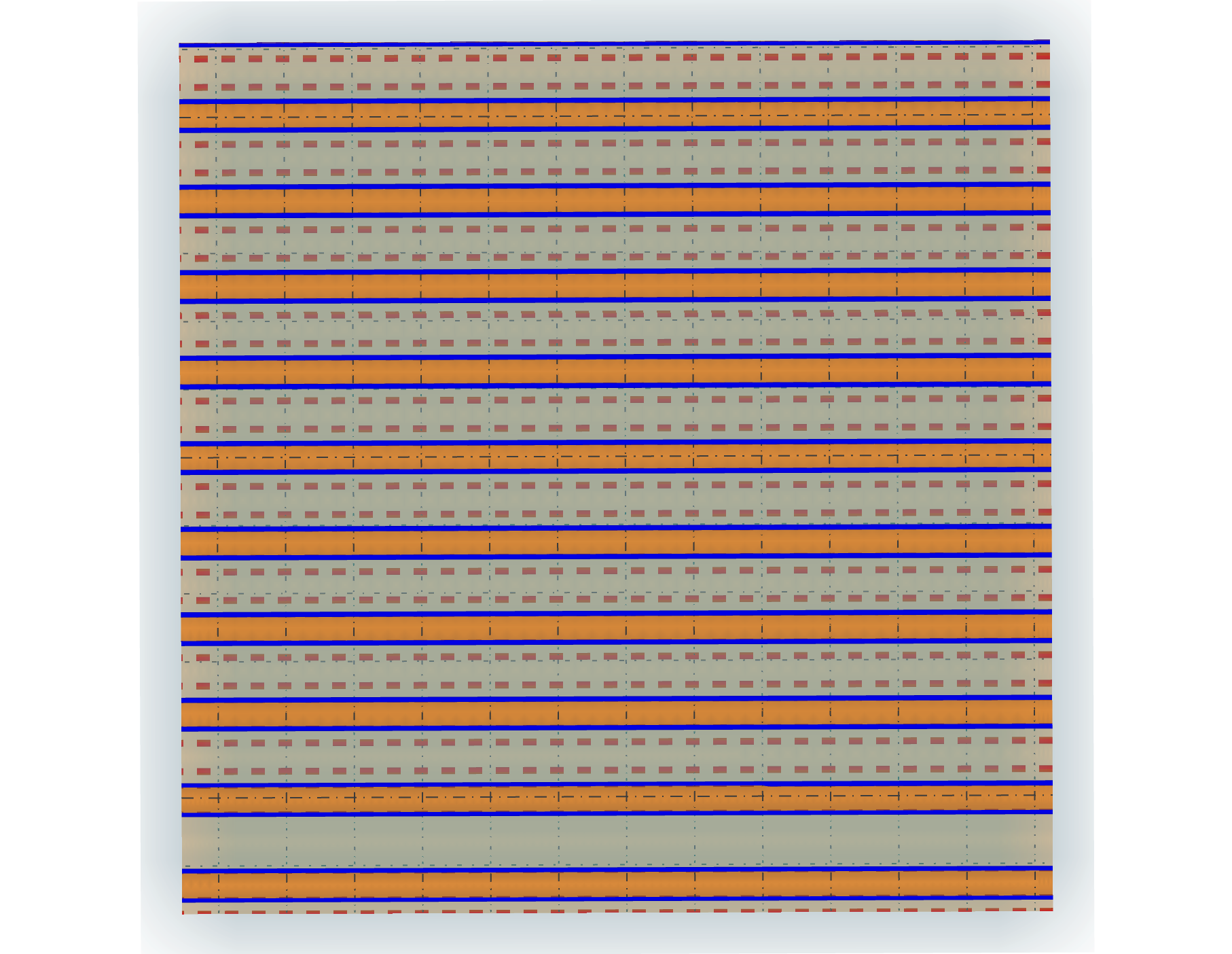}
\end{minipage}
 \begin{minipage}[b]{0.44\linewidth}
\includegraphics[scale=0.11]{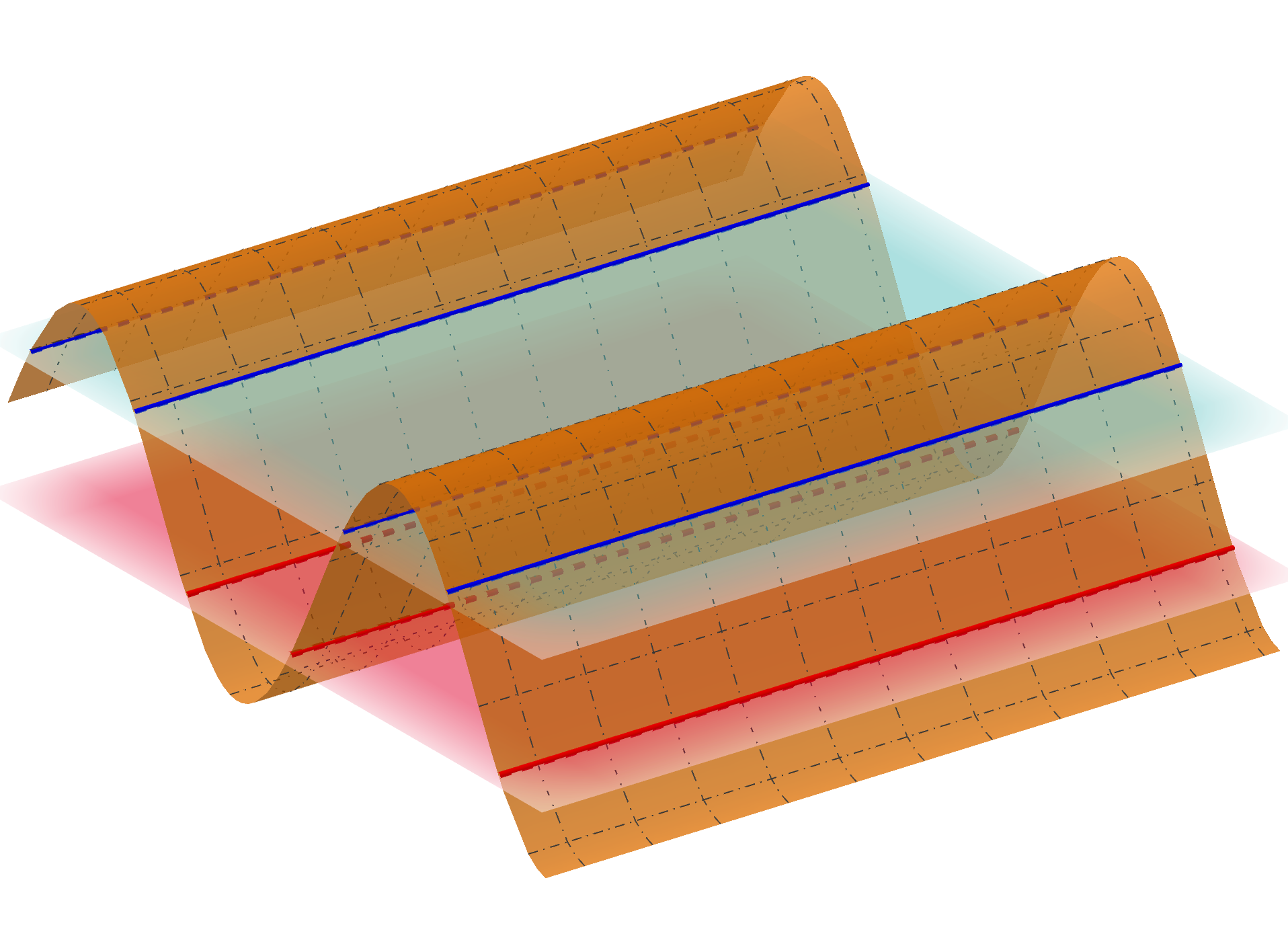}
  \end{minipage}\\
\vspace{20pt} 
\subcaption{The sheared magnetic field (the Kolmogorov flow) $B={}^{t}(\sin{y},0)$ (The eigenfunction of the least eigenvalue $|k|^{2}=1$)}
\vspace{30pt} 
  \begin{minipage}[b]{0.44\linewidth}
\hspace{-40pt}
\includegraphics[scale=0.14]{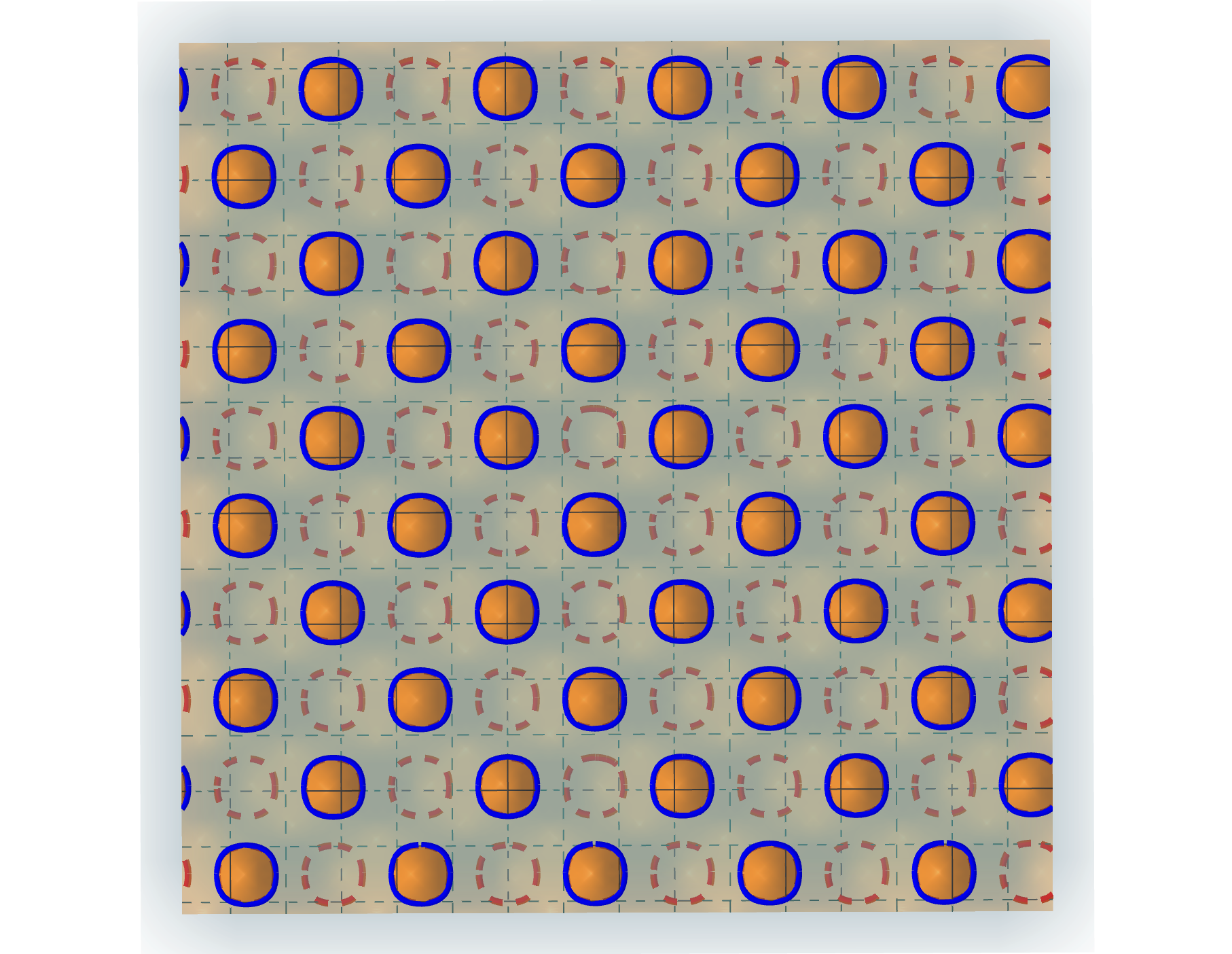}
  \end{minipage}
  \begin{minipage}[b]{0.44\linewidth}
\includegraphics[scale=0.12]{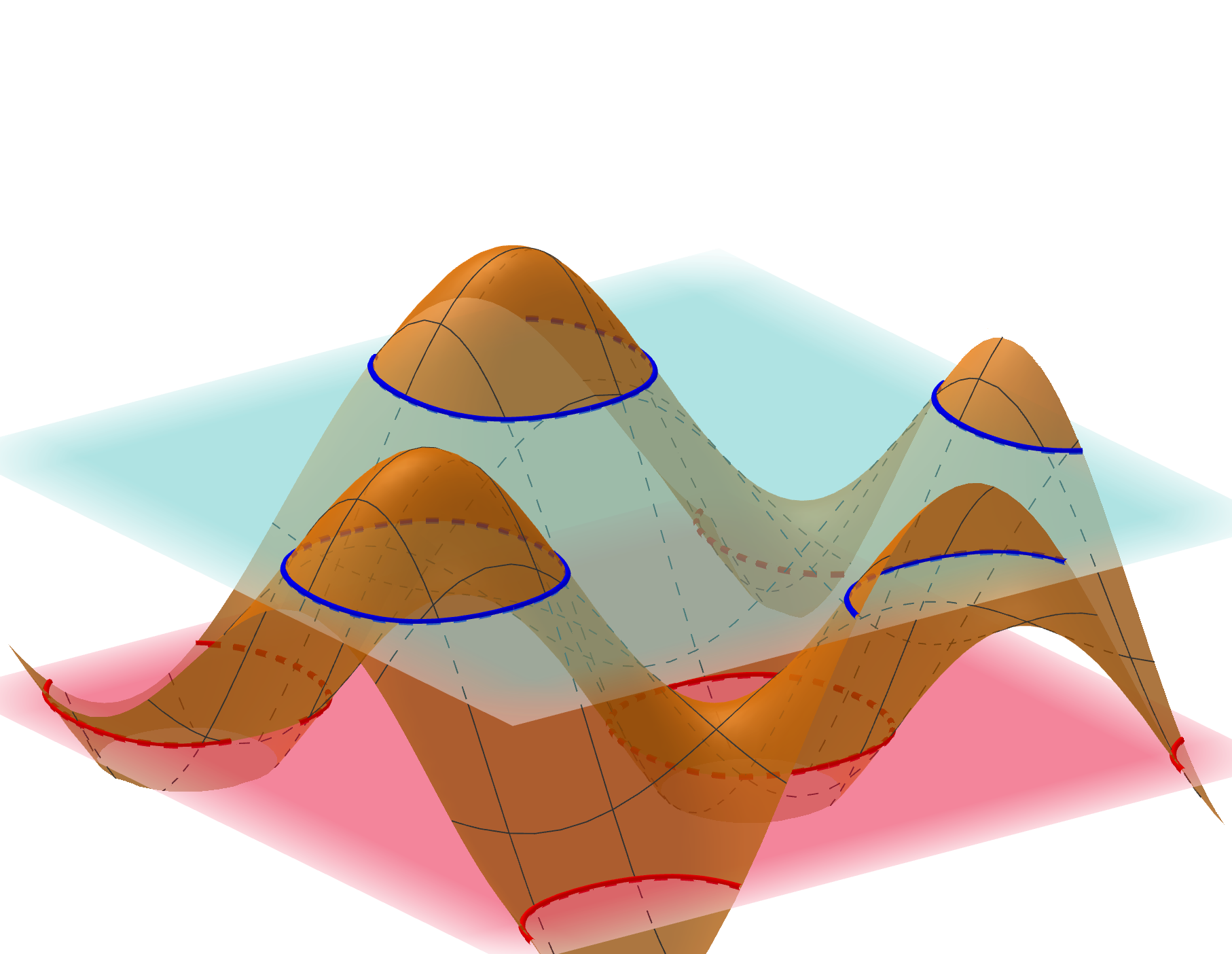}
\end{minipage}\\
\vspace{20pt} 
\subcaption{The magnetic field with islands $B={}^{t}(\sin{x}\sin{y}, \cos{x}\cos{y})$ (The eigenfunction of the second least eigenvalue $|k|^{2}=2$)}
\end{figure}

\setcounter{figure}{0}

\begin{figure}[h]
\caption{2D finite Fourier mode MHS equilibria (eigenfunctions of the Stokes operator). The blue (resp. red) lines represent magnetic field lines in the region where the current field $J=\nabla^{\perp}\cdot B$ is positive (resp. negative). The orange surfaces represent graphs of the flux functions $\phi=\textrm{curl}^{-1}B$.}
\end{figure}

\clearpage

An important subset of smooth MHS equilibria (with explicit forms) in $F_{MHS}$ is a set of finite Fourier mode solutions \cite{EHS17}:
\begin{align*}
F_{\textrm{FFM}}=\left\{\ B=\sum_{k\in K}\hat{B}_ke^{ik\cdot x}\in F_{MHS}\  \middle|\  K=-K\subset \mathbb{Z}^{d}_{0},\ |K|<\infty\  \right\}.
\end{align*}
(By the symmetry of $K$, $B=\sum_{k\in K}\hat{B}_ke^{ik\cdot x}$ is real-valued.) For $d=2$, the set $F_{\textrm{FFM}}$ includes the set of all eigenfunctions of the Stokes operator
\begin{align*}
F_{\textrm{ESO}}=\left\{\ B\in H\cap C^{\infty}(\mathbb{T}^{2})\  \middle|\ -\Delta B=|k|^{2}B,\  k\in \mathbb{Z}^{2}_{0}\ \right\}\subset F_{\textrm{FFM}}.
\end{align*}
(Eigenfunctions of the Stokes operator $B=\nabla^{\perp}\phi$ satisfy the equivalent form of (1.2): $B\cdot \nabla J=0$ for $J=\Delta \phi$ by $-\Delta \phi =|k|^{2}\phi$). Eigenfunctions for the least and second least eigenvalues $|k|^{2}=1$ and $|k|^{2}=2$ describe specific sheared magnetic fields and magnetic fields with islands; see Figure 1. The eigenvalue $|k|^{2}=1$ has multiplicity four with eigenfunctions expressed by superpositions of $\phi=\cos{x}, \sin{x}, \cos{y}, \sin{y}$. The eigenvalue $|k|^{2}=2$ has multiplicity four with eigenfunctions expressed by superpositions of $\phi=\cos{(x+y)}$, $\sin{(x+y)}$, $\cos{(x-y)}$, $\sin{(x-y)}$.

For $d=3$, the set $F_{\textrm{FFM}}$ includes the set of all eigenfunctions of the rotation operator (strong Beltrami fields)
\begin{align*}
F_{\textrm{ERO}}=\left\{\ B\in H\cap C^{\infty}(\mathbb{T}^{3})\  \middle|\ \nabla \times  B=\pm |k| B,\  k\in \mathbb{Z}^{3}_{0}\ \right\}\subset F_{\textrm{FFM}}.
\end{align*}
(Eigenfunctions of the rotation operator $B$ satisfy the equivalent form of (1.2): $\nabla P=J\times B$ for $J=\nabla \times B$ with constant $P=p-|B|^{2}/2$. Note that $F_{\textrm{ERO}}\subset F_{\textrm{ESO}}$. Moreover, $F_{\textrm{ESO}}\cap F_{\textrm{FFM}}=F_{\textrm{ERO}}$ \cite{KY22}; see below.) The least positive eigenvalue of the rotation operator is $|k|=1$ with multiplicity six, and its eigenfunctions are expressed by superpositions of 
\begin{align*}
\mathbf{a}(z)=(\sin z, \cos z,0),\quad \mathbf{b}(x)=(0,\sin x, \cos x), \quad \mathbf{c}(y)=(\cos y, 0,\sin y),
\end{align*}
and $\mathbf{a}(z-\pi/2)$, $\mathbf{b}(x-\pi/2)$, $\mathbf{c}(y-\pi/2)$; see Remark A.5. The eigenfunctions include the ABC flow \cite[II]{AK21} 
\begin{align*}
\qquad \mathsf{A}\mathbf{a}+\mathsf{B}\mathbf{b}+\mathsf{C}\mathbf{c}=(\mathsf{A}\sin{z}+\mathsf{C}\cos{y}, \mathsf{B}\sin{x}+\mathsf{A}\cos{z}, \mathsf{C}\sin{y}+\mathsf{B}\cos{x}),\quad \mathsf{A}, \mathsf{B}, \mathsf{C}\in \mathbb{R}, 
\end{align*}
which has chaotic field lines except for $\mathsf{A}=0$, $\mathsf{B}=0$, or $\mathsf{C}=0$, cf.  \cite[12 (vii)]{Moffatt21}, \cite[Q3]{BFV21}. The set $F_{\textrm{ERO}}$ also includes magnetic fields with field lines of arbitrary knotted topology \cite{EPS17}.

The sets $F_{\textrm{ESO}}$ and $F_{\textrm{ERO}}$ include magnetic energy minimizers with constant mean-square potential $\mathscr{M}(B)=m>0$ for $d=2$ and constant helicity $\mathscr{H}(B)=h\in \mathbb{R}$ for $d=3$, respectively. Those minimizers are important candidates for self-organized states at long-time limits of 2D and 3D MHD turbulence; see \cite{Hasegawa85}, \cite{Biskamp93}, \cite{FLS22} for reviews. For $d=2$, the constraint can be replaced by arbitrary Casimir invariants $\mathscr{C}(B)$. 

\begin{table}[h]
\begin{tabular}{|c|c|c|c|}
\hline
Constraints      & Dimensions & Euler-Lagrange equations                  & Spaces             \\ \hline
$\mathscr{M}(B)$ & $2$        & $-\Delta \phi=\phi$                       & $F_{\textrm{ESO}}$ \\ \hline
$\mathscr{C}(B)$ & $2$        & $-\Delta \phi=\lambda f'( \phi)$ & -                  \\ \hline
$\mathscr{H}(B)$ & $3$        & $\nabla \times B=\pm B$                   & $F_{\textrm{ERO}}$ \\ \hline
\end{tabular}
\caption{Magnetic energy minimizers with conserved constraints}
\end{table}

Elements of $F_{\textrm{FFM}}$ have a constraint on finite and symmetric sets $K\subset \mathbb{Z}^{2}_{0}$ \cite{EHS17}: the set $K$ must be a subset of a line passing through the origin in $\mathbb{Z}^{2}_{0}$ or a subset of a circle centered at the origin in $\mathbb{Z}^{2}_{0}$. For $d=3$, the set $K$ must be a subset of a line passing through the origin in $\mathbb{Z}^{3}_{0}$, a subset of a plane containing the origin, or a subset of a sphere centered at the origin in $\mathbb{Z}^{3}_{0}$. In the last case, finite Fourier mode solutions must be elements of $F_{\textrm{ERO}}$ \cite{KY22}.

\begin{figure}[h]
  \begin{minipage}[b]{0.44\linewidth}
\hspace{63pt}
\includegraphics[scale=0.07]{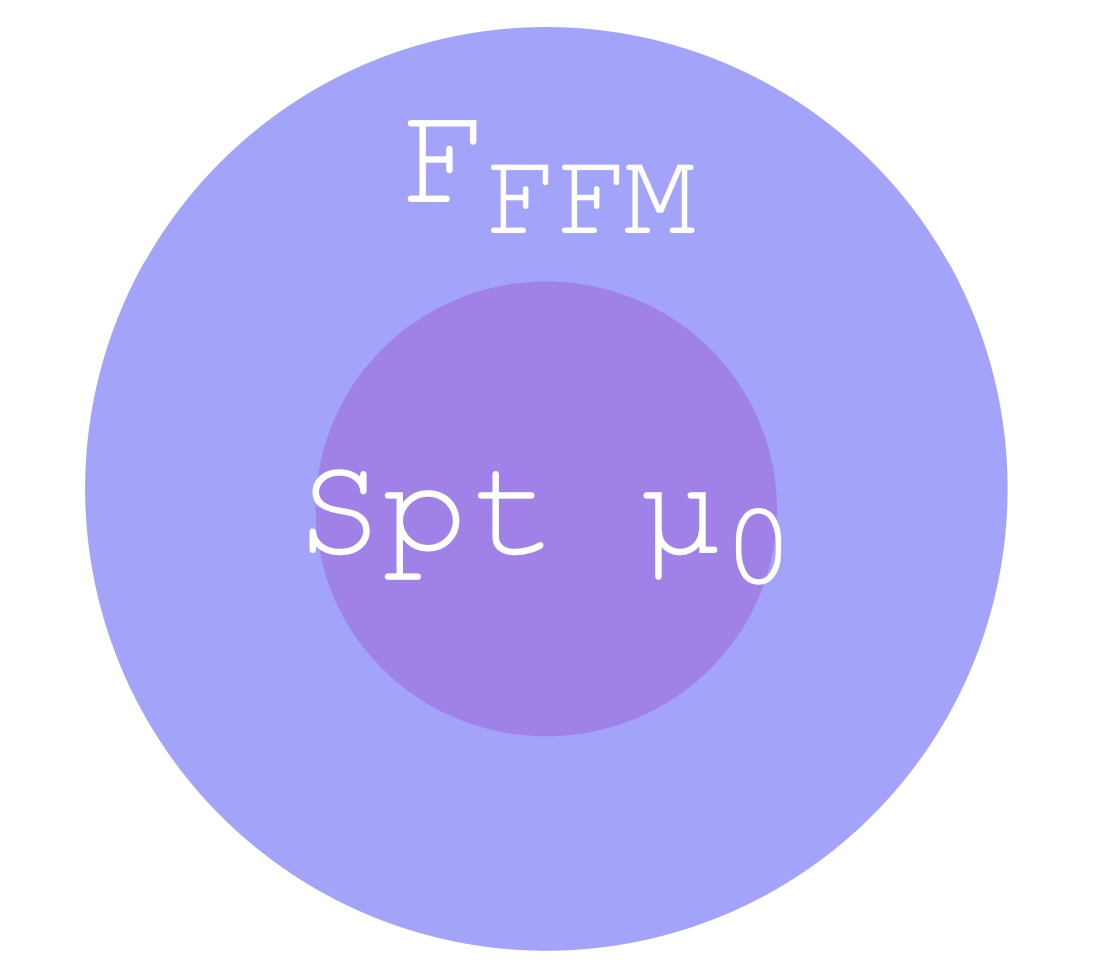}
\vspace{5pt}
\subcaption*{(C1) $\mu_0(F_{\textrm{FFM}})=1$}
  \end{minipage}
 \begin{minipage}[b]{0.44\linewidth}
\hspace{30pt}
\includegraphics[scale=0.07]{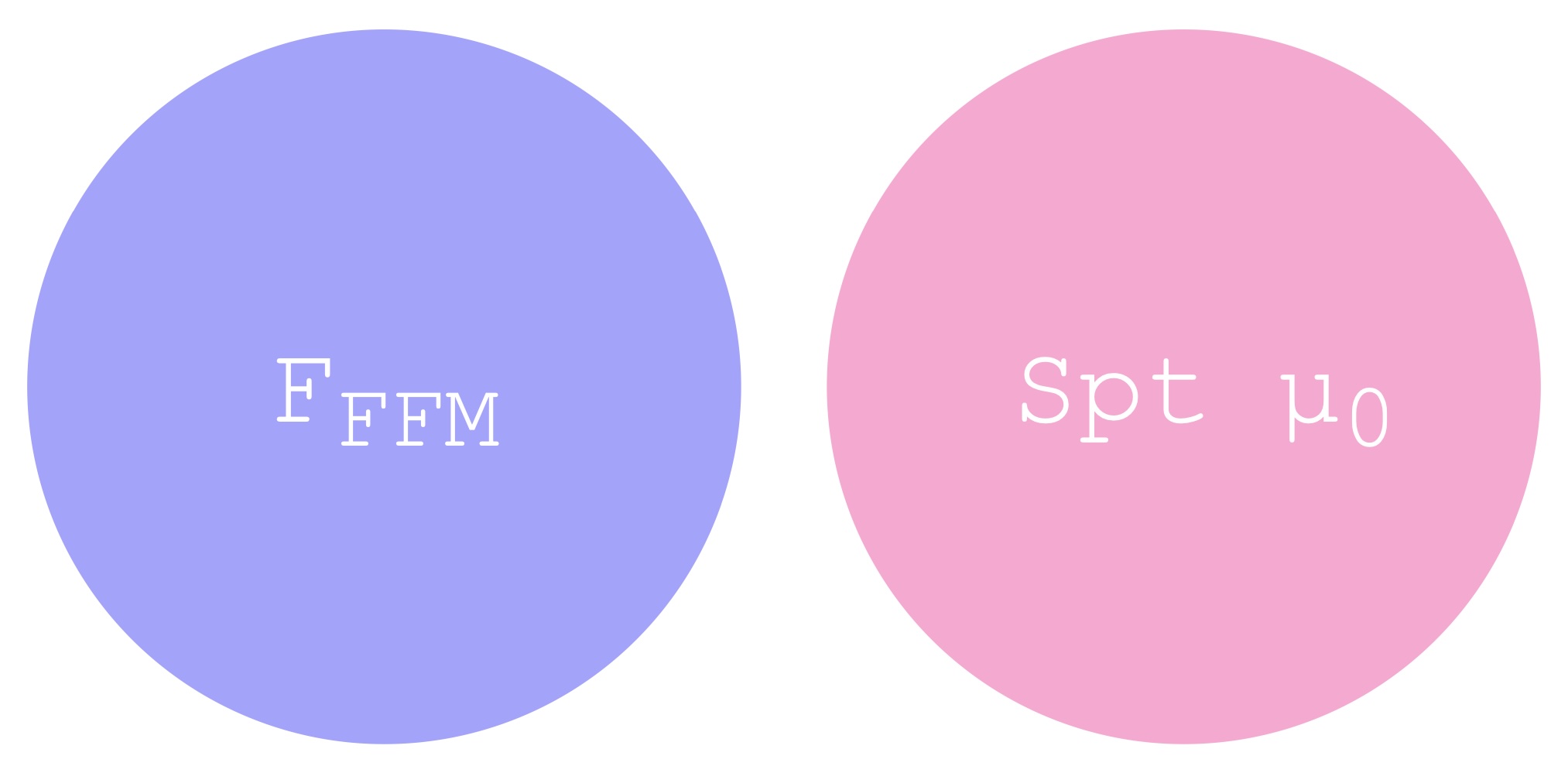}
\vspace{5pt}
\subcaption*{(C2) $\mu_0(F_{\textrm{FFM}})=0$}
\end{minipage}\\
 \begin{minipage}[b]{0.44\linewidth}
\vspace{10pt}
\hspace{130pt}
\includegraphics[scale=0.07]{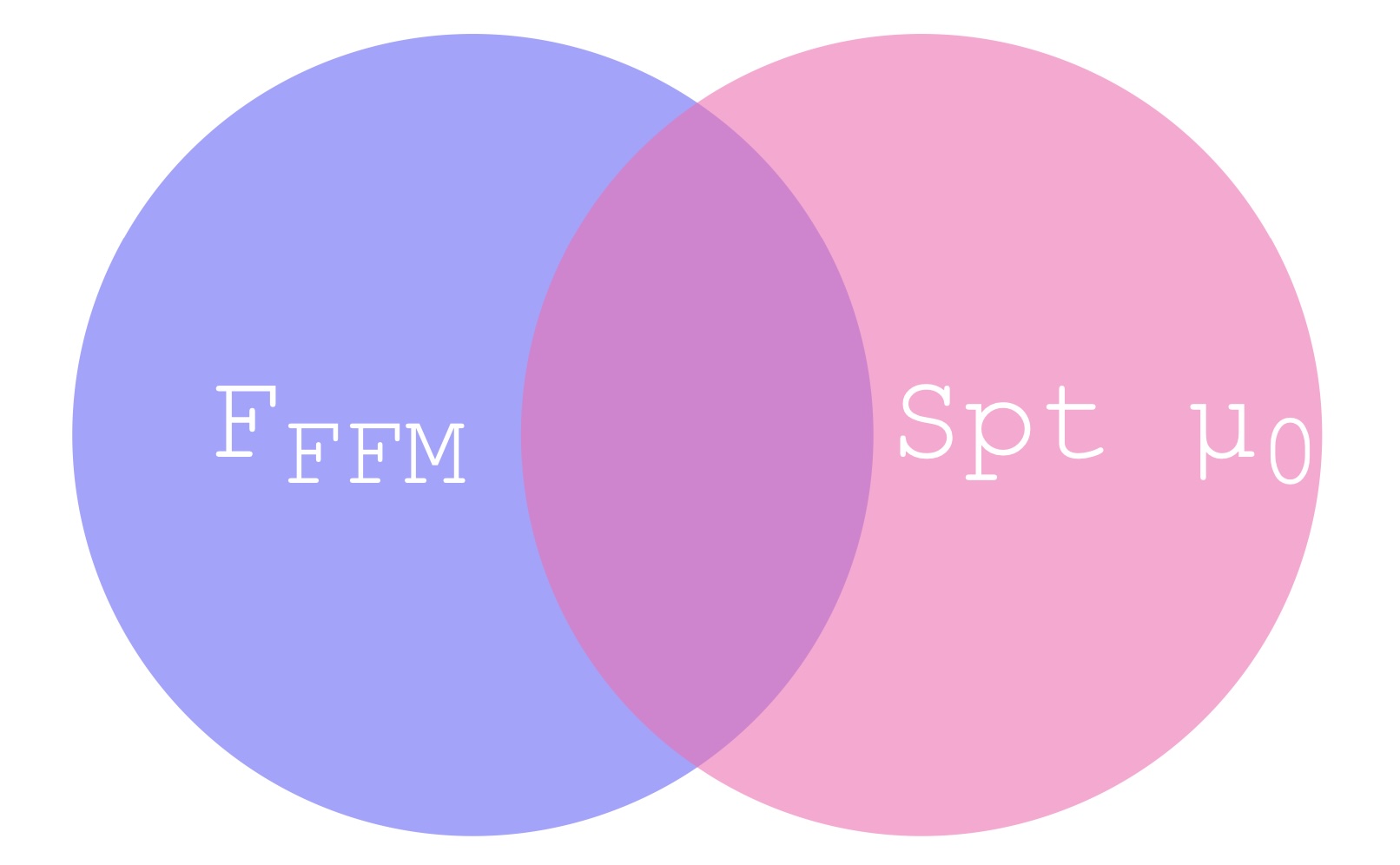}
\end{minipage}
\hspace{45pt}
\begin{minipage}[b]{0.44\linewidth}
\includegraphics[scale=0.07]{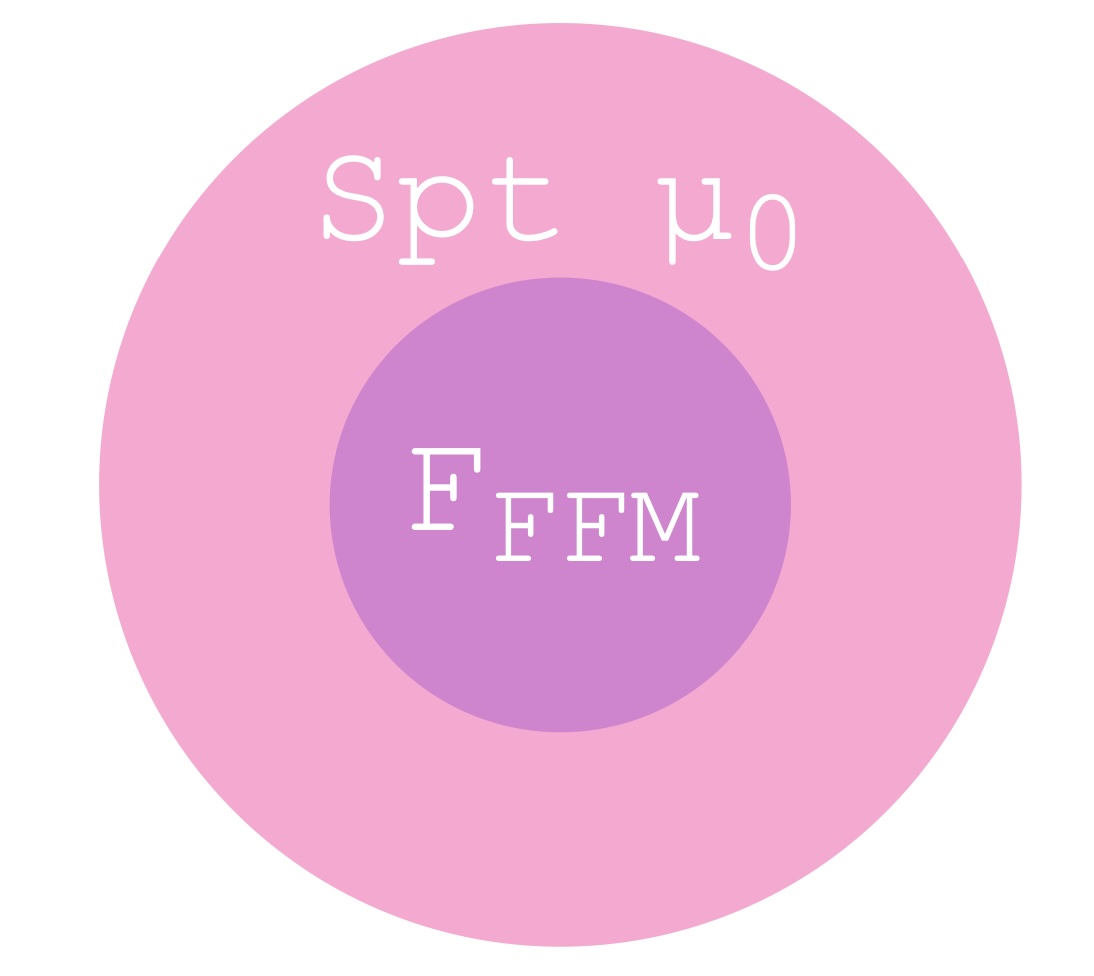}
  \end{minipage}
  \hspace{5pt}  
  \subcaption*{(C3) $0<\mu_0(F_{\textrm{FFM}})<1$}
\end{figure}

\setcounter{figure}{1}

\begin{figure}[h]
\vspace{-15pt}  
\caption{The inclusion relationship between $F_{\textrm{FFM}}$ and $\textrm{spt}\ \mu_0$}
\end{figure}

The inclusion relationship between $F_{\textrm{FFM}}$ and $\textrm{spt}\ \mu_0$ for the measure $\mu_0$ in Theorem 1.1 is one of the following; see Figure 2:\\

\noindent
\textbf{(C1)}: $\mu_0(F_{\textrm{FFM}})=1$. All random equilibrium $B$ realizations are almost surely finite Fourier mode solutions.

\noindent
\textbf{(C2)}: $\mu_0(F_{\textrm{FFM}})=0$. All random equilibrium $B$ realizations are almost surely not finite Fourier mode solutions.

\noindent
\textbf{(C3)}: $0<\mu_0(F_{\textrm{FFM}})<1$. All random equilibrium $B$ realizations are finite Fourier mode solutions and other equilibria with positive probabilities.\\

The second main result of this study shows (C2) for $d=2$.

\begin{thm}[Infinite dimensionarity of the measure]
Let $d=2$. Assume that $b_j\neq 0$ for all $j\in \mathbb{N}$. Then, the measure $\mu_0$ in Theorem 1.1 does not concentrate on compact sets in $H\cap H^{1}(\mathbb{T}^{2})$ with finite Hausdorff dimensions. In particular, $\mu_0(F_{\textrm{FFM}})=0$.
\end{thm}

For $d=3$, (C1)-(C3) are relevant to Grad's conjecture in plasma \cite{Grad67}, \cite{CDG21b}, which states that smooth non-symmetric MHS equilibria do not exist except for Beltrami fields, cf. \cite{CDP}. The work \cite{CP22} constructed smooth MHS equilibria in $H^{\alpha}(\mathbb{T}^{3})$ for any $\alpha\geq 1$ as long-time limits of the Voigt-MHD system without assuming any symmetries; see also \cite{Bhatt}. Furthermore, the recent work \cite{EPP25} constructed H\"older continuous MHS equilibria in $C^{\gamma}(\mathbb{T}^{3})$ for some $0<\gamma<1$ with arbitrary field line topology. Theorem 1.1 is a particular case of a general convergence result to a random MHS equilibrium in $H^{\alpha}(\mathbb{T}^{d})$ for $\alpha\geq 1$; see Theorem 5.16.

The infinite dimensionality of the two-dimensional measure in Theorem 1.3 is based on the fact that 2D MRE (1.1) admits infinitely many Casimir invariants in contrast to 3D MRE, for which helicity is the only available Casimir invariant, cf. \cite{Khesin21}. For $d=3$ and $\mathcal{C}_{-1/2}\neq 0$, $\textrm{spt}\ \mu_0$ is an uncountable set whose elements take infinitely different helicity by the absolute continuity of the law $\mathcal{D}_{\mu_0}(\mathscr{H})$ for the Lebesgue measure in $\mathbb{R}\backslash \{0\}$ (Theorem 6.6). For $d\geq 4$, $\textrm{spt}\ \mu_0$ is uncountable unless $\mu_0=\delta_0$ (Remarks 6.19 (i)). For $d\geq 4$, any conservation law besides the energy equality is unknown for MRE (1.1), unlike the Euler equations, which conserve the helicity of vortex for any odd-dimension, and integrals of any functions of the vorticity function for any even-dimension \cite[II.9]{AK21}.

Kuksin \cite{Kuk04} (see also Kuksin and Shirikyan \cite{Kuk12}) constructed an invariant measure of the Euler equations with several universal properties relevant to the two-dimensional turbulence by inviscid limits of invariant measures to the stochastic Navier--Stokes equations; see \cite[Remark 5.2.19]{Kuk12}, \cite{Latocca} for the hyper-viscosity case. The works \cite{GSV15} and \cite{Bed16} discuss the relationship between the support of the Kuksin measure and long-time limits of the 2D Euler and passive scalar equations; see, e.g.,  \cite{Bouchot}, \cite{KMS23}, \cite{DE} for reviews on statistical mechanics of the two-dimensional turbulence. The Kuksin measure also plays a role in the global well-posedness of the nonlinear Schr\"odinger equations \cite{KS04}, \cite{Shi11}, \cite[5.2.6]{Kuk12}, \cite{Sy21}, and the SQG equations \cite{FS21}.

%The MRE (1.1) can exhibit magnetic reconnection, a topological rearrangement of magnetic field lines at $t=\infty$. Magnetic reconnection is a key self-organization process in laboratory and astrophysical plasmas, such as solar flares, magnetospheric substorms, protostar flares, and fusion plasmas; see, e.g., \cite{Eyink20} and \cite{MasaakiYamada} for reviews. It is worth noting that magnetic fields in turbulent fluids exhibit stochasticity, and violate the conventional frozen-in-law (flux freezing/Alfv\'en's theorem) \cite{EA06}, \cite{Eyink11}, \cite{Eyink13}, \cite{Eyink15}. An alternative concept is the stochastic flux freezing \cite{Eyink11}, a statistical version of flux freezing associated with the stochastic Lagrangian flow map. 

\subsection{The outline of the proof}

We show Theorem 1.1 based on the fluctuation-dissipation method for the 2D Euler equations \cite{Kuk04}, \cite{Kuk12}. The main steps of the proof are as follows: 
\begin{itemize}
\item[\S 2] Global well-posedness of the resistive MRE
\item[\S 3] Path-wise global well-posedness of the randomly forced resistive MRE 
\item[\S 4] The existence of invariant measures
\item[\S 5] The non-resistive limits
\item[\S 6] Absolute continuity of laws 
\end{itemize} 
The system (1.3) is a cubic nonlinear problem for $d\geq 2$, and we work in a subcritical regime, $\gamma > d/2$ and $ B_0 \in H$, cf. \cite{Kuk04}, \cite{Kuk12}. Namely, we consider the following forced system for $\gamma>d/2$:
\begin{equation}
\begin{aligned}
\partial_t B+u\cdot \nabla B-B\cdot \nabla u&=  \kappa \Delta B +f,\\
\nabla p&=B\cdot \nabla B-(-\Delta)^{\gamma}u,\\
\nabla \cdot u=\nabla \cdot B&=0.
\end{aligned}
\end{equation}
For the deterministic force $f$, this system possesses the energy inequality  
\begin{align}
\int_{\mathbb{T}^{d}}|B|^{2}dx+\int_{0}^{t}\left( \int_{\mathbb{T}^{d}} \kappa |\nabla B|^{2}dx+2\int_{\mathbb{T}^{d}}\left|(-\Delta)^{\frac{\gamma}{2}}u\right|^{2}dx\right)ds
\leq \int_{\mathbb{T}^{d}}|B_0|^{2}dx+\frac{1}{\kappa}\int_{0}^{t}||f||_{H^{-1}}^{2}ds,
\end{align}
and the scaling law 
\begin{equation*}
\begin{aligned}
B_{\lambda}(x,t)&=\lambda^{\gamma}B(\lambda x,\lambda^{2} t), \quad
u_{\lambda}(x,t)=\lambda u(\lambda x,\lambda^{2} t),  \\
p_{\lambda}(x,t)&=\lambda^{2\gamma} p(\lambda x,\lambda^{2} t),  \quad
f_{\lambda}(x,t)=\lambda^{\gamma+2}f(\lambda x,\lambda^{2} t),\quad \lambda>0.
\end{aligned}
\end{equation*}
The $L^{p}$-norm of $B$ is invariant under this scaling for (1.11) with $\gamma=d/p$, i.e., $||B_{\lambda}||_{L^{p}}(0)=\lambda^{\gamma-d/p}||B||_{L^{p}}(0)$. The global well-posedness of the deterministic system for $\gamma\geq d/2$ and $B_0\in H$ (without force) is demonstrated in \cite{Robinson14} for $d=2$ and \cite{JiTan} for $d=3$; see also \cite{BKS23}. In \S 2-\S 4, we show the existence of invariant measures to (1.11) for the random force $f=\partial_t \zeta$ based on global well-posedness of the deterministic system and It\^o formulas. 

\subsubsection{Global well-posedness of the resistive MRE}

In \S 2, we first show existence and uniqueness of solutions to (1.11) for $\gamma\geq d/2$ in $\mathcal{X}_T=C([0,T]; H)\cap L^{2}(0,T; V)$ for $B_0\in H$ and the force $f\in L^{2}(0,T; V^{*})$, where $V^{*}$ is the dual space of $V=H\cap H^{1}(\mathbb{T}^{d})$. The system (1.11) for $B=b+Z$ with the force $f=(\partial_t-\kappa \Delta )Z$ and $Z\in \mathcal{X}_T$ provides an equivalent system:
\begin{equation}
\begin{aligned}
\partial_t b+u\cdot \nabla (b+Z)-(b+Z)\cdot \nabla u&=\kappa \Delta b, \\
\nabla p&=(b+Z)\cdot \nabla (b+Z)-(-\Delta)^{\gamma}u,\\
\nabla \cdot u=\nabla \cdot b&=0.
\end{aligned}
\end{equation}
We show global well-posedness of this system for $\gamma>d/2$, given $b_0\in H$ and $Z\in \mathcal{X}_T$ by the energy estimate using the continuous embedding $||u||_{L^{\infty}}\lesssim ||u||_{\dot{H}^{\gamma}}$. 

\subsubsection{Path-wise global well-posedness of the randomly forced resistive MRE}

In \S 3, we show path-wise global well-posedness of the system (1.11) for the random force $f=\partial_t \zeta$ and the Wiener process $\zeta$ in (1.5) by the reduction to (1.13) applying stochastic convolution. 

\subsubsection{The existence of invariant measures}

In \S4, we apply It\^o formulas (in Appendix B) for the stochastic process $B=B(t; B_0)$ constructed in \S 3 to obtain the energy, helicity, and mean-square potential balances for $t\geq 0$:
\begin{align*}
\mathbb{E}||B||_{H}^{2}(t)+2\mathbb{E}\int_{0}^{t}\left( \kappa ||\nabla B||_{H}^{2}+\left\|u\right\|^{2}_{\dot{H}^{\gamma}}\right)ds
&=\mathbb{E}||B_0||_{H}^{2}+\mathsf{C}_0t, \\
\mathbb{E}(\textrm{curl}^{-1}B,B)_{H}(t)
+2\kappa \mathbb{E}\int_{0}^{t}(\nabla \times B,B)_{H}ds
&=\mathbb{E}(\textrm{curl}^{-1}B_0,B_0)_{H}
+\mathcal{C}_{-\frac{1}{2}}t,\quad d=3,  \\
\mathbb{E}||\textrm{curl}^{-1}B ||_{H}^{2}(t)
+2\kappa  \mathbb{E}\int_{0}^{t}|| B||_{H}^{2}ds
&=\mathbb{E}||\textrm{curl}^{-1}B_0||_{H}^{2}
+\mathsf{C}_{-1}t,\quad d=2.
\end{align*}
We also show an exponential moment bound: 
\begin{align*}
\mathbb{E}\exp \left(\kappa \varrho ||B||_H^{2}  \right)(t)\leq 
\mathbb{E}\exp\left(\kappa \varrho ||B_0||_H^{2} \right) e^{-\kappa^{2}\varrho t}+\frac{1}{\kappa}(\mathsf{C}_0+\kappa)e^{\varrho(\mathsf{C}_0+\kappa) },\quad 0<\varrho \leq \frac{1}{2\sup_{j\geq 1}b_j^{2} }.
\end{align*}
The existence of an invariant measure follows from the energy balance and Krylov--Bogoliubov theorem. 

\subsubsection{The non-resistive limits}

\S5 is the main observation of this study. The invariant measure $\mu_{\kappa}$ of the stochastic system (1.3) and statistically stationary solutions $B_{\kappa}: (\Omega,\mathcal{F},\mathbb{P})\to (H,\mathcal{B}(H))$ with law $\mathcal{D}(B_{\kappa})=\mu_{\kappa}$ satisfy the following balance relations for the associated velocity field $u_{\kappa}$: 
\begin{align}
\mathbb{E}\left(\kappa ||\nabla B_{\kappa}||_{H}^{2}+||u_{\kappa}||_{\dot{H}^{\gamma }}^{2} \right)
&=\kappa \frac{\mathsf{C}_0}{2}, \\
\mathbb{E}\exp\left(   \rho ||B_{\kappa}||_{H}^{2} \right) 
&\leq (\mathsf{C}_0+1)e^{\rho(\mathsf{C}_0+1) },\quad 0<\rho\leq \frac{1}{2\sup_{j\geq 1}b_j^{2}},\\
\mathbb{E}\left(\nabla \times B_{\kappa}, B_{\kappa}   \right)_H
&=\frac{\mathcal{C}_{-\frac{1}{2}}}{2},\quad d=3 \\
\mathbb{E}\left\|B_{\kappa}   \right\|_H^{2}
&=\frac{\mathsf{C}_{-1}}{2},\quad d=2.
\end{align}
The balance relation (1.14) is a statistical version of the energy equality of MRE (1.1), which states that the velocity field $u_{\kappa}$ vanishes as $\kappa\to 0$ with the bounded magnetic field $B_{\kappa}$. The measure $\mu_{\kappa}$ weakly converges to a measure $\mu_{0}$ in $\mathcal{P}(V^{1-\varepsilon})$ for $\varepsilon>0$ and $V^{m}=H\cap H^{m}(\mathbb{T}^{d})$ by the compactness $V=V^{1}\Subset V^{1-\varepsilon}$ and the bound $\mathbb{E}||B_{\kappa}||_{V}^{2}\lesssim \mathsf{C}_0$ by Prokhorov theorem. Here, we use the symbol $X\Subset Y$ if any bounded sequence in $X$ has a convergent subsequence in $Y$. 

We demonstrate strong convergence of the magnetic field $B_{\kappa}$ toward a random MHS equilibrium (1.2) as $\kappa\to 0$ by showing convergence of the \textit{lifted} measure $\boldsymbol{\mu}_{\kappa}=\mathcal{D}(B_{\kappa})$ for $B_{\kappa}: (\Omega,\mathcal{F},\mathbb{P}) \to (\mathcal{X}_{T}, \mathcal{B}(\mathcal{X}_{T}))$. Namely, we show the weak convergence of probability measures 
\begin{align}
\boldsymbol{\mu}_{\kappa}\to \boldsymbol{\mu}_{0}\quad \textrm{in}\ \mathcal{P}(\mathcal{X}_{T}^{-\varepsilon}),\quad 
\mathcal{X}^{-\varepsilon}_T=C([0,T]; H^{-\varepsilon})\cap L^{2}(0,T; H^{1-\varepsilon}),
\end{align}
and deduce from Skorokhod theorem the convergence of statistically stationary solutions for $B$ with law $\mathcal{D}(B)=\mu_{0}$,
\begin{align}
B_{\kappa}\to B\quad \textrm{in}\ \mathcal{X}^{-\varepsilon}_{T},\quad \textrm{a.s.}
\end{align}
The convergence (1.19) is strong enough to take a limit in the constitutive law (1.3$)_2$ and conclude that the limit $B=B(x)$ is a random MHS equilibrium (1.2).

We show the tightness of the lifted measures (1.18) by choosing a compact set $\mathcal{K}_T$ such that 

\begin{align}
\mathcal{K}_{T}\Subset \mathcal{X}^{-\varepsilon}_T,\quad 
\mathbb{E}||B_{\kappa}||_{\mathcal{K}_{T}}^{p}&\leq C,\quad 0<\kappa\leq 1,\quad 1\leq p<2.
\end{align}
The space $\mathcal{K}_{T}$ is an alternative space to $\mathcal{X}_{T}$; note that the embedding $\mathcal{X}_{T}\subset \mathcal{X}_{T}^{-\varepsilon}$ is not compact. We set $\mathcal{K}_T=\mathcal{W}_T\cap L^{2}(0,T; V)$ for a sum space $\mathcal{W}_T=\mathcal{W}_T^{1}+\mathcal{W}_T^{2}+\mathcal{W}_T^{3}$ and choose each $\mathcal{W}_T^{i}$ depending on regularity of decomposed statistically stationary solutions (see \S 5.3). Due to the cubic nonlinearity, we show the $p$-th moment estimate (1.20$)_2$ for $p<2$.

We show the compactness (1.20$)_1$ by Lions--Aubin--Simon theorem and the boundedness (1.20$)_2$ by estimating statistically stationary solutions on each $\mathcal{W}_T^{i}$. We then deduce the tightness (1.18) from Prokhorov theorem. 

\subsubsection{Absolute continuity of laws}
\S 6 provides the proof of Theorem 1.3 by showing absolute continuity of laws of $\mathscr{E}$, $\mathscr{H}$, $\mathscr{M}$, and $\mathscr{C}$ under $\mu_0$ for Lebesgue measures. We apply It\^o formulas for the general functionals $f(\mathscr{E}(B))$, $f(\mathscr{H}(B))$, $f(\mathscr{M}(B))$ and $f\in C^{2}(\mathbb{R})$, an identity for one-dimensional stationary stochastic processes (B.28), and Krylov's estimate for $n$-dimensional stationary stochastic processes (B.29). The absolute continuity of $\mathcal{D}_{\mu_0}(\mathscr{C})$ for the Lebesgue measure on $\mathbb{R}^{n}$ implies the infinite dimensionality of $\mu_0$ in Theorem 1.3.

%This work was initiated during the OCAMI workshop The Grad Conjecture in Fluid Mechanics and Magnetohydrodynamics: Theory and Applications, held in Osaka in March 2024. KA is grateful to all workshop participants and the support of the MEXT Joint Usage/Research Center on Mathematics and Theoretical Physics. 

\subsection{Acknowledgements}
KA is supported by the JSPS through the Grant in Aid for Scientific Research (C) 24K06800 and MEXT Promotion of Distinctive Joint Research Center Program JPMXP0619217849. IJ is supported by the National Research Foundation of Korea grant No. 2022R1C1C1011051 and RS-2024-00406821. NS is supported by JSPS KAKENHI Grant No. 25K07267, No. 22H04936, and No. 24K00615.

\section{Global well-posedness of the resistive MRE}

We show global well-posedness of the system (1.11) for $\gamma\geq d/2$ and its perturbed system (1.13) for $\gamma> d/2$. Although the global well-posedness of the system (1.11) without force is known \cite{Robinson14}, \cite{JiTan}, \cite{BKS23}, we need to obtain the global well-posedness result of the perturbed system (1.13) to show the path-wise global well-posedness of the system (1.11) for the random force in \S 3. The local well-posedness of (1.11) and (1.13) is parallel. We express (1.11) for $f\in L^{2}(0,T; V^{*})$ as 
\begin{equation}
\begin{aligned}
\partial_t B&=\kappa \Delta B +\nabla \cdot (B\otimes u-u\otimes B)+f,\\
u&=K_{\gamma}(B,B),
\end{aligned}
\end{equation}
by the bilinear operator $K_{\gamma}(\cdot,\cdot)=(-\Delta)^{-\gamma}\Pi \nabla \cdot (\cdot\otimes \cdot)$ and the projection $\Pi: L^{2}(\mathbb{T}^{d})\to L^{2}_{\sigma}(\mathbb{T}^{d})=\{g\in L^{2}(\mathbb{T}^{2})\ |\ \nabla \cdot g=0\}$ and  construct mild solutions for $B_0\in H$ in $(0,T)$: 
\begin{equation}
B=e^{t\kappa \Delta}B_0+\int_{0}^{t}e^{(t-s)\kappa \Delta}f(s)ds +\int_{0}^{t}e^{(t-s)\kappa \Delta}\nabla \cdot (B\otimes  u- u\otimes B)ds,\quad u=K_{\gamma}(B, B).
\end{equation}

\subsection{Regularity estimates for linear operators}

We use the Fourier series 
\begin{align*}
g=\sum_{k\in \mathbb{Z}^{d}}\hat{g}_ke^{ik\cdot x},\quad \hat{g}_k=\frac{1}{(2\pi)^{d}}\int_{\mathbb{T}^{d}}g(x)e^{-ik\cdot x}dx,
\end{align*}
and define the Sobolev space $H^{m}(\mathbb{T}^{d})$ for $m\in \mathbb{R}$ by  
\begin{align*}
H^{m}(\mathbb{T}^{d})=\left\{ g\in \mathcal{D}'(\mathbb{T}^{d})\ \middle|\ ||g||_{H^{m}(\mathbb{T}^{d})}<\infty \right\},\quad ||g||_{H^{m}(\mathbb{T}^{d})}^{2} =\sum_{k\in \mathbb{Z}^{d}}(|k|^{2}+1)^{m}|\hat{g}_k|^{2},
\end{align*}
where $\mathcal{D}'(\mathbb{T}^{d})$ denotes the space of distributions. The space $C^{\infty}(\mathbb{T}^{d})$ is dense in $H^{m}(\mathbb{T}^{d})$ for $m\geq 0$. We denote by $C^{m}_{b}(\mathbb{T}^{d})$ the space of all bounded and continuous functions on $\mathbb{T}^{d}$ up to order $m\geq 0$. We use the following embeddings \cite{Kuk12}: 
\begin{align}
H^{m}(\mathbb{T}^{d})&\subset L^{q}(\mathbb{T}^{d}),\quad 0\leq m\leq \frac{d}{2},\ 2\leq q\leq \frac{2d}{d-2m}<\infty,\\  
H^{m}(\mathbb{T}^{d})&\subset C^{m-\frac{d}{2}}_{b}(\mathbb{T}^{d}),\quad m>\frac{d}{2}, \\
H^{m_1}(\mathbb{T}^{d})&\Subset H^{m_2}(\mathbb{T}^{d}),\quad m_1>m_2, \\
||g||_{H^{\theta m_1+(1-\theta)m_2}}&\leq ||g||_{H^{m_1}}^{\theta}||g||_{H^{m_2}}^{1-\theta},\quad g\in H^{m_1}\cap H^{m_2}(\mathbb{T}^{d}),\ 0\leq \theta\leq 1.
\end{align}
We define the fractional operators and the heat semigroup by 
\begin{align*}
(-\Delta)^{m}g&=\sum_{k\in \mathbb{Z}^{d}}|k|^{2m}\hat{g}_ke^{ik\cdot x},\quad m\geq 0, \\
e^{t\Delta}g&=\sum_{k\in \mathbb{Z}^{d}}e^{-t |k|^{2}}\hat{g}_{k}e^{ik\cdot x}.
\end{align*}
For $m<0$, we define the operator $(-\Delta)^{m}$ on average-zero functions by taking the summuation for $k\in \mathbb{Z}_{0}^{d}=\mathbb{Z}^{d}\backslash \{0\}$. By the Sobolev inequality on the torus \cite{BO13}, the operator 
\begin{align}
(-\Delta)^{-\frac{\alpha}{2}}: L^{r}(\mathbb{T}^{d})\to L^{q}(\mathbb{T}^{d}),\quad 0<\alpha<d,\ 1<r<q<\infty,\ \frac{1}{q}=\frac{1}{r}-\frac{\alpha}{d},
\end{align}
is a bounded operator. The heat semigroup is a bounded operator on $H^{s}(\mathbb{T}^{d})$ for $s\in \mathbb{R}$ by Parseval's identity. We set the $L^{2}$-solenoidal space $L^{2}_{\sigma}(\mathbb{T}^{d})=\{g\in L^{2}(\mathbb{T}^{d})\ |\ \nabla \cdot g=0\ \}$ and decompose $g\in L^{2}(\mathbb{T}^{d})$ as 
\begin{align*}
g=\left(\hat{g}_0+\sum_{k\in \mathbb{Z}^{d}_{0}}\left(I-\frac{k\otimes k}{|k|^{2}} \right) \hat{g}_ke^{ik\cdot x}\right)+\nabla \sum_{k\in \mathbb{Z}^{d}_{0}}\frac{-ik }{|k|^{2}}\cdot   \hat{g}_ke^{ik\cdot x}
=:\Pi g+(1-\Pi)g \in L^{2}_{\sigma}(\mathbb{T}^{d})\oplus \nabla H^{1}(\mathbb{T}^{d}).
\end{align*}
We denote the space of all average zero $L^{p}$-functions for $1\leq p\leq \infty$ by $L^{p}_{\textrm{av}}(\mathbb{T}^{d})$. We define $H$ by (1.4) and $V^{m}=H\cap H^{m}(\mathbb{T}^{d})$ for $m>0$. We define $V^{*}$ as the dual space of $V=V^{1}$. The decomposition above also holds on average-zero function spaces, i.e.,  
\begin{align}
L^{2}_{\textrm{av}}(\mathbb{T}^{d})=H\oplus \nabla H^{1}(\mathbb{T}^{d}).
\end{align}
For average-zero functions $g\in L^{2}_{\textrm{av}}(\mathbb{T}^{d})\cap H^{m}(\mathbb{T}^{d})$ and $m\in \mathbb{R}$, we define the homogeneous norm $|| g||_{\dot{H}^{m}}=||(-\Delta)^{\frac{m}{2}}g||_{L^{2}}$. Since the summation in the $H^{m}$-norm of $g\in L^{2}_{\textrm{av}}(\mathbb{T}^{d})\cap H^{m}(\mathbb{T}^{d})$ is taken for $|k|\geq 1$,  
\begin{align*}
||g||_{H^{m}}=\left(\sum_{k\in \mathbb{Z}_0^{d}}\left(|k|^{2}+1\right)^{m}|\hat{g}_{k}|^{2}\right)^{\frac{1}{2}}
\leq 2^{\frac{m}{2}}\left(\sum_{k\in \mathbb{Z}_0^{d}}|k|^{2m}|\hat{g}_{k}|^{2}\right)^{\frac{1}{2}}
=2^{\frac{m}{2}} ||(-\Delta)^{\frac{m}{2}}g||_{L^{2}}
=2^{\frac{m}{2}} || g||_{\dot{H}^{m}},\quad m\geq 0,
\end{align*}
and $||g||_{H^{m}}\geq 2^{m/2}||g||_{\dot{H}^{m}}$ for $m<0$. The following Proposition 2.1 enables one to define MHS equilibria (1.2) for $B\in H^{1}(\mathbb{T}^{d})$ without pressure. 

\begin{prop}
Assume that $f\in L^{1}_{\textrm{av}}(\mathbb{T}^{d})$ satisfies $(f,\varphi)=0$ for all $\varphi\in H\cap C^{\infty}(\mathbb{T}^{d})$. Then, $f=\nabla p$ for some function $p\in W^{1,1} (\mathbb{T}^{d})$. 
\end{prop}

\begin{proof}
If $f\in L^{2}_{\textrm{av}}(\mathbb{T}^{d})$, $(f,\varphi)=0$ for all $\varphi\in H$ by the density of $H\cap C^{\infty}(\mathbb{T}^{d})$ in $H$. Thus, $f=\nabla p\in \nabla H^{1}(\mathbb{T}^{d})$.

For $f\in L^{1}_{\textrm{av}}(\mathbb{T}^{d})$ and $m>d/4$, we set
\begin{align*}
f_m=(-\Delta)^{-m}f=\sum_{k\in \mathbb{Z}^{d}_{0}}|k|^{-2m}\hat{f}_ke^{ik\cdot x}.
\end{align*}
By $|\hat{f}_k|\leq (2\pi)^{-d}||f||_{L^{1}}$ and Parseval's identity, 
\begin{align*}
||f_m||_{L^{2}}^{2}=\sum_{k\in \mathbb{Z}^{d}_{0}}|k|^{-4m}|\hat{f}_{k}|^{2}\leq (2\pi)^{-d}\left(\sum_{k\in \mathbb{Z}^{d}_{0}}|k|^{-4m}\right)||f||_{L^{1}}<\infty.
\end{align*}
For $\varphi\in H\cap C^{\infty}(\mathbb{T}^{d})$,
\begin{align*}
0=\left(f,(-\Delta)^{-m}\varphi  \right)=\left((-\Delta)^{-m} f,\varphi\right)=(f_m,\varphi).
\end{align*}
Thus, $f_{m}=\nabla p_{m}\in \nabla H^{1}(\mathbb{T}^{d})$. By applying $(-\Delta)^{-m}$, $f=\nabla p$ for $p=(-\Delta)^{-m} p_{m}\in  W^{1,1}(\mathbb{T}^{d})$. 
\end{proof}

For $T>0$, we set the Banach spaces 
\begin{equation}
\begin{aligned}
{\mathcal{H}}_T&=\{B\in L^{2}(0,T; V)\ |\ \partial_t B\in L^{2}(0,T; V^{*})\ \},\quad 
||B||_{\mathcal{H}_T}=\left(\int_{0}^{T}\left(||B||_{V}^{2}+ ||\partial_t B||_{V^{*}}^{2}\right)dt\right)^{\frac{1}{2}},\\
\mathcal{X}_T&=C([0,T]; H)\cap L^{2}(0,T; V),\quad
||B||_{\mathcal{X}_T}=\max\left\{||B||_{C([0,T]; H)},\ ||B||_{L^{2}(0,T; V)}     \right\}.
\end{aligned}
\end{equation}
The function $||B||_{L^{2}}^{2}(t)$ for $B\in \mathcal{H}_T$ is almost everywhere differentiable and continuous \cite[5.9, Theorem 3]{E} and the continuous embedding holds:  
\begin{align}
\mathcal{H}_T\subset \mathcal{X}_T.
\end{align}
By the interpolation inequality (2.6), 
\begin{align}
\mathcal{X}_T\subset L^{p}(0,T; H^{s}(\mathbb{T}^{d})),\quad 2\leq p\leq \infty,\quad sp=2.
\end{align}
We say that $Z\in \mathcal{H}_T$ is a solution to the inhomogeneous Stokes equations for $f\in L^{2}(0,T; V^{*})$ if $Z$ satisfies  
\begin{equation}
\begin{aligned}
\partial_t Z-\kappa\Delta Z&=f,\\
\nabla \cdot Z&=0,
\end{aligned}
\end{equation}
in the distributional sense.

\begin{prop}
For $f\in L^{2}(0,T; V^{*})$, the function $Z(t)=\int_{0}^{t}e^{(t-s)\kappa\Delta }f(s)ds$ satisfies 
\begin{align}
\left\|Z\right\|_{\mathcal{H}_T}
\lesssim_{\kappa} ||f||_{L^{2}(0,T; V^{*})},
\end{align}
and is a unique solution to (2.12) satisfying the initial condition $Z(x,0)=0$. 
\end{prop}

\begin{proof}
The function $Z(t)$ belongs to $C([0,T]; V^{*})$ by the boundedness of the heat semigroup on $V^{*}$. The function 
\begin{align*}
f_N=\sum_{|k|\leq N}\hat{f}_ke^{ik\cdot x},
\end{align*}
is smooth for $x$ and converges to $f$ in $L^{2}(0,T; V^{*})$ as $N\to\infty$. The function $Z_N(t)=\int_{0}^{t}e^{(t-s)\kappa\Delta}f_N(s)ds$ is a solution to (2.12) for $f_N$ and satisfies the energy inequality,
\begin{align*}
\sup_{0\leq t\leq T}\int_{\mathbb{T}^{d}}|Z_N|^{2}dx+\kappa\int_{0}^{T}\int_{\mathbb{T}^{d}}|\nabla Z_N|^{2}dxds\leq \frac{1}{\kappa}\int_{0}^{T}||f_N||_{H^{-1}}^{2}ds.
\end{align*}
By the continuous embedding (2.10),
\begin{align*}
||Z_N||_{\mathcal{H}_T}\lesssim_{\kappa}  ||f_N||_{L^{2}(0,T; H^{-1})},
\end{align*}
and $Z_n$ converges to a limit $Z\in \mathcal{H}_T$ and the limit $Z$ is a solution to (2.12) for $f$. 

If $Z\in \mathcal{H}_T$ is a solution to (2.12) for $f=0$, for almost every $t\in (0,T)$,  
\begin{align*}
0=2<\partial_t Z,Z>-2 \kappa<\Delta Z,Z>=\frac{d}{dt}\int_{\mathbb{T}^{d}}|Z|^{2}dx+2\kappa\int_{\mathbb{T}^{2}}|\nabla Z|^{2}dx.
\end{align*}
By integrating this identity on $(0,T)$ and using $Z\to 0$ on $H$ as $t\to0$, we find that $Z=0$. Thus, the solution of (2.12) is unique.
\end{proof}

\subsection{The cubic estimate}

We show the boundedness of the bilinear operator $K_{\gamma}$ for $\gamma=d/2$.

\begin{prop}
The bilinear operator 
\begin{align}
K_{d/2}: H^{s}(\mathbb{T}^{d})\times H^{s}(\mathbb{T}^{d})\to H^{m}(\mathbb{T}^{d}),\quad m=\frac{d}{2}-1+2s,\ 0< s<\frac{d}{4},
\end{align}
is a bounded operator. 
\end{prop}

\begin{proof}
By the Sobolev embedding (2.3), $H^{s}(\mathbb{T}^{d})\subset L^{2r}(\mathbb{T}^{d})$ for $1<r=d/(d-2s)<2$. For $B_1, B_2\in H^{s}(\mathbb{T}^{d})$, 
\begin{align*}
\left\|K_{d/2}(B_1,B_2)\right\|_{H^{m}}=\left\|(-\Delta)^{\frac{m}{2}}K_{d/2}(B_1,B_2)\right\|_{L^{2}}
=\left\|(-\Delta)^{\frac{m-d}{2}}\Pi \nabla \cdot (B_1\otimes  B_2)\right\|_{L^{2}}
\lesssim \left\|(-\Delta)^{\frac{m-d+1}{2}} (B_1\otimes  B_2)\right\|_{L^{2}}.
\end{align*}
For $\alpha=-m+d-1$, the operator $(-\Delta)^{-\alpha/2}: L^{r}(\mathbb{T}^{d})\to L^{2}(\mathbb{T}^{d})$ is bounded by (2.7). By $||B_1\otimes B_2||_{L^{r}}\leq ||B_1||_{L^{2r}}||B_2||_{L^{2r}}$, (2.14) follows.
\end{proof}

\begin{prop}
The bilinear operator 
\begin{align}
K_{d/2}: L^{p}(0,T; H^{s}(\mathbb{T}^{d}))\times L^{p}(0,T; H^{s}(\mathbb{T}^{d}))\to L^{\frac{p}{2}}(0,T; H^{m}(\mathbb{T}^{d})),
\end{align}
is a bounded operator for $2\leq p\leq \infty$ and $(s,m)$ in (2.14). 
\end{prop}

\begin{proof}
The boundedness follows from (2.14) and H\"older's inequality.
\end{proof}

\begin{prop}
For $B_1, B_2\in L^{p}(0,T; H^{s}(\mathbb{T}^{d}) )$, $8/d<p<\infty$, and $sp=2$, 
\begin{align}
u=K_{d/2}(B_1,B_2)\in L^{\frac{p}{2}}(0,T; H^{\frac{d}{2}-1+\frac{4}{p} } (\mathbb{T}^{d})).
\end{align}
In particular, $u\in L^{2}(0,T; H^{\frac{d}{2}} (\mathbb{T}^{d}))$ for $d\geq 3$ and $u\in L^{2+\varepsilon}(0,T; H^{\frac{2 \varepsilon}{2+\varepsilon} } (\mathbb{T}^{2}))$ for $d=2$ and $\varepsilon>0$. 
\end{prop}

\begin{proof}
For $8/d<p<\infty$, $s=2/p$ satisfies the condition (2.14). Thus, property (2.16) follows from property (2.15).
\end{proof}

\begin{lem}[Cubic estimate]
Let $4<p<\infty$ and $0<s<1/2$ satisfy $sp=2$. Then, 
\begin{align}
||K_{d/2}(B_1, B_2) \otimes B_3||_{L^{2}(0,T; L^{2})}\lesssim ||B_1||_{L^{p}(0,T; H^{s})}||B_2||_{L^{p}(0,T; H^{s})} ||B_3||_{L^{\frac{2}{1-2s}}(0,T; H^{1-2s})},
\end{align}
holds for $B_1, B_2\in {L^{p}(0,T; H^{s})}$ and $B_3\in L^{\frac{2}{1-2s}}(0,T; H^{1-2s})$. In particular, $K_{d/2}(B,B)\otimes B\in L^{2}(0,T; L^{2})$ for $B\in L^{p}(0,T; H^{s})\cap L^{\frac{2}{1-2s}}(0,T; H^{1-2s})$.
\end{lem}

\begin{proof}
We set $u=K_{d/2}(B_1, B_2)$. By H\"older's inequality for $1<q,r<\infty$ and their conjugate exponents $q'$ and $r'$,
\begin{align*}
||u\otimes B||_{L^{2}(0,T; H) }\leq ||u||_{L^{2r}(0,T; L^{2q}) }||B||_{L^{2r'}(0,T; L^{2q'}) }.
\end{align*}
By Sobolev inequality (2.3), $H^{l}(\mathbb{T}^{d})\subset L^{2q}(\mathbb{T}^{d})$ for $0<l<d/2$ satisfying $q=d/(d-2l)$ and $H^{m}(\mathbb{T}^{d})\subset L^{2q'}(\mathbb{T}^{d})$ for $0<m<d/2$ satisfying $q'=d/(d-2m)$. Thus, $l+m=d/2$ and 
\begin{align*}
||u\otimes B||_{L^{2}(0,T; H) }\lesssim ||u||_{L^{2r}(0,T; H^{l}) }||B||_{L^{2r'}(0,T; H^{m}) }.
\end{align*}
We take $r$ and $l$ by $r=p/4$ and $l=d/2-1+4/p$. By $sp=2$, $l=d/2-1+2s$. By $8/d\leq 4<p<\infty$ and $0<l<d/2$, we apply (2.15) to estimate 
\begin{align*}
||u||_{L^{2r}(0,T; H^{l}) }\lesssim ||B_1||_{L^{p}(0,T; H^{s}) }||B_2||_{L^{p}(0,T; H^{s}) }.
\end{align*}
By 
\begin{align*}
\frac{1}{r'}&=1-\frac{1}{r}=1-\frac{4}{p}=1-2s,\\
m&=\frac{d}{2}-l=1-2s,
\end{align*}
we obtain (2.17).
\end{proof}

\begin{cor}
\begin{align}
||K_{d/2}(B,B)\otimes B||_{H}\lesssim ||B||_{H}^{2}||B||_{H^{1}},\quad B\in H^{1}(\mathbb{T}^{d}).
\end{align}
\end{cor}

\begin{proof}
By applying (2.17) for the time-independent function $B_i=B\in H^{1}(\mathbb{T}^{d})$ and $i=1,2,3$,
\begin{align*}
||K_{d/2}(B,B)\otimes B||_{H}\lesssim ||B||_{H^{s}}^{2}||B||_{H^{1-2s}}.
\end{align*}
By the interpolation inequality (2.6), (2.18) follows.  
\end{proof}

\subsection{Local well-posedness}
For $B\in \mathcal{H}_T$ and $u=K_{d/2}(B,B)$, $B\otimes u$ and $u\otimes B$ are locally square-integrable in $\mathbb{T}^{d}\times (0,T)$ by Lemma 2.6. We say that $B\in \mathcal{H}_T$ is a solution to (2.1) for $\gamma=d/2$ and $f\in L^{2}(0,T; V^{*})$ if $B$ satisfies $(2.1)_1$ for $u=K_{d/2}(B,B)$ on $\mathbb{T}^{d}\times (0,T)$ in the distributional sense.

We take $4<p<\infty$ and set $0<s=2/p<1/2$ and 
\begin{align}
\mathcal{Y}_{T}&=L^{p}(0,T; H^{s})\cap L^{\frac{2}{1-2s}}(0,T; H^{1-2s}),\\
||B||_{\mathcal{Y}_{T}}&=\max\left\{||B||_{L^{p}(0,T; H^{s})}, ||B||_{L^{\frac{2}{1-2s}}(0,T; H^{1-2s})}\right\}.
\end{align}
By the continuous embeddings (2.10) and (2.11), 
\begin{align}
\mathcal{H}_T\subset \mathcal{X}_T\subset \mathcal{Y}_{T}.
\end{align}
The $\mathcal{Y}_{T}$-norm vanishes as $T\to0$. By the cubic estimate (2.17), 
\begin{align}
||K_{d/2}(B_1,B_2)\otimes B_3||_{L^{2}(0,T; L^{2}) }\lesssim ||B_1||_{\mathcal{Y}_T }||B_2||_{\mathcal{Y}_T }||B_3||_{\mathcal{Y}_T },\quad B_i\in \mathcal{Y}_T,\ i=1,2,3.
\end{align}
We construct unique local-in-time mild solutions (2.2) by using the regularity estimate for the Stokes equations (2.13) and the cubic estimate (2.22).

\begin{prop}
Solutions $B\in \mathcal{H}_T$ of (2.1) satisfying the initial condition $B(x,0)=B_0(x)$ are mild solutions (2.2). Conversely, mild solutions $B\in \mathcal{X}_T$ belong to $\mathcal{H}_T$ and are solutions to (2.2) satisfying the initial condition $B(x,0)=B_0(x)$.
\end{prop}

\begin{proof}
By the uniqueness of the Stokes equations (2.12), solutions $B\in \mathcal{H}_T$ of (2.1) are mild solutions (2.2). Conservely, for mild solutions $B\in \mathcal{X}_T$, the nonlinear terms in (2.1$)_1$ belong to $L^{2}(0,T; V^{*})$ by the embedding (2.21) and the cubic estimate (2.22). Thus, the second and third terms on the right-hand side of (2.2$)_1$ belong to $\mathcal{H}_T$ by the regularity estimate (2.13). Since $e^{t\kappa \Delta}B_0\in \mathcal{H}_T$, $B\in \mathcal{H}_T$ and $B$ satisfies (2.1$)_1$ in the distributional sense.
\end{proof}

\begin{prop}
Solutions $B\in \mathcal{H}_T$ of (2.1) are unique.  
\end{prop}

\begin{proof}
Suppose that $B^1$ and $B^2$ are two mild solutions of (2.1) in $\mathbb{T}^{d}\times (0,T')$ with the associated velocity fields $u^1$ and $u^2$, respectively. We set $B=B^1-B^2$ and $u=u^1-u^2$. By using the bilinear operator,
\begin{align*}
u=K_{d/2}(B^1,B^1)-K_{d/2}(B^2,B^2)=K_{d/2}(B,B^{1})+K_{d/2}(B^2,B).
\end{align*}
The functions $B$ and $u$ satisfy 
\begin{align*}
\partial_t B
=\kappa \Delta B+\nabla \cdot (B\otimes u^1 +B^{2}\otimes u-u^{1}\otimes B-u\otimes B^{2}) 
=:\kappa \Delta B+\nabla \cdot F,
\end{align*}
in the sense that 
\begin{align*}
B=\int_{0}^{t}e^{(t-s)\kappa\Delta}\nabla \cdot F(s)ds.
\end{align*}
By the cubic estimates (2.22) for $\mathcal{Y}=\mathcal{Y}_{T'}$,
\begin{align*}
||F||_{L^{2}(0,T'; L^{2})}\lesssim ||B||_{\mathcal{Y}}||B^1||_{\mathcal{Y}}^{2}+||B||_{\mathcal{Y}} ||B^2||_{\mathcal{Y}}\left( ||B^{1}||_{\mathcal{Y}} +||B^2||_{\mathcal{Y}}\right).
\end{align*} 
Thus, $\nabla \cdot F\in L^{2}(0,T'; V^{*})$. By the continuous embeddings (2.21), and the regularity estimates (2.13), 
\begin{align*}
||B||_{\mathcal{Y}}
\lesssim ||B||_{\mathcal{X}}
=\left\|\int_{0}^{t}e^{(t-s)\kappa\Delta}\nabla \cdot F(s)ds\right\|_{\mathcal{X}} 
\lesssim \left\|F\right\|_{L^{2}(0,T'; L^{2})} 
\lesssim ||B||_{\mathcal{Y}}||B^1||_{\mathcal{Y}}^{2}+||B||_{\mathcal{Y}} ||B^2||_{\mathcal{Y}}\left( ||B^{1}||_{\mathcal{Y}} +||B^2||_{\mathcal{Y}}\right).
\end{align*}
If $B\neq 0$, $1\lesssim ||B^1||_{\mathcal{Y}_{T'}}^{2} +||B^1||_{\mathcal{Y}_{T'}} ||B^2||_{\mathcal{Y}_{T'}}+||B^2||_{\mathcal{Y}_{T'}}^{2}\to 0$  as $T'\to 0$. This is a contradiction. Thus, $B=0$.
\end{proof}

\begin{lem}
For $B_0\in H$ and $f\in L^{2}(0,T; V^{*})$, there exist $T_0>0$ and a unique solution $B\in \mathcal{H}_{T_0}$ of (2.1) for $\gamma=d/2$.
\end{lem}

We set a sequence $\{B_j\}$ by 
\begin{equation}
\begin{aligned}
B_{j+1}&=B_1 +\int_{0}^{t}e^{(t-s)\kappa\Delta}\nabla \cdot (B_j\otimes  u_j- u_j\otimes B_j)ds,\quad u_j=K_{d/2}(B_j, B_j),\\
B_1&=e^{t\kappa \Delta}B_0+\int_{0}^{t}e^{(t-s)\kappa \Delta}f(s)ds.
\end{aligned}
\end{equation}

\begin{prop}
Let $\mathcal{Y}=\mathcal{Y}_{T_0}$ and $T_0>0$. Then,
\begin{align}
||B_{j+1}||_{\mathcal{Y}}&\leq ||B_1||_{\mathcal{Y}}+C||B_j||_{\mathcal{Y}}^{3},\\
||B_{j+1}||_{\mathcal{X}}&\leq ||B_1||_{\mathcal{X}}+C'||B_j||_{\mathcal{Y}}^{3},\quad j\in \mathbb{N},
\end{align}
with some constants $C$ and $C'$, independent of $j$ and $T_0>0$.
\end{prop}

\begin{proof}
By the continuous embeddings (2.21), the regularity estimate (2.13), and the cubic estimate (2.22), 
\begin{align*}
||B_{j+1}-B_1||_{\mathcal{Y}}
\lesssim ||B_{j+1}-B_1||_{\mathcal{X}}
&= \left\|\int_{0}^{t}e^{(t-s)\kappa\Delta}\nabla \cdot (B_j\otimes  u_j- u_j\otimes B_j)ds\right\|_{\mathcal{X}} \\
&\lesssim ||u_j \otimes B_j||_{L^{2}(0,T_0;L^{2})}+|| B_j \otimes u_j||_{L^{2}(0,T_0;L^{2})}  
\lesssim ||B_j||_{\mathcal{Y}}^{3}.
\end{align*}
Thus, (2.24) and (2.25) hold.
\end{proof}

\begin{prop}
There exists an absolute constant $\delta_1>0$ such that for $||B_1||_{\mathcal{Y}}\leq \delta_1$, 
\begin{align}
||B_{j}||_{\mathcal{Y}}&\leq 2||B_1||_{\mathcal{Y}},\\
||B_{j}||_{\mathcal{X}}&\leq ||B_1||_{\mathcal{X}}+8C'||B_1||_{\mathcal{Y}}^{3},\quad j\in \mathbb{N}.
\end{align}
\end{prop}

\begin{proof}
By (2.24), (2.26) holds for $B_1$ satisfying $||B_1||_{\mathcal{Y}}\leq 1/(2\sqrt{2C})=:\delta_1$. The estimate (2.27) follows from (2.25) and (2.26). 
\end{proof}

We estimate $H_{j+1}=B_{j+1}-B_j$ and $v_{j+1}=u_{j+1}-u_j$. By (2.23), 
\begin{equation}
\begin{aligned}
H_{j+1}&=\int_{0}^{t}e^{(t-s)\kappa\Delta}\nabla \cdot (H_j\otimes u_j -u_j\otimes H_j+B_{j-1}\otimes v_j-v_j\otimes B_{j-1} )ds,\\
v_{j}&=K_{d/2}(H_j,B_j)+K_{d/2}(B_{j-1}, H_j).
\end{aligned}
\end{equation}

\begin{prop}
There exists an absolute constant $\delta_2>0$ such that for $||B_1||_{\mathcal{Y}}\leq \delta_2$,
\begin{align}
||H_{j+1}||_{\mathcal{X}}\leq \frac{1}{2}||H_{j}||_{\mathcal{X}},\quad j\in \mathbb{N}.
\end{align}
\end{prop}

\begin{proof}
By the cubic estimate (2.22),
\begin{align*}
||H_j\otimes u_j -u_j\otimes H_j+B_{j-1}\otimes v_j-v_j\otimes B_{j-1} ||_{L^{2}(0,T_0; L^{2})}\lesssim ||H_j||_{\mathcal{Y}}\left( ||B_j||_{\mathcal{Y}}^{2}+ ||B_{j-1}||_{\mathcal{Y}}||B_{j}||_{\mathcal{Y}}+||B_{j-1}||_{\mathcal{Y}}^{2}  \right).
\end{align*}
We use the uniform bound (2.26) for $B_1$ satisfying $||B_1||_{\mathcal{Y}}\leq \delta_1$. By the regularity estimates (2.13) and the continuous embedding (2.21), 
\begin{align*}
||H_{j+1}||_{\mathcal{X}}
&=\left\|\int_{0}^{t}e^{(t-s)\kappa\Delta}\nabla \cdot (H_j\otimes u_j -u_j\otimes H_j+B_{j-1}\otimes v_j-v_j\otimes B_{j-1} )ds\right\|_{\mathcal{X}} \\
&\lesssim ||H_j\otimes u_j -u_j\otimes H_j+B_{j-1}\otimes v_j-v_j\otimes B_{j-1} ||_{L^{2}(0,T_0; L^{2})} \\
&\lesssim ||H_j||_{\mathcal{Y}}\left( ||B_j||_{\mathcal{Y}}^{2}+ ||B_{j-1}||_{\mathcal{Y}}||B_{j}||_{\mathcal{Y}}+||B_{j-1}||_{\mathcal{Y}}^{2}  \right)
\lesssim ||H_j||_{\mathcal{Y}}||B_1||_{\mathcal{Y}}^{2}
\lesssim ||H_j||_{\mathcal{X}}||B_1||_{\mathcal{Y}}^{2}.
\end{align*}
Thus, $||H_{j+1}||_{\mathcal{X}}\leq C||B_1||_{\mathcal{Y}}^{2} ||H_j||_{\mathcal{X}}$ with a constant $C$, independent of $j$. For $B_1$ satisfying $||B_1||_{\mathcal{Y}}\leq \min\{\delta_1, 1/\sqrt{2C}\}=\delta_2$, (2.29) holds.
\end{proof}

\begin{proof}[Proof of Lemma 2.10]
Since $||B_1||_{\mathcal{Y}_{T_0}}\to 0$ as $T_0\to 0$, we find that $||B_1||_{\mathcal{Y}_{T_0}}\leq \delta_2 $ for small $T_0>0$ and $\{B_j\}\subset \mathcal{X}_{T_0}$ is a Cauchy sequence in $\mathcal{X}_{T_0}$ by (2.29). Thus, there exists $B\in \mathcal{X}_{T_0}$ such that $B_j\to B$ in $\mathcal{X}_{T_0}$. By the boundedness of the bilinear operator (2.15),
\begin{align*}
u_j\to u=K_{d/2}(B,B)\quad \textrm{in}\ L^{\frac{p}{2}}(0,T; H^{m}),\quad m=\frac{d}{2}-1+2s.
\end{align*}
Letting $j\to \infty$ implies that the limit $B\in \mathcal{X}_{T_0}$ satisfies the integral equations (2.2). By Proposition 2.8, $B\in \mathcal{H}_{T_0}$ is a solution to (2.1). The uniqueness follows from Proposition 2.9.
\end{proof}

\subsection{Global well-posedness}

\begin{thm}
For $B_0\in H$ and $f\in L^{2}(0,T; V^{*})$, there exists a unique solution $B\in \mathcal{H}_T$ to (2.1) for $\gamma=d/2$ satisfying the initial condition $B(x,0)=B_0(x)$ and the energy inequality (1.12).
\end{thm}

\begin{proof}
Local-in-time solutions $B\in \mathcal{H}_{T_0}$ of (2.1) in Lemma 2.10 satisfy (2.1) in the distributional sense and 
\begin{align*}
0&=2<\partial_t B,B>-2\kappa <\Delta B,B>-2<\nabla \cdot (B\otimes u-u\otimes B),B>-2<f,B> \\
&=\frac{d}{dt}\int_{\mathbb{T}^{d}}|B|^{2}dx+2\kappa \int_{\mathbb{T}^{d}}|\nabla B|^{2}dx-2\int_{\mathbb{T}^{d}}(B\cdot \nabla u)\cdot Bdx+2\int_{\mathbb{T}^{d}}(u\cdot \nabla B)\cdot Bdx-2<f,B>.
\end{align*}
The fourth term vanishes since $u$ is solenoidal. By integration by parts and $(-\Delta)^{\frac{d}{4}}u=(-\Delta)^{-\frac{d}{4}}\Pi B\cdot \nabla B$, 
\begin{align*}
-\int_{\mathbb{T}^{d}}(B\cdot \nabla u)\cdot Bdx
=\int_{\mathbb{T}^{d}}\Pi(B\cdot \nabla B)\cdot udx
=\int_{\mathbb{T}^{d}}(-\Delta)^{-\frac{d}{4}} \Pi(B\cdot \nabla B)\cdot (-\Delta)^{\frac{d}{4}} udx
=\int_{\mathbb{T}^{d}} \left|(-\Delta)^{\frac{d}{4}} u\right|^{2}dx.
\end{align*}
By $2|<f,B>|\leq \kappa^{-1} ||f||_{H^{-1}}^{2}+\kappa ||\nabla B ||_{L^{2}}^{2}$ and integrating the above identity, (1.12) holds for $0\leq t\leq T_0$. Since the $L^{2}$-norm of the local-in-time solution $B$ is uniformly bounded, $B$ is extendable to a global-in-time solution in $(0,T)$ and (1.12) holds for $0\leq t\leq T$.
\end{proof}

\begin{rem}
The global well-posedness result (Theorem 2.14) also holds for $\gamma>d/2$ because the bilinear operator estimate (2.22) is also valid for $K_{\gamma}=(-\Delta)^{-(\gamma-d/2)}K_{d/2}$ by the boundedness of the operator $(-\Delta)^{-(\gamma-d/2)}$ on Sobolev spaces. We will apply the estimate (2.18) for $K_{\gamma}$ in \S 5.
\end{rem}

\subsection{The perturbed system}

We show the global well-posedness of the perturbed system for $\gamma>d/2$ and $Z\in \mathcal{X}_T$: 
\begin{equation}
\begin{aligned}
\partial_t b&=\kappa \Delta b+\nabla \cdot \left((b+Z)\otimes u-u\otimes (b+Z) \right),\\
u&=K_{\gamma}(b+Z,b+Z).
\end{aligned}
\end{equation}
We say that $b\in \mathcal{H}_T$ is a solution to (2.30) if $b$ satisfies (2.30$)_1$ for $u=K_{\gamma}(b+Z,b+Z)$ on $\mathbb{T}^{d}\times (0,T)$ in the distributional sense. 

\begin{thm}
For $B_0\in H$ and $Z\in \mathcal{X}_{T}$, there exists a unique solution $b\in \mathcal{H}_T$ to (2.30) for $\gamma>d/2$ satisfying the initial condition $b(x,0)=B_0(x)$. 
\end{thm}

\begin{proof}
We find a unique local-in-time mild solution  
\begin{equation*}
\begin{aligned}
B=e^{t\kappa \Delta}B_0+Z +\int_{0}^{t}e^{(t-s)\kappa\Delta}\nabla \cdot (B\otimes  u- u\otimes B)ds,\quad u=K_{\gamma}(B, B),
\end{aligned}
\end{equation*}
by applying a similar iteration argument to (2.2). By changing the unknown from $B$ to $b=B-Z$, we obtain a local-in-time unique mild solution $b\in \mathcal{H}_T$ to (2.30). 

We show that local-in-time solutions are global. The equations (2.30) can be expressed with the pressure function $p$ as 
\begin{equation}
\begin{aligned}
\partial_t b&=\kappa \Delta b+(b+Z)\cdot \nabla u-u \cdot \nabla (b+Z),\\
\nabla p&=(b+Z)\cdot \nabla (b+Z)-(-\Delta)^{\gamma}u, \\ 
\nabla \cdot b&=\nabla \cdot u=0.
\end{aligned}
\end{equation}
By multiplying $2b$ by the first equation and integrating by parts,
\begin{align*}
\frac{d}{dt}\int_{\mathbb{T}^{d}}|b|^{2}
+2\kappa \int_{\mathbb{T}^{d}}|\nabla b|^{2}dx
=2\int_{\mathbb{T}^{d}}\left((b+Z)\cdot\nabla  u\right)\cdot (b+Z)dx 
-2\int_{\mathbb{T}^{d}}\left((b+Z)\cdot\nabla  u\right)\cdot Zdx
-2\int_{\mathbb{T}^{d}}\left(u\cdot\nabla  Z\right)\cdot b dx. 
\end{align*}
By multiplying $2u$ by the second equation and integrating by parts,
\begin{align*}
2\int_{\mathbb{T}^{d}}\left|(-\Delta)^{\frac{\gamma}{2}}u \right|^{2}dx
=2\int_{\mathbb{T}^{d}}\left((b+Z)\cdot\nabla  (b+Z) \right)\cdot udx. 
\end{align*}
Taking the sum of the two equations,
\begin{align*}
\frac{d}{dt}\int_{\mathbb{T}^{d}}|b|^{2}
+2\kappa \int_{\mathbb{T}^{d}}|\nabla b|^{2}dx
+2\int_{\mathbb{T}^{d}}\left|(-\Delta)^{\frac{\gamma}{2}}u \right|^{2}dx
=2\int_{\mathbb{T}^{d}}\left( (b+Z)\cdot\nabla  Z\right)\cdot udx
-2\int_{\mathbb{T}^{d}}\left(u\cdot\nabla  Z\right)\cdot b dx. 
\end{align*}
The right-hand side is bounded by 
\begin{align*}
2||u||_{L^{\infty}}||\nabla Z||_{L^{2}}\left(2||b||_{L^{2}}+||Z||_{L^{2}}\right).
\end{align*}
By the continuous embedding (2.4), $H^{\gamma}\subset C_{b}^{\gamma-d/2}$ for $\gamma>d/2$. By Cauchy--Schwarz inequality, we obtain 
\begin{align*}
\frac{d}{dt}\int_{\mathbb{T}^{d}}|b|^{2}
+2\kappa \int_{\mathbb{T}^{d}}|\nabla b|^{2}dx
+\int_{\mathbb{T}^{d}}\left|(-\Delta)^{\frac{\gamma}{2}}u \right|^{2}dx
\leq C\left(|| b||_{L^{2}}^{2} + ||Z||_{C([0,T]L^{2})}^{2} \right)||\nabla Z||_{L^{2}}^{2},
\end{align*}
with a constant $C$. By Gronwall's inequality, the $L^{2}$-norm of $b$ is globally bounded, and the local-in-time solutions are global.
\end{proof}

\begin{rem}
The solution operator $b=b(B_0,Z): H\times \mathcal{X}_T\ni (B_0,Z)\mapsto b(B_0,Z)\in  \mathcal{H}_T$ of (2.30) is continuous.
\end{rem}

\section{Path-wise global well-posedness of the randomly forced resistive MRE}

We consider the filtered probability space $(\Omega,\mathcal{F},\mathcal{G}_t, \mathbb{P})$ and independent Brownian motions $\{\beta_j\}_{j=1}^{\infty}$ (see below) and show the path-wise global well-posedness of the system (1.11) for random initial data $B_0\in H$ and spatially regular white noise $f=\partial_t \zeta$ defined by the Wiener process $\zeta$ in (1.5). We also show that the constructed stochastic process is an It\^o process in $V^{*}$ with constant diffusion for applying It\^o formulas in the next section.

\subsection{The stochastic setting}

\subsubsection{A filtered probability space}
We work on a probability space $(\Omega, \mathcal{F}, \mathbb{P})$ with a right-continuous filtration $\mathcal{G}_t$, i.e., a family of $\sigma$-algebras $\mathcal{G}_t$ such that $\mathcal{G}_t\subset \mathcal{F}$ for $t\geq 0$, $\mathcal{G}_t$ is non-decreasing, and $\mathcal{G}_t=\cap_{s>t} \mathcal{G}_{s}$ for $t\geq 0$. We assume that a right continuous $(\Omega, \mathcal{F}, \mathcal{G}_t, \mathbb{P})$ satisfies usual hypothesis: $(\Omega, \mathcal{F}, \mathbb{P})$ is complete and for each $t\geq 0$, $\mathcal{G}_t$ contains all $\mathbb{P}$-null sets of $\mathcal{F}$. 

\subsubsection{Adapted and progressively measurable}
We say that an $H$-valued stochastic process $\{Z(t)\}$ is $\mathcal{G}_t$-adapted if $Z(t): (\Omega, \mathcal{G}_t) \to (H, \mathcal{B}(H))$ is measurable for any $t\geq 0$. We say that $\{Z(t)\}$ is $\mathcal{G}_t$-progressively measurable if $Z^{\omega}(s): ([0,t]\times \Omega, \mathcal{B}([0,t])\times \mathcal{G}_t)\to (H, \mathcal{B}(H))$ is mesurable for any $t\geq 0$. An $H$-valued $\mathcal{G}_t$-adapted stochastic process $\{Z(t)\}$ has a $\mathcal{G}_t$-progressively measurable modification \cite[Proposition 3.5]{DZ92}.

\subsubsection{Brownian motion}
We say that a $\mathcal{G}_t$-adapted stochastic process $\{\beta(t)\}$ on $(\Omega, \mathcal{F}, \mathcal{G}_t, \mathbb{P})$ is a Brownian motion if 
\begin{itemize}
\item $\beta(0)=0$ almost surely
\item For almost sure $\omega\in \Omega$, $\beta(t)=\beta^{\omega}(t)$ is continuous for $t\geq 0$ 
\item For arbitrary points $0=t_0<t_1<t_2<\cdots<t_n$, increments $\{\beta(t_i)-\beta(t_{i-1})\}_{1\leq i\leq n}$ are independent and each $\beta(t_i)-\beta(t_{i-1})$ is normally distributed with mean zero and variance $t_{i}-t_{i-1}$.
\end{itemize}
The moment of the Brownian motion is $\mathbb{E}\beta^{2p}(t)=(2p-1)!! t^{p}$ and $\mathbb{E}\beta^{2p-1}(t)=0$ for $p\in \mathbb{N}$. We denote by $\sigma(\beta(t)|\ t\geq 0 )$ a family of $\sigma$-algebra generated by $\{\beta(t)\}$. We say that Brownian motions $\{\beta_j(t)\}_{j=1}^{\infty}$ are independent if a family of $\sigma$-algebra $\{\sigma(\beta_j(t)|\ t\geq 0 )\}_{j=1}^{\infty}$ is independent.  

\subsubsection{Martingales}
We say that an $H$-valued continuous stochastic process $\{M_t\}$ is a martingale concerning $\mathcal{G}_t$ if 
\begin{itemize}
\item $\mathbb{E}||M_t||_{H}$ is finite for $t\geq 0$
\item $M_t$ is $\mathcal{G}_t$-adapted
\item $\mathbb{E}[M_t\ |\ \mathcal{G}_s]= M_s$ for any $t\geq s\geq 0$, a.s.
\end{itemize}
Here, $Y_t=\mathbb{E}[M_t\ |\ \mathcal{G}_s]$ is a conditional expection, i.e., $Y_t$ is $\mathcal{G}_s$-adapted and $\mathbb{E}[Y_t\ |\ A]=\mathbb{E}[M_t\ |\ A]$ for $A\in \mathcal{G}_s$. For $H=\mathbb{R}$, we define a submartingale $\{M_t\}$ by replacing the last condition with a weaker condition $\mathbb{E}[M_t\ |\ \mathcal{G}_s]\geq M_s$ for any $t\geq s\geq 0$, a.s. For a martingale $\{M_t\}$, $||M_t||_{H}$ is a submartingale. Non-negative submartingales $\{M_t\}$ satisfy Doob's moment inequality  
\begin{align*}
\mathbb{E}\left(\sup_{0\leq t\leq T}M_t^{p}\right)\leq \left(\frac{p}{p-1}\right)^{p}\mathbb{E}M_T^{p},\quad 1<p<\infty.
\end{align*}
We say that a random variable $\tau: (\Omega,\mathcal{F},\mathcal{G}_t)\to \mathbb{R} $ is $\mathcal{G}_t$-stopping time if $\{\tau\leq t\}\in \mathcal{G}_t$ for any $t\geq 0$ (The random variables $\tau=0$ and $\tau=t$ are $\mathcal{G}_t$-stopping times.) For two almost surely finite stopping times $\sigma\leq \tau$, submartingales $\{M_t\}$ satisfy Doob's optional sampling theorem   
\begin{align*}
\mathbb{E}M_\sigma=\mathbb{E}M_\tau.
\end{align*}
In particular, $\mathbb{E}M_0=\mathbb{E}M_t$ for $t\geq 0$.

\subsection{It\^o processes with constant diffusion}

We will apply It\^o formula (B.1) (resp. (B.5)) for It\^o processes in $H$ (resp. $V^{*}$) with constant diffusion for finite (resp. infinite) dimensional stochastic processes.
  
\begin{defn}[It\^o processes in $H$ with constant diffusion]
Let $\{g_j\}_{j=1}^{\infty}\subset H$ be a sequence such that 
\begin{align}
\sum_{j=1}^{\infty}||g_j||_{H}^{2}<\infty.
\end{align}
Let $\{f(t)\}$ be a $\mathcal{G}_t$-progressively meaurable $H$-valued stochastic process such that 
\begin{align}
\mathbb{P}\left(\int_{0}^{T}||f(t)||_{H}dt<\infty \right)=1,\quad \textrm{for any}\ T>0.
\end{align}
We say that a $\mathcal{G}_t$-adapted and continuous $H$-valued stochastic process $\{B(t)\}$ is an It\^o process in $H$ with constant diffusion if 
\begin{align}
B(t)=B(0)+\int_{0}^{t}f(s)ds+\sum_{j=1}^{\infty}g_j  \beta_j(t)\quad \textrm{on}\ H,\quad t\geq 0,\quad \textrm{a.s.}
\end{align}
\end{defn}

\begin{defn}[It\^o processes in $V^{*}$ with constant diffusion]
Let $\{g_j\}_{j=1}^{\infty}\subset H$ be as in Definition 3.1. Let $\{f(t)\}$ be a $\mathcal{G}_t$-progressively measurable $V^{*}$-valued stochastic process such that 
\begin{align}
\mathbb{P}\left(\int_{0}^{T}||f(t)||_{V^{*}}^{2}dt<\infty \right)=1,\quad \textrm{for any}\ T>0.
\end{align}
We say that a $\mathcal{G}_t$-adapted and continuous $H$-valued stochastic process $\{B(t)\}$ is an It\^o process in $V^{*}$ with constant diffusion if (3.3) holds on $V^{*}$.
\end{defn}

\subsection{The Wiener process}

Let $\{b_j\}_{j=1}^{\infty}$ be a sequence such that $\mathsf{C}_0$ in (1.6) is finite. Let $\{e_j\}_{j=1}^{\infty}\subset H$ be a complete orthogonal system of $H$ consisting of eigenfunctions of the Stokes operator $-\Delta e_j=\lambda_j e_{j}$ with the eigenvalues $0<\lambda_1\leq \lambda_2\leq \cdots\leq \lambda_j\to\infty$. For $d=3$, we choose $\{e_j\}_{j=1}^{\infty}\subset H$ by eigenfunctions of the rotation operator $\nabla \times e_{j}=\tau_j e_j$ for $\tau_j$ satisfying $\tau_j^{2}=\lambda_j$ (see Theorems A.1 and A.4). 

We first define the Wiener process (1.5) in $L^{2}\left(\Omega;  C([0,T]; H)\right)$. 

\begin{prop}
The random process
\begin{align}
\zeta_n=\sum_{j=1}^{n}b_j e_j \beta_j(t)
\end{align}
belongs to $C([0,T]; H_n)$ for $H_n=\textrm{span}\ (e_1,\cdots,e_n)$ almost surely and satisfies 
\begin{align}
\zeta_n\to \zeta=\sum_{j=1}^{\infty}b_j e_j \beta_j(t) \quad \textrm{in}\ L^{2}\left(\Omega; C([0,T]; H)\right).
\end{align}\\
In particular, $\zeta\in L^{2}\left(\Omega;  C([0,T]; H)\right)$.
\end{prop}

\begin{proof}
The stochastic process $\zeta_{m,n}=\zeta_{m}-\zeta_{n}=\sum_{j=n+1}^{m}b_j e_j \beta_j$ for $m>n$ is an $H$-valued martingale. Since $||\zeta_{m,n}||_{H}(t)$ is a submartingale, applying Doob's moment inequality implies 
\begin{align*}
\mathbb{E}\left(\sup_{0\leq t\leq T}||\zeta_{m,n}||_{H}^{2} \right)
\leq 4\mathbb{E} ||\zeta_{m,n}||_{H}^{2}(T)
=4 \sum_{j=n+1}^{m}|b_j|^{2}\mathbb{E}\beta_j^{2}(T)
=4T \sum_{j=n+1}^{m}|b_j|^{2},
\end{align*}
and the convergence (3.6) follows.   
\end{proof}

\subsection{Stochastic convolution}

We say that an $H$-valued random process $\{Z(t)\}$ is a solution to the inhomogeneous Stokes equations (2.12) for $f=\partial_t \zeta$ and $\zeta$ in (1.5) if $Z\in \mathcal{X}_T$ almost surely and 
\begin{align}
Z(t)=Z(0)+\int_{0}^{t}\kappa\Delta Z(s)ds+\zeta(t)\quad \textrm{on}\ V^{*},\ 0\leq t\leq T,\ a.s.
\end{align}
We define unique solutions to (2.12) by stochastic convolution. 

\begin{lem}[Stochastic convolution]
Let $\zeta\in L^{2}\left(\Omega;  C([0,T]; H)\right)$ be as in (1.5). Set  
\begin{align}
Z(t)=\int_{0}^{t}\kappa\Delta e^{(t-s)\kappa\Delta}\zeta(s)ds+\zeta(t)=:\int_{0}^{t}e^{(t-s)\kappa \Delta}d\zeta(s).
\end{align}
Then, $Z\in L^{2}(\Omega; \mathcal{X}_T )$ is a unique solution to (2.12) for $Z(0)=0$. The $H$-valued random process $\{Z(t)\}$ is an It\^o process in $V^{*}$ with constant difusion.
\end{lem}

\begin{prop}
Let $\zeta_n$ be as in (3.5). Set 
\begin{align}
Z_n(t)=\int_{0}^{t}\kappa\Delta e^{(t-s)\kappa\Delta}\zeta_n(s)ds+\zeta_n(t).
\end{align}
Then, $Z_n$ is a solution to (2.12) for $Z_n(0)=0$ and $f=\partial_t \zeta_n$. The process $\{Z_n(t)\}$ is an It\^o process in $H$ with constant diffusion.
\end{prop}

\begin{proof}
Since the heat semigroup acts as a multiplier operator for eigenfunctions, i.e., $e^{t\Delta}e_{j}=e^{-\lambda_j t}e_j$, $Z_n\in C([0,T]; H_n)$ for $H_n=\textrm{span}(e_1,\cdots,e_n)$ almost surely. In particular, $Z_n\in \mathcal{X}_T$. By integrating 
\begin{align*}
Z_n(t)
=\int_{0}^{t}\kappa\Delta e^{(t-s)\kappa\Delta}\zeta_n(s)ds+\zeta_n(t)
=\frac{d}{dt}\int_{0}^{t} e^{(t-s)\kappa\Delta}\zeta_n(s)ds, 
\end{align*}
we find that 
\begin{align*}
\int_{0}^{t}\kappa \Delta Z_n(s)ds=\int_{0}^{t}\kappa \Delta e^{(t-s)\kappa\Delta}\zeta_n(s)ds.
\end{align*}
Thus, 
\begin{align}
Z_n(t)=\int_{0}^{t}\kappa \Delta Z_n(s)ds+\zeta_n(t),\quad t\geq 0,\quad \textrm{a.s.},
\end{align}
and $Z_n$ is a solution to (2.12). By (3.9), $\Delta Z_n$ is $\mathcal{G}_t$-adapted and hence $\mathcal{G}_t$-progressively measurable. Since $Z_n\in C([0,T]; H_n)$ almost surely,   
\begin{align*}
\mathbb{P}\left(\int_{0}^{T}||\Delta Z_n||_{H}dt<\infty \right)=1.
\end{align*}
Thus, $\{Z_n\}$ is an It\^o process in $H$ with constant diffusion. 
\end{proof}

\begin{prop}
The sequence $\{Z_n\}$ converges to $Z$ in (3.8) in $C\left([0,T]; L^{2}(\Omega; V^{*})\right)$.
\end{prop}

\begin{proof}
By the boundedness of the operator $t^{1/2}\Delta e^{t\kappa\Delta}:H \to V^{*}$ and the continuous embedding $H \subset V^{*}$,
\begin{align*}
||Z_m-Z_n||_{V^{*}}(t)=\left\|\int_{0}^{t}\kappa\Delta e^{(t-s)\kappa\Delta}(\zeta_m-\zeta_n)(s)ds+(\zeta_m-\zeta_n)(t)  \right\|_{V^{*}}
&\lesssim_\kappa \int_{0}^{t}\frac{1}{(t-s)^{\frac{1}{2}}}||\zeta_m-\zeta_n||_{H}ds+||\zeta_m-\zeta_n||_{V^{*}}(t) \\
&\lesssim_\kappa (T^{\frac{1}{2}}+1) \sup_{0\leq t\leq T}||\zeta_m-\zeta_n||_{H}.
\end{align*}
By taking the expectation and supremum for $t\in [0,T]$,
\begin{align*}
\sup_{0\leq t\leq T}\mathbb{E}||Z_m-Z_n||_{V^{*}}^{2}(t)
\lesssim_\kappa (T+1)\mathbb{E} \sup_{0\leq t\leq T} ||\zeta_m-\zeta_n||_{H}^{2}(t).
\end{align*}
Since $\zeta_n$ converges in $L^{2}(\Omega; C([0,T]; H))$ by (3.6), $Z_n$ converges in $C([0,T]; L^{2}(\Omega; V^{*}) )$. By (3.9) and (3.8), $Z_n$ converges to $Z$.
\end{proof}

\begin{prop}
The sequence $\{Z_n\}$ converges to $Z$ in (3.8) in $L^{2}(\Omega;\mathcal{X}_T)$ and the $H$-valued process $\{Z(t)\}$ is a solution to (2.12) for $Z(0)=0$.
\end{prop}

\begin{proof}
We set $Z_{m,n}=Z_m-Z_n$ and $\zeta_{m,n}=\zeta_m-\zeta_n$ for $m>n$ by $Z_n$ and $\zeta_n$ in (3.9) and (3.5). Then, $Z_{m,n}$ is an It\^o process in $H$ with constant diffusion by Proposition 3.5. By applying the It\^o formula (B.1) for $F(z)=||z||_{H}^{2}$,
\begin{align}
||Z_{m,n}||_{H}^{2}(t)+2\kappa \int_{0}^{t}||\nabla Z_{m,n}||_{H}^{2}ds=t \sum_{j=n+1}^{m}|b_j|^{2}+2\sum_{j=n+1}^{m}\int_{0}^{t}\left(Z_{m,n}, b_je_j\right)_{H}d \beta_j(s),\quad 0\leq t\leq T.
\end{align}
The last term denoted by $2M_t$ is a martingale. By Doob's optional sampling theorem, $\mathbb{E}M_t=\mathbb{E}M_0=0$. By taking the mean,
\begin{align}
\mathbb{E}||Z_{m,n}||_{H}^{2}(t)+2\kappa \mathbb{E}\int_{0}^{t}||\nabla Z_{m,n}||_{H}^{2}ds=t \sum_{j=n+1}^{m}|b_j|^{2},\quad 0\leq t\leq T.
\end{align}
By applying Doob's moment inequality,
\begin{align*}
\left(\mathbb{E}\sup_{0\leq t\leq T}\left|M_t \right|\right)^{2}
\leq \mathbb{E}\sup_{0\leq t\leq T}M_t^{2} 
\leq 4\mathbb{E}M_T^{2}.
\end{align*}
By the independence of $\beta_j$ and It\^o isometry \cite[Corollary 3.1.7]{Okse},
\begin{align*}
\mathbb{E}\left(\int_{0}^{T}(Z_{m,n}, b_je_j)_Hd \beta_j(s)\int_{0}^{T}(Z_{m,n}, b_le_l)_Hd \beta_l(s) \right)
=\delta_{j,l}\mathbb{E}\left(\int_{0}^{T}|(Z_{m,n},b_je_j)_{H}|^{2}ds \right).
\end{align*}\\
By the continuous embedding (2.5) and (3.12),  
\begin{align*}
\mathbb{E}M_T^{2}=
\mathbb{E}\left|\sum_{j=n+1}^{m}\int_{0}^{T}(Z_{m,n}, b_je_j)_Hd \beta_j(s)\right|^{2}
&=\sum_{j,l=n+1}^{m}\mathbb{E}\left(\int_{0}^{T}(Z_{m,n}, b_je_j)_Hd \beta_j(s)\int_{0}^{T}(Z_{m,n}, b_le_l)_Hd \beta_l(s) \right) \\
&=\mathbb{E}\sum_{j=n+1}^{m}\int_{0}^{T}|(Z_{m,n}, b_je_j)_H|^{2}ds \\
&\leq \left(\sum_{j=n+1}^{m}b_j^{2} \right)\mathbb{E}\int_{0}^{T}||Z_{m,n}||_{H}^{2}ds \\
&\leq \left(\sum_{j=n+1}^{m}b_j^{2} \right)\mathbb{E}\int_{0}^{T}||\nabla Z_{m,n}||_{H}^{2}ds \lesssim \frac{T}{\kappa}\left(\sum_{j=n+1}^{m}b_j^{2} \right)^{2}.
\end{align*}
By (3.11), we obtain 
\begin{align*}
\mathbb{E}\left(\sup_{0\leq t\leq T} ||Z_{m,n}||_{H}^{2}(t)+2\kappa \int_{0}^{T}||\nabla Z_{m,n}||_{H}^{2}ds\right)\lesssim \left(T+\sqrt{\frac{T}{\kappa}} \right)\left(\sum_{j=n+1}^{m}b_j^{2} \right).
\end{align*}\\
Thus, $\{Z_{n}\}$ converges in $L^{2}(\Omega;\mathcal{X}_T)$. By the continuous embedding $L^{2}\left(\Omega; C([0,T]; H)\right)\subset C\left([0,T]; L^{2}(\Omega; V^{*})  \right)$ and $Z\in C\left([0,T]; L^{2}(\Omega; V^{*})  \right)$, we find that $Z\in L^{2}(\Omega;\mathcal{X}_T)$. Letting $n\to\infty$ in (3.10) implies that $Z$ satisfies (3.7) and is a solution to (2.12) for $Z(0)=0$.
\end{proof}

\begin{proof}[Proof of Lemma 3.4]
The convergence $Z_n\to Z$ in $L^{2}(\Omega; \mathcal{X}_T)$ implies $Z_n(t)\to Z(t)$ in $H$ for $t\geq 0$ almost surely. Since the $H$-valued process $Z_n$ is $\mathcal{G}_t$-adapted, so is $Z$, cf. \cite[4.2.2 Theorem]{Dud}. By $\Delta Z_n\to \Delta Z$ in $L^{2}(\Omega; L^{2}(0,T; V^{*}) )$, 
\begin{align*}
\mathbb{P}\left(\int_{0}^{T}||\Delta Z||_{V^{*}}^{2}dt<\infty \right)=1.
\end{align*}
In particular, $\Delta Z_n\to \Delta Z$ in $V^{*}$ for a.e. $(s,\omega)\in [0,t]\times \Omega$. Since $\Delta Z_n$ is $\mathcal{G}_t$-progressively measurable, so is $\Delta Z$. We showed that $\{Z(t)\}$ is an It\^o process in $V^{*}$ with constant diffusion.

It remains to show the uniqueness. For two random solutions $Z_1$ and $Z_2$ of (2.12), $Z=Z_1-Z_2$ satisfies 
\begin{align*}
Z(t)=\int_{0}^{t}\kappa\Delta Z(s)ds\quad \textrm{on}\ V^{*},\ 0\leq t\leq T,\ a.s.
\end{align*}
By differentiating, $\partial_t Z=\kappa\Delta Z\in L^{2}(0,T; V^{*})$ and $Z\in \mathcal{H}_T$ almost surely. Thus, $Z=0$ follows from the uniqueness of the deterministic case (Proposition 2.2).
\end{proof}

\subsection{Path-wise global well-posedness}
We say that an $H$-valued random process $\{B(t)\}$ is a solution to the system (1.11) for $\gamma>d/2$, $f=\partial_t \zeta$ and $\zeta$ in (1.5) if $B\in \mathcal{X}_T$ almost surely and 
\begin{align}
B(t)=B(0)
+\int_{0}^{t}\left(\kappa \Delta B+\nabla \cdot (B\otimes u-u\otimes B)  \right)ds
+\zeta(t)\quad \textrm{on}\ V^{*},\ 0\leq t\leq T,\ a.s.,
\end{align}
for the velocity field $u=K_{\gamma}(B,B)$. 

\begin{thm}[Path-wise global well-posedness]
For $\mathcal{G}_0$-measurable random initial data $B_0\in H$, there exists a unique solution $B(t)$ to (1.11) for $\gamma>d/2$, $f=\partial_t \zeta$ and $\zeta$ in (1.5) satisfying the initial condition $B(x,0)=B_0(x)$ almost surely. Moreover, 
\begin{align}
B(t)=e^{t\kappa \Delta}B_0+\int_{0}^{t}e^{(t-s)\kappa \Delta}d \zeta(s)
+\int_{0}^{t}e^{(t-s)\kappa \Delta} \nabla \cdot (B\otimes u-u\otimes B)ds,\quad a.s.
\end{align}
The random process $\{B(t)\}$ belong to $\mathcal{X}_T$ almost surely and is an It\^o process in $V^{*}$ with constant diffusion.    
\end{thm}

\begin{proof}
We first show the uniqueness. We set $B=B^{1}-B^{2}$ and $u=u^{1}-u^{2}$ for two solutions $B^{1}$ and $B^{2}$ and the associated velocity fields $u^{1}=K_{\gamma}(B^{1},B^{1})$ and $u^{2}=K_{\gamma}(B^{2},B^{2})$. Then, for 
\begin{align*}
F=B\otimes u^{1}-u^{1}\otimes B+B^{2}\otimes u-u\otimes B^{2},
\end{align*}
the function $B$ satisfies 
\begin{align*}
B(t)=\int_{0}^{t}\left(\kappa \Delta B(s)+\nabla \cdot F(s)\right) ds\quad \textrm{on}\ V^{*},\ 0\leq t\leq T,\ a.s.
\end{align*}
By the same uniqueness argument as that of the deterministic case (for $\gamma=d/2$ in Proposition 2.9), we conclude $B=0$ almost surely. 

We set the stochastic convolution $Z\in L^{2}(\Omega; \mathcal{X}_T)$ by (3.8) and take a set of full measure $\Omega_0\subset \Omega$ such that $B_0^{\omega}\in H$ and $Z^{\omega}\in \mathcal{X}_{T}$ for $\omega\in \Omega_0$. We set $b^{\omega}$ by a global-in-time unique solution to the perturbed system (2.30) for $\omega\in \Omega_0$ by Theorem 2.16 and $b^{\omega}=0$ for $\omega\in \Omega_0^{c}$. Then, $B=b+Z$ is a solution to (1.11) for $f=\partial_t \zeta$ and satisfies (3.14) almost surely.

It remains to show that $\{B(t)\}$ is an It\^o process in $V^{*}$ with constant diffusion. We show that $B(t)$ is a $\mathcal{G}_t$-adapted $H$-valued process and $f=\kappa \Delta B+\nabla \cdot (B\otimes u-u\otimes B)$ is a $\mathcal{G}_t$-progressively measurable $V^{*}$-valued process. By construction, $B=b+Z$ and $Z$ is a $\mathcal{G}_t$-adapted $H$-valued process (Lemma 3.4). The continuity of the solution operator $b=b(B_0,Z): H\times \mathcal{X}_T\to \mathcal{H}_T\subset \mathcal{X}_T$ (Remark 2.17) implies the continuity for fixed $t\geq 0$, i.e.,  $b=b(B_0,Z): H\times H\to H$. Thus $b=b(B_0, Z(t))$ is a $\mathcal{G}_t$-adapted $H$-valued process.

The function $f=f(B)=\kappa \Delta B+\nabla \cdot (B\otimes u-u\otimes B): V\to V^{*}$ is also continuous for fixed $t\geq 0$. Since $B\in \mathcal{X}_T$ almost surely, $B: ([0,t]\times \Omega, \mathcal{B}[0,t]\times \mathcal{G}_t)\to (V,\mathcal{B}(V))$ is measurable. Namely, $B$ is a $\mathcal{G}_t$-progressively measurable $V$-valued process. Thus $f=f(B(t))$ is a $\mathcal{G}_t$-progressively measurable $V^{*}$-valued process and 
\begin{align*}
\mathbb{P}\left(\int_{0}^{T}||f||_{V^{*}}^{2}dt<\infty \right)=1.
\end{align*}
We showed that $\{B(t)\}$ is an It\^o process in $V^{*}$ with constant diffusion. 
\end{proof}

\subsection{Invariant measures and statistically stationary solutions}

According to \cite[Chapter 7]{DZ92}, we define an invariant measure and a statistically stationary solution for the system (1.11). We denote by $B(t; \xi)$ a solution to (1.11) for $f=\partial_t \zeta$ satisfying the initial condition $B(0; \xi)=\xi\in H$. We define the Markov semigroup on the space of all bounded Borel functions $B_b(H)$ by 
\begin{align*}
P_{t}\varphi(\xi)=\mathbb{E}\varphi(B(t; \xi) ),\quad  \xi\in H,\ t\geq 0.
\end{align*}
The integral form (3.14) implies the continuity $B(t;\xi_n)\to B(t;\xi)$ in $H$ as $\xi_n\to \xi$ in $H$ and that the Markov semigroup is a Feller semigroup, i.e., $P_{t}\in \mathcal{L}(C_b(H))$, where $\mathcal{L}(C_b(H))$ denotes the space of all bounded lienar operators on $C_b(H)$. We define the dual semigroup on the space of probability measures $\mathcal{P}(H)$. For $\varphi\in B_b(H)$ and $\mu\in \mathcal{P}(H)$, we set 
\begin{align*}
(\varphi,\mu)=\int_{H}\varphi(\xi)\mu(d\xi),
\end{align*}
and define the dual semigroup $P_{t}^{*}\in \mathcal{L}(\mathcal{P}(H))$ by 
\begin{align*}
(\varphi,P^{*}_{t}\mu)=(P_{t} \varphi,\mu).
\end{align*}
The dual semigroup provides the evolution of the law $P^{*}_{t}\mu=\mathcal{D}(B(t; B_0))$ for random $B_0\in H$ with law $\mu=\mathcal{D}(B_0)$. We say that $\mu\in \mathcal{P}(H)$ is an invariant measure of (1.11) if $P_{t}^{*}\mu=\mu$ for all $t\geq 0$. We say that $B(t)$ is a statistically stationary solution (or a stationary process) to (1.11) if the law $\mathcal{D}(B(t))$ is time-independent. For an invariant measure $\mu$ and random initial data $B_0$ with law $\mathcal{D}(B_0)=\mu$, the global-in-time solution $B(t)$ in Theorem 3.8 is a statistically stationary solution with law $\mathcal{D}(B(t))=P^{*}_{t}\mu=\mu$.

\section{The existence of invariant measures}

We apply It\^o formulas for It\^o processes in $V^{*}$ with constant diffusion associated with the system (1.11) for $\gamma>d/2$, the force $f=\partial_t \zeta$ and $\zeta$ in (1.5) constructed in Theorem 3.8. We obtain an energy balance, an exponential moment decay, a helicity balance, and a mean-square potential balance. We then deduce the existence of an invariant measure for the system (1.11) from the energy balance and Krylov--Bogoliubov theorem. 

\subsection{The energy, helicity, and mean-square-potential balances}

\begin{lem}[Energy balance]
The following holds for solutions $B(t)$ to (1.11) in Theorem 3.8 for $B_0\in L^{2}(\Omega; H)$:
\begin{align}
\mathbb{E}||B||_{H}^{2}(t)+2\mathbb{E}\int_{0}^{t}\left( \kappa ||\nabla B||_{H}^{2}+\left\|u\right\|^{2}_{\dot{H}^{\gamma}}\right)ds
=\mathbb{E}||B_0||_{H}^{2}+\mathsf{C}_0t,\quad t\geq 0.
\end{align}
\end{lem}

\begin{proof}
The stochastic process $\{B_t\}$ is an It\^o process in $V^{*}$ with constant difusion and belongs to $\mathcal{X}_T$ almost surely for
\begin{equation}
\begin{aligned}
f&=\kappa \Delta B+\nabla \cdot (B\otimes u-u\otimes B),\quad u=K_{\gamma}(B,B),\\
g_j&=b_je_j.
\end{aligned}
\end{equation}
By the application of the It\^o formula (B.16),
\begin{align*}
||B||_{H}^{2}(t\wedge \tau_n)
=||B_0||_{H}^{2}+2\mathbb{E}\int_{0}^{t\wedge \tau_n}(B_s,f_s)ds+t\wedge \tau_n \sum_{j=1}^{\infty}||g_j||_{H}^{2}+2\sum_{j=1}^{\infty}\int_{0}^{t\wedge \tau_n}(B_s,g_j)_H d\beta_j(s),\quad t\geq 0,
\end{align*}
with the stopping time $\tau_n=\inf\{t\geq 0\ |\ ||B(t)||_{H}>n\ \}$ for $n\in \mathbb{N}$. By multiplying $(-\Delta)^{\gamma}u$ by $u=K_{\gamma}(B,B)$,
\begin{align*}
\left\|u\right\|^{2}_{\dot{H}^{\gamma  }}
=\left\| (-\Delta)^{\frac{\gamma}{2}}u\right\|_{H}^{2}
=\left(u,(-\Delta)^{\gamma } u \right)_{H}
=(u,B\cdot \nabla B)_{H}=-(B\cdot \nabla u,B )_{H}.
\end{align*}
By integration by parts, $(B,f)_H=-\kappa ||\nabla B||_{H}^{2}-\left\|u\right\|^{2}_{\dot{H}^{\gamma }}$ and we obtain 
\begin{align*}
||B||_{H}^{2}(t\wedge \tau_n)
+2\int_{0}^{t\wedge \tau_n}\left( \kappa ||\nabla B||_{H}^{2}+\left\|u\right\|^{2}_{\dot{H}^{\gamma  }}\right)ds
=||B_0||_{H}^{2}+t\wedge \tau_n \mathsf{C}_0+2\sum_{j=1}^{\infty}\int_{0}^{t\wedge \tau_n}(B_s,g_j)_H d\beta_j(s).
\end{align*}
By taking the mean and applying Doob's optional sampling theorem,  
\begin{align*}
\mathbb{E}||B||_{H}^{2}(t\wedge \tau_n)
+2\mathbb{E}\int_{0}^{t\wedge \tau_n}\left( \kappa ||\nabla B||_{H}^{2}+\left\|u\right\|^{2}_{\dot{H}^{\gamma }}\right)ds
=\mathbb{E}||B_0||_{H}^{2}+\mathbb{E}(t\wedge \tau_n) \mathsf{C}_0.
\end{align*}
Since $B\in C([0,\infty); H)$ almost surely, $\tau_n(\omega)$ monotonously diverges as $n\to\infty$ almost surely. Letting $n\to\infty$ implies (4.1).

%By the energy balance (4.1), 
%\begin{align*}
%\sum_{j=1}^{\infty}\mathbb{E}\int_{0}^{t}|(B,b_je_j)_H|^{2}ds
%&=\sum_{j=1}^{\infty}b_j^{2}\mathbb{E}\int_{0}^{t}|(B,e_j)_H|^{2}ds \\
%&\leq \mathsf{C}_0 \mathbb{E}\int_{0}^{t}||B||_H^{2}ds
%\leq \mathsf{C}_0 \mathbb{E}\int_{0}^{t}||\nabla B||_H^{2}ds 
%\leq \frac{\mathsf{C}_0}{2\kappa}\left(\mathbb{E}||B_0||_{H}^{2}+\mathsf{C}_0t \right)<\infty.
%\end{align*}
%Thus, $M_t$ is a square-integrable martingale with a continuous trajectory and (4.2) holds by Theorem B.2. By applying Doob's moment inequality for a non-negative martingale $|M_t|$,
%\begin{align*}
%\left(\mathbb{E}\sup_{0\leq t\leq T}M_t\right)^{2} \leq \mathbb{E}\left(\sup_{0\leq t\leq T}M_t^{2}\right)
%\leq 4\mathbb{E}M_T^{2}.
%\end{align*}
%By the independence of $\beta_{j}$ and It\^o isometry,
%\begin{align*}
%\mathbb{E}\left(\int_{0}^{T}(B,g_j)_Hd \beta_j(s)\int_{0}^{T}(B,g_l)_Hd \beta_l(s)\right)=\delta_{j,l}\mathbb{E}\left(\int_{0}^{T}|(B,g_j)_H|^{2}d s\right).
%\end{align*}
%By (4.1), 
%\begin{align*}
%\mathbb{E}M_T^{2}
%&=\mathbb{E}\left[\left(\sum_{j=1}^{\infty}\int_{0}^{T}(B,g_j)_Hd \beta_j(s) \right)\left(\sum_{l=1}^{\infty}\int_{0}^{T}(B,g_l)_Hd \beta_l(s) \right) \right]  \\
%&=\sum_{j,l=1}^{\infty}\delta_{j,l}\mathbb{E}\left[\int_{0}^{T}|(B,g_j)_H|^{2}d s \right] 
%=\sum_{j=1}^{\infty}b_j^{2}\mathbb{E}\int_{0}^{T}|(B,e_j)_H|^{2}d s
%\leq \frac{\mathsf{C}_0}{2\kappa}\left(\mathbb{E}||B_0||_{H}^{2}+\mathsf{C}_0T \right).
%\end{align*}
%Thus, (4.4) follows from (4.2).
\end{proof}

\begin{lem}[Helicity and mean-square-potential balances]
The following holds for solutions $B(t)$ to (1.11) in Theorem 3.8 for $B_0\in L^{2}(\Omega; H)$:
\begin{align}
\mathbb{E}(A,B)_{H}(t)
+2\kappa \mathbb{E}\int_{0}^{t}(\nabla \times B,B)_{H}ds
&=\mathbb{E}(A_0,B_0)_{H}
+\mathcal{C}_{-\frac{1}{2}}t,\quad A=\textrm{curl}^{-1}B,\quad t\geq 0,\quad d=3,\\
\mathbb{E}||\phi||_{L^{2}}^{2}(t)
+2\kappa  \mathbb{E}\int_{0}^{t}||\nabla \phi||_{L^{2}}^{2}ds
&=\mathbb{E}||\phi_0||_{L^{2}}^{2}
+\mathsf{C}_{-1}t,\quad \phi=\textrm{curl}^{-1} B,\quad t\geq 0,\quad  d=2.
\end{align}
\end{lem}

\begin{proof}
By the application of the It\^o formula (B.18), 
\begin{align*}
(A,B)_H(t\wedge \tau_n)
=&(A_0,B_0)_H
+2\int_{0}^{t\wedge \tau_n}(B,\textrm{curl}^{-1} f_s)_{H}ds+t\wedge \tau_n \sum_{j=1}^{\infty}(\textrm{curl}^{-1}g_j,g_j )_{H}\\
&+2\sum_{j=1}^{\infty}\int_{0}^{t\wedge \tau_n}(B, \textrm{curl}^{-1} g_j)_{H}d\beta_j(s),\quad t\geq 0.
\end{align*}
By $f=-\nabla \times ( \kappa \nabla \times B+B\times u)$ and $\nabla \times e_j=\tau_j e_j$,
\begin{align*}
(B,\textrm{curl}^{-1} f)_H&=-(B,\kappa \nabla \times B+B\times u)_{H}
=-\kappa (B, \nabla \times B)_{H},\\
\left(\textrm{curl}^{-1} g_j,g_j \right)_H&=b_j^{2}\left(\textrm{curl}^{-1}e_j,e_j\right)_{H}=\frac{1}{\tau_j}b_j^{2}.
\end{align*}
We thus obtain
\begin{align*}
(A, B)_H(t\wedge \tau_n)
+2\kappa \int_{0}^{t\wedge \tau_n}(B,\nabla \times B)_{H}ds
=(A_0, B_0)_H+t\wedge \tau_n \mathcal{C}_{-\frac{1}{2}}
+2\sum_{j=1}^{\infty}\int_{0}^{t\wedge \tau_n}(B, \textrm{curl}^{-1} g_j)_{H}d\beta_j(s).
\end{align*}
By taking the mean and applying Doob's optional sampling theorem,
\begin{align*}
\mathbb{E}(A, B)_H(t\wedge \tau_n)
+2\kappa \mathbb{E}\int_{0}^{t\wedge \tau_n}(B,\nabla \times B)_{H}ds
=\mathbb{E}(A_0, B_0)_H+\mathbb{E}(t\wedge \tau_n) \mathcal{C}_{-\frac{1}{2}}.
\end{align*}
Letting $n\to\infty$ implies (4.3).

By the application of the It\^o formula (B.19), 
\begin{align*}
||\phi||^{2}_{L^{2}}(t\wedge \tau_n)
=&||\phi_0 ||^{2}_{L^{2}}
+2\int_{0}^{t\wedge \tau_n}\left(\phi_s,\textrm{curl}^{-1} f_s\right)_{L^{2}}ds+t\wedge \tau_n \sum_{j=1}^{\infty}\left\|(-\Delta)^{-\frac{1}{2}}g_j \right\|_{L^{2}}^{2}\\
&+2\sum_{j=1}^{\infty}\int_{0}^{t\wedge \tau_n}\left(\phi_s,\textrm{curl}^{-1} g_j\right)_{L^{2}}d\beta_j(s),\quad t\geq 0.
\end{align*}
By using $B=\nabla^{\perp}\phi$ and $u\cdot \nabla B-B\cdot \nabla u=-\nabla^{\perp}(u\cdot \nabla \phi)$, $f=\nabla^{\perp}(\kappa\Delta \phi-u\cdot \nabla \phi)$ and 
\begin{align*}
\left(\phi, \textrm{curl}^{-1} f \right)_{L^{2}}
=\left(\phi, \kappa\Delta \phi -u\cdot \nabla \phi\right)_{L^{2}}
=-\kappa ||\nabla \phi||_{L^{2}}^{2}. 
\end{align*}
By $-\Delta e_j=\lambda_j e_j$, $\|(-\Delta)^{-\frac{1}{2}}g_j \|_{L^{2}}^{2}=((-\Delta)^{-1}g_j,g_j )_{L^{2}}=b_j^{2}/\lambda_j$. We obtain (4.4) by taking the mean, applying Doob's optional sampling theorem, and letting $ n\to \infty$.
\end{proof}

\subsection{The exponential moment}

\begin{lem}
Let $0<\varrho \leq 1/(2\sup_{j\geq 1}b_j^{2})$. The following holds for solutions $B(t)$ to (1.11) in Theorem 3.8 for $B_0\in L^{2}(\Omega; H)$ such that $\mathbb{E}\exp\left(\kappa \varrho ||B_0||_H^{2} \right)$ is bounded:
\begin{align}
\mathbb{E}\exp \left(\kappa \varrho ||B||_H^{2}  \right)(t)\leq 
\mathbb{E}\exp\left(\kappa \varrho ||B_0||_H^{2} \right) e^{-\kappa^{2}\varrho t}+\frac{1}{\kappa}(\mathsf{C}_0+\kappa)e^{\varrho(\mathsf{C}_0+\kappa) },\quad t\geq 0.
\end{align}
\end{lem}

\begin{proof}
By the application of the It\^o formula (B.17) with $\delta=\kappa \varrho$,
\begin{equation*}
\begin{aligned}
\exp\left(\kappa \varrho  ||B||_{H}^{2} \right)(t\wedge \tau_n)
=&\exp\left(\kappa \varrho ||B_0||_{H}^{2} \right)+\kappa \varrho \int_{0}^{t\wedge \tau_n}\exp\left(\kappa \varrho ||B||_{H}^{2} \right)\left(2(B_s,f_s)_H +\sum_{j=1}^{\infty}\left(||g_j||_{H}^{2}+2\kappa \varrho |(B_s,g_j)_H|^{2} \right)   \right)ds\\
&+2\kappa \varrho \sum_{j=1}^{\infty}\int_{0}^{t\wedge \tau_n} \exp\left(\kappa \varrho ||B||_{H}^{2} \right) (B_s,g_j)_{H}d\beta_j(s),
\end{aligned}
\end{equation*}
for $f$ and $g_j$ in (4.2). By taking the mean and applying Doob's optional sampling theorem,
\begin{equation*}
\begin{aligned}
\mathbb{E}\exp\left(\kappa \varrho ||B||_{H}^{2} \right)(t\wedge \tau_n)
=\mathbb{E}\exp\left(\kappa \varrho ||B_0||_{H}^{2} \right)+\kappa \varrho \mathbb{E}\int_{0}^{t\wedge \tau_n}\exp\left(\kappa \varrho ||B||_{H}^{2} \right)\left(2(B_s,f_s)_H +\sum_{j=1}^{\infty}\left(||g_j||_{H}^{2}+2\kappa \varrho |(B_s,g_j)_H|^{2} \right)   \right)ds.
\end{aligned}
\end{equation*}
By $(B,f)_H=-\kappa ||\nabla B||_{H}^{2}-||u||_{\dot{H}^{\gamma }}^{2}\leq -\kappa  || B||_{H}^{2}$ and the condition for $\varrho$ in (4.5),
\begin{align*}
\sum_{j=1}^{\infty} \left(||g_j||_{H}^{2}+2\kappa \varrho |(B,g_j)_H|^{2} \right) 
=\mathsf{C}_0+2\kappa \varrho \sum_{j=1}^{\infty}b_j^{2}|(B,e_j)_H|^{2}
\leq \mathsf{C}_0+\kappa  ||B||_{H}^{2}.
\end{align*}
By combining these two estimates and using the pointwise estimate,
\begin{align*}
e^{\varrho r}\left(\mathsf{C}_0 +\kappa- r  \right)\leq (\mathsf{C}_0 +\kappa)e^{\varrho(\mathsf{C}_0 +\kappa)},\quad r\geq 0,
\end{align*}
we estimate 
\begin{align*}
\exp\left(\kappa \varrho ||B||_{H}^{2} \right)\left(2(B_s,f_s)_H +\sum_{j=1}^{\infty}\left(||g_j||_{H}^{2}+2\kappa \varrho |(B_s,g_j)_H|^{2} \right)   \right)
&\leq \exp\left(\kappa \varrho ||B||_{H}^{2} \right) \left(-\kappa  ||B||_{H}^{2}+\mathsf{C}_0 \right) \\
&=-\kappa \exp\left(\kappa \varrho ||B||_{H}^{2} \right)+\exp\left(\kappa \varrho ||B||_{H}^{2} \right) \left( \mathsf{C}_0+\kappa-\kappa ||B||_{H}^{2} \right) \\
&\leq -\kappa\exp\left(\kappa \varrho ||B||_{H}^{2} \right)+(\mathsf{C}_0+\kappa)e^{\varrho(\mathsf{C}_0 +\kappa)} \\
&=\kappa\left( -\exp\left(\kappa \varrho ||B||_{H}^{2} \right)+\frac{1}{\kappa}(\mathsf{C}_0+\kappa) e^{\varrho(\mathsf{C}_0 +\kappa)}\right).
\end{align*}
We thus obtain 
\begin{equation*}
\begin{aligned}
\mathbb{E}\exp\left(\kappa\varrho  ||B||_{H}^{2} \right)(t\wedge \tau_n)
+\kappa^{2}\varrho \mathbb{E}\int_{0}^{t\wedge \tau_n}\exp\left(\kappa\varrho ||B||_{H}^{2} \right)ds
\leq \mathbb{E}\exp\left(\kappa\varrho ||B_0||_{H}^{2} \right)+\kappa^{2}\varrho 
\left(\frac{1}{\kappa}(\mathsf{C}_0+\kappa) e^{\varrho(\mathsf{C}_0 +\kappa)}\right) \mathbb{E}(t\wedge \tau_n).
\end{aligned}
\end{equation*}
By letting $n\to\infty$,
\begin{equation*}
\begin{aligned}
\mathbb{E}\exp\left(\kappa\varrho  ||B||_{H}^{2} \right)(t)
+\kappa^{2}\varrho \mathbb{E}\int_{0}^{t}\exp\left(\kappa\varrho ||B||_{H}^{2} \right)ds
\leq \mathbb{E}\exp\left(\kappa\varrho ||B_0||_{H}^{2} \right)+\kappa^{2}\varrho 
\left(\frac{1}{\kappa}(\mathsf{C}_0+\kappa) e^{\varrho(\mathsf{C}_0 +\kappa)}\right)t.
\end{aligned}
\end{equation*}
By applying Gr\"{o}nwall's inequality (Proposition C.1), we obtain (4.5).
\end{proof}

\subsection{Krylov--Bogoliubov theorem}

We apply Krylov--Bogoliubov theorem for a stochastic system on a Fr\'{e}chet space $S$ and deduce the existence of invariant measures. Let $C_b(S)$ denote a space of bounded and continuous functions on $S$. We say that a sequence of probability measures $\{\nu_k\}_{k=1}^{\infty}\subset \mathcal{P}(S)$ weakly converges to $\nu$ if for arbitrary $\varphi\in C_b(S)$,
\begin{align*}
(\varphi,\nu_k)\to (\varphi,\nu)\quad \textrm{as}\ k\to\infty.
\end{align*}
We denote the weak convergence of the measures by $\nu_k\to \nu$ in $\mathcal{P}(S)$. We apply Portmanteau, Prokhorov, and Skorokhod theorems for convergence of measures; see, e.g.,  \cite[11.1.1, 11.7.2]{Dud}, \cite[Theorems 2.3, 2.4]{DZ92}, \cite[p.15 and Theorem 1.2.14]{Kuk12}.

\begin{prop}[Portmanteau theorem]
The weak convergence of the measures $\nu_k\to \nu$ in $\mathcal{P}(S)$ is equivalent to the condition $\nu(U)\leq \liminf_{n\to\infty}\nu_n(U)$ for all open sets $U\subset S$ or $\limsup_{n\to\infty}\nu_n(F)\leq \nu(F)$ for all closed sets $F\subset S$.
\end{prop}

\begin{prop}[Prokhorov theorem]
A family of measures $\{\nu_\lambda\}_{\lambda\in \Lambda}\subset \mathcal{P}(S)$ is relatively compact if and only if it is tight in the sense that for arbitrary $\varepsilon>0$, there exists a compact set $K_{\varepsilon}\Subset S$ such that 
\begin{align*}
\nu_\lambda(K_{\varepsilon})\geq 1-\varepsilon,\quad \lambda\in \Lambda.
\end{align*}
\end{prop}

\begin{prop}[Skorokhod theorem]
If a family of measures $\{\nu_k\}\subset \mathcal{P}(S)$ weakly converges to $\nu$ as $k\to\infty$, there exists a probability space $(\Omega,\mathcal{F},\mathbb{P})$ and random variables $X_k, X: (\Omega,\mathcal{F},\mathbb{P})\to (S,\mathcal{B}(S))$ such that $\mathcal{D}(X_k)=\nu_k$, $\mathcal{D}(X)=\nu$, and $X_k\to X$ in $S$ almost surely.   
\end{prop}

We apply the following existence theorem of invariant measures for a Feller semigroup $P_t\in \mathcal{L}(C_b(S))$ associated with a continuous stochastic process $B(t;\cdot)$ on $S$ \cite[Theorem 3.1.1]{DaPrato96}. Let $0\leq P(t,\xi, \Gamma)=\mathbb{E}1_{\Gamma}(B(t;\xi))\leq 1$ denote a transition function for $t\geq 0$, $ \xi\in S$, and $\Gamma\in \mathcal{B}(S)$. For the dual semigroup $P_t^{*}\in \mathcal{L}(\mathcal{P}(S))$ and $\nu_t=P^{*}_t \nu_0=\int_{S}P(t,\xi,\cdot)\nu_0(d\xi)$ for $\nu_0\in \mathcal{P}(S)$, the Ces\`{a}ro mean 
\begin{align*}
\overline{\nu}_t(\Gamma)=\fint_{0}^{t}\nu_s(\Gamma) ds=\int_{S}\left(\fint_{0}^{t}P(s,\xi,\Gamma)ds  \right)\nu_0(d\xi),
\end{align*}\\
is a probability measure $\overline{\nu}_t\in \mathcal{P}(S)$ for $t>0$. 

\begin{prop}[Krylov--Bogoliubov theorem]
Assume that $\overline{\nu}_{t_n}$ weakly converges to $\nu$ in $\mathcal{P}(S)$ for some $\nu_0\in \mathcal{P}(S)$ and a sequence $\{t_n\}$ such that $t_n\to\infty$. Then, $\nu$ is an invariant measure for $P_t$.
\end{prop}

\subsection{The existence of invariant measures}

In Lemma 4.8 below, we express integrals of functions $\varphi: H\to \mathbb{R}$ by an invariant mesure $\mu$ by using a statistically stationary solution $B^{S}$ and the associated velocity $u^{S}=K_{\gamma}(B^{S},B^{S})$ with law $\mathcal{D}(B^{S})=\mu$. Namely, 
\begin{align*}
\int_{H}\varphi(B) \mu(dB)
=\mathbb{E} \varphi(B^{S}).
\end{align*}
We suppress the superscript $S$ on the symbols of statistically stationary solutions, $B=B^{S}$ and $u=u^{S}$.

\begin{lem}
There exists an invariant measure $\mu$ to the stochastic system (1.11). Moreover, any invariant measure satisfies the following: 
\begin{align}
\mathbb{E}\left(\kappa ||\nabla B||_{H}^{2}+||u||_{\dot{H}^{\gamma }}^{2} \right)
&=\frac{\mathsf{C}_0}{2}, \\
\mathbb{E}\exp\left(  \kappa \varrho ||B||_{H}^{2} \right) 
&\leq \frac{1}{\kappa}(\mathsf{C}_0+\kappa)e^{\varrho(\mathsf{C}_0+\kappa) },\quad 0<\varrho \leq \frac{1}{2\sup_{j\geq 1}b_j^{2} },\\
\kappa\mathbb{E}\left(\nabla \times B, B   \right)_H
&=\frac{\mathcal{C}_{-\frac{1}{2}}}{2},\quad d=3, \\
\kappa\mathbb{E}\left\|B   \right\|_H^{2}
&=\frac{\mathsf{C}_{-1}}{2},\quad d=2.
\end{align}
Invariant measures $\mu\in \mathcal{P}(H)$ satisfy $\mu(V)=1$. 
\end{lem}

\begin{proof}
For $B_0=0$ and a solution $B(t; 0)$ to (1.11), $\delta_0=\mathcal{D}(0)$ and $\mu_t=P^{*}_t\delta_0=\mathcal{D}(B(t; 0))$. By Chebyshev's inequality, 
\begin{align*}
\mu_t(H\backslash B_V(R))=\mathcal{D}\left( B(t;0) \right)\left(H\backslash B_V(R)\right)=\mathbb{P}\left(||B(t;0)||_{V}\geq R\right)\leq \frac{1}{R^{2}}\mathbb{E}||B||_{V}^{2},
\end{align*}
where $B_V(R)$ denotes an open ball in $V$ with radius $R>0$. By the energy balance (4.1),
\begin{align*}
2\kappa\mathbb{E}\fint_{0}^{t}||B||_{V}^{2}ds\leq \mathsf{C}_0.
\end{align*}
Thus, the Ces\`{a}ro mean $\overline{\mu}_t=\fint_{0}^{t}\mu_s ds$ satisfies 
\begin{align*}
\sup_{t>0}\overline{\mu}_t\left(H\backslash B_V(R) \right)\leq \frac{\mathsf{C}_0}{2\kappa R^{2}}.
\end{align*}
Since $B_V(R)\Subset H$, $\{\overline{\mu}_t\}_{t>0}$ is tight. We take a sequence $\{t_n\}$ such that $t_n\to\infty$ and apply Prokhorov theorem (Proposition 4.5) to obtain a subsequence $\{\overline{\mu}_{t_n}\}$ weakly converging to a measure $\mu$ in $\mathcal{P}(H)$. Then, the limit $\mu$ is an invariant measure to (1.11) by Krylov--Bogoliubov theorem (Proposition 4.7).

For an invariant measure $\mu$, we take a random initial data $B_0$ with law $\mathcal{D}(B_0)=\mu$. Then, the global-in-time solution $B^{S}(t; B_0)$ of (1.11) is a statistically stationary solution with law $\mathcal{D}(B^{S}(t; B_0))=P^{*}_{t}\mu=\mu$. Since the law $\mathcal{D}(B^{S}(t; B_0))=\mu$ is time-independent, the constants 
\begin{align*}
\mathbb{E}||B^{S}||_{H}^{2}&=\int_{H}||B||_{H}^{2}\mu(dB)=\mathbb{E}||B_0||_{H}^{2},\\
\mathbb{E}\left(\kappa ||B^{S}||_{V}^{2}+||u^{S}||_{\dot{H}^{\gamma }} \right)&=\int_{H}\left(\kappa ||B||_{V}^{2} +||u||_{\dot{H}^{\gamma }}^{2} \right) \mu(dB),
\end{align*}
are time-independent. Thus, (4.6) follows from (4.1). Similarly, (4.8) and (4.9) follow from (4.3) and (4.4). By (4.6), 
\begin{align*}
\mu(B_V(R)^{c})=\int_{\{||B||_V\geq R\}}\mu(dB)\leq \frac{1}{R^{2}}\int_{H}||B||_{V}^{2}\mu(d B)\lesssim \frac{\mathsf{C}_0}{\kappa R^{2}}.
\end{align*}
Thus,
\begin{align*}
1=1-\lim_{R\to \infty}\mu(B_V(R)^{c})=\lim_{R\to \infty}\mu(B_V(R))
=\mu\left(\bigcup_{R\geq 1}B_V(R)\right)=\mu(V).
\end{align*}

We set $\varphi(B)=\exp\left( \kappa \varrho ||B||_{H}^{2} \right) $ and 
\begin{align*}
\varphi_R(B)=
\begin{cases}
\ \varphi(B)&\quad ||B||_H< R,\\
\ \exp\left( \kappa \varrho R^{2} \right) &\quad ||B||_H\geq R.
\end{cases}
\end{align*}
Then, $\varphi_R(B)\leq \varphi(B)$ and $\varphi_R(B)\leq \exp\left( \kappa \varrho R^{2} \right)$ for $B\in H$. We take an arbitrary $r>0$. For an invariant measure $\mu\in \mathcal{P}(H)$ and the Markov semigroup $P^{*}_t$,
\begin{align*}
\int_{H}\varphi_R(B)\mu (dB)=(\varphi_R, \mu)=(\varphi_R, P^{*}_t\mu)
=(P_t \varphi_R, \mu)
=\int_{||\xi||_H< r}(P_t \varphi_R)(\xi)\mu(d\xi)+\int_{||\xi||_H\geq r}(P_t \varphi_R)(\xi)\mu(d\xi).
\end{align*}
For $||\xi ||_{H}< r$, we apply the exponential moment estimate (4.5) to estimate 
\begin{align*}
(P_t \varphi_R)(\xi)=\mathbb{E}\varphi_R(B(t;\xi))
\leq \mathbb{E}\varphi (B(t;\xi))
\leq \varphi (\xi) e^{-\kappa^{2} \varrho t}+\frac{1}{\kappa}(\mathsf{C}_0+\kappa)e^{\varrho(\mathsf{C}_0+\kappa) }
\leq \exp\left(\kappa\varrho r^{2} \right)e^{-\kappa^{2} \varrho t} +\frac{1}{\kappa}(\mathsf{C}_0+\kappa)e^{\varrho(\mathsf{C}_0+\kappa) }.
\end{align*}
For $||\xi ||_{H}\geq r$, we estimate $(P_t\varphi_R)(\xi)\leq \exp(\kappa\varrho R^{2} )$. Thus we have
\begin{align*}
\int_{H}\varphi_R(B)\mu (dB)
\leq \exp\left(\kappa\varrho r^{2} \right)e^{-\kappa^{2} \varrho t} +\frac{1}{\kappa}(\mathsf{C}_0+\kappa)e^{\varrho(\mathsf{C}_0+\kappa) }+\exp\left( \kappa \varrho R^{2} \right) \mu\left(H\backslash B_H(r)\right).
\end{align*}
By letting $t\to\infty$ and $r\to\infty$, 
\begin{align*}
\int_{H}\varphi_R(B)\mu (dB)
\leq \frac{1}{\kappa}(\mathsf{C}_0+\kappa)e^{\varrho(\mathsf{C}_0+\kappa) }.
\end{align*}
Letting $R\to\infty$ and Fatou's lemma imply (4.7).  
\end{proof}

\section{The non-resistive limits}

We complete the proof of Theorem 1.1. By Lemma 4.8, invariant measures $\mu_{\kappa}$ of (1.3) exist. We denote by $B_{\kappa}$ the statistically stationary solution with law $\mathcal{D}(B_{\kappa})=\mu_{\kappa}$ and by $u_{\kappa}=K_{\gamma}(B_{\kappa},B_{\kappa})$ the associated velocity field.

\subsection{Balance relations}

\begin{lem}
Any invariant measure $\mu_{\kappa}$ to the system (1.3) for $\gamma>d/2$ satisfies the following:  
\begin{align}
\mathbb{E}\left(\kappa ||\nabla B_{\kappa}||_{H}^{2}+\left\|u_{\kappa}\right\|_{\dot{H}^{\gamma  }}^{2} \right)
&=\kappa\frac{\mathsf{C}_0}{2},\\
\mathbb{E}\exp\left(\rho ||B_{\kappa}||_{H}^{2}\right)
&\leq (\mathsf{C}_0+1)e^{\rho(\mathsf{C}_0+1)},
\quad 0<\rho\leq \frac{1}{2\sup_{j\geq 1}b_j^{2}},\\
\mathbb{E}\left(\nabla \times B_{\kappa}, B_{\kappa}   \right)_H
&=\frac{\mathcal{C}_{-\frac{1}{2}}}{2},\quad d=3, \\
\mathbb{E}\left\|B_{\kappa}   \right\|_H^{2}
&=\frac{\mathsf{C}_{-1}}{2},\quad d=2.
\end{align}
The measure $\mu_{\kappa}\in \mathcal{P}(H)$ satisfies $\mu_k(V)=1$. 
\end{lem}

\begin{proof}
The claimed properties follow from the properties of invariant measures for (1.11) with force $f=\partial_t \zeta$ as stated in Lemma 4.8.  
\end{proof}

\begin{prop}
For $\varepsilon>0$, there exists a subsequence such that $\mu_{\kappa}\to \mu_0$ in $\mathcal{P}(H^{1-\varepsilon})$ for some measure $\mu_0$ as $\kappa\to 0$. The measure $\mu_0$ satifies $\mu_0(V)=1$.
\end{prop}

\begin{proof}
By Chebyshev's inequality and (5.1),
\begin{align*}
\mu_{\kappa}\left(\overline{B_{V }(R)}^{c}\right)=\int_{\{||B||_V>R\}}\mu_{\kappa}(dB)\leq \frac{1}{R^{2}}\int_{H}|| B||_{V}^{2}\mu_k(dB)\lesssim \frac{\mathsf{C}_0}{R^{2}}.
\end{align*}
Thus, $\mu_k(\overline{B_{V }(R)})\geq 1- C/R^{2}$ for some $C>0$. Since $\overline{B_{V }(R)}\Subset H^{1-\varepsilon}$, $\{\mu_k\}\subset \mathcal{P}(H^{1-\varepsilon})$ is tight. By Prokhorov theorem (Proposition 4.5), there exists a subsequence such that $\mu_{\kappa}\to \mu_0$ on $\mathcal{P}(H^{1-\varepsilon})$. By Portmanteau theorem (Proposition 4.4), 
\begin{align*}
\mu_0(\overline{B_{V }(R)})\geq \limsup_{\kappa\to 0}\mu_{\kappa}(\overline{B_{V }(R)})
\geq 1- \frac{C}{R^{2}}.
\end{align*}
By $\mu_{0}(V)\leq \mu_{0}(H^{1-\varepsilon})=1$, $\mu_0(V)=\lim_{R\to\infty}\mu_0(\overline{B_{V }(R)})=1$. 
\end{proof}

\subsection{Lifted invariant measures}

For an invariant measure $\mu_{\kappa}$ in Proposition 5.2, we take a random initial data $B_{0}$ with law $\mathcal{D}(B_0)=\mu_{\kappa}$. Then, the global-in-time solution $B_{\kappa}$ of (1.3) is a statistically stationary solution with law $\mathcal{D}(B_{\kappa})=\mu_{\kappa}$. We define a lifted invariant measure $\boldsymbol{\mu}_{\kappa}$ of $\mu_{\kappa}\in \mathcal{P}(H)$ by the law of a statistically stationary solution $B_{\kappa}: (\Omega,\mathcal{F},\mathbb{P})\to L^{2}_{\textrm{loc}}([0,\infty); V)$ as  
\begin{align}
\boldsymbol{\mu}_{\kappa}(A)=\mathcal{D}(B_{\kappa})(A),\quad A\in \mathcal{B}\left(L^{2}_{\textrm{loc}}\left([0,\infty); V\right)\right).
\end{align}
Since $B_{\kappa}\in \mathcal{X}_{T}$ for every $T>0$ by Theorem 3.8, $\boldsymbol{\mu}_{\kappa}$ is a probability measure on $C([0,\infty); H)\cap L^{2}_{\textrm{loc}}([0,\infty); V)$.

\begin{prop}
\begin{align}
\mathbb{E}\left(\kappa ||\nabla B_{\kappa}||_{L^{2}(0,T; H)}^{2}+||u_{\kappa}||_{L^{2}(0,T; \dot{H}^{\gamma })}^{2}\right)&= \kappa\frac{\mathsf{C}_0T}{2}, \\
\mathbb{E} ||B_{\kappa}||_{L^{r}(0,T; H)}^{r}&\lesssim_{r} \frac{T}{\rho^{r}}(\mathsf{C}_0+1)e^{\rho(\mathsf{C}_0+1)},\quad r>0,\ T>0.
\end{align}
\end{prop}

\begin{proof}
By integrating (5.1) in time, 
\begin{align*}
\mathbb{E}\left( \kappa ||\nabla B_{\kappa}||_{L^{2}(0,T; H)}^{2}+||u_{\kappa}||_{L^{2}(0,T; \dot{H}^{\gamma })}^{2} \right) 
=\int_{0}^{T}\mathbb{E}\left(\kappa ||B_{\kappa}||_{H^{1}}^{2}+||u_{\kappa}||_{\dot{H}^{\gamma  }}^{2} \right)dt
= \kappa \frac{\mathsf{C}_0T}{2}.
\end{align*}
Thus, (5.6) holds. By the moment bound (5.2),
\begin{align*}
\frac{\rho^{n}}{n!}\mathbb{E}||B_{\kappa}||_{H}^{2n} \leq \mathbb{E}\exp\left(\rho ||B_{\kappa}||_{H}^{2}\right)\leq (\mathsf{C}_0+1)e^{\rho(\mathsf{C}_0+1) },\quad n\in \mathbb{N}_0.
\end{align*}
For $r>0$, we choose $n\in \mathbb{N}$ such that $n\geq r/2>n-1$ and apply H\"older's inequality to estimate 
\begin{align*}
\mathbb{E}||B_{\kappa}||_{H}^{r}
\leq \left(\mathbb{E}||B_{\kappa}||_{H}^{2n}\right)^{\frac{r}{2n}}
\leq \left(\frac{n!}{\rho^{n}}  (\mathsf{C}_0+1)e^{\rho(\mathsf{C}_0+1)} \right)^{\frac{r}{2n}}
\leq \frac{1}{\rho^{\frac{r}{2}}}\left(\frac{r}{2}+1\right)^{\frac{r}{2}}(\mathsf{C}_0+1)e^{\rho(\mathsf{C}_0+1)}.
\end{align*}
Integrating both sides in time implies (5.7).  
\end{proof}

\subsection{A compact embedding of a Bochner space}

We set the fractional Bochner space for $1\leq p\leq \infty$ and $0<s<1$ by 
\begin{align*}
W^{s,p}(0,T; H)&=
 \left\{f\in L^{p}(0,T; H)\ \middle| \ ||f||_{W^{s,p}(0,T; H) }=||f||_{L^{p}(0,T; H) }+|f|_{W^{s,p}(0,T; H) }<\infty\ \right\}, \\
|f|_{W^{s,p}(0,T; H) }&=
\begin{cases}
&\ \displaystyle\left(\int_0^{T}\int_0^{T}\frac{||f(t)-f(\tau)||_{H}^{p} }{|t-\tau|^{1+s p}}d\tau dt \right)^{\frac{1}{p}},\quad 1\leq p<\infty, \\
&\ \displaystyle\textrm{ess sup}\left\{\frac{||f(t)-f(\tau)||_{H}}{|t-\tau|^{s}}\ \middle|\ t,\tau\in [0,T],\ t\neq \tau\ \right\},\quad p=\infty.
\end{cases}
 \end{align*}
For $s\geq 1$ satisfying $s\notin \mathbb{N}$, we set
\begin{align*}
W^{s,p}(0,T; H)=\left\{f\in W^{[s],p}(0,T; H)\ \middle|\  ||f||_{W^{s,p}(0,T; H)}=||f||_{W^{[s],p}(0,T; H)}+\left|\partial_t^{[s]}f  \right|_{W^{s-[s],p}(0,T; H)}<\infty \right\}.
\end{align*}
We set a space larger than $\mathcal{X}_{T}=C([0,T]; H)\cap L^{2}(0,T; V)$: 
\begin{align*}
\mathcal{X}^{-\varepsilon}_T=C([0,T]; H^{-\varepsilon} )\cap L^{2}(0,T; H^{1-\varepsilon} ),\quad \varepsilon>0.
\end{align*}
We fix $1\leq p<2$ and $\max\left\{1/4, 1/p-1/2\right\}<s<1/2$ and set $\mathcal{K}_T=\mathcal{W}_{T}\cap L^{2}(0,T; V)$ by the sum space $\mathcal{W}_T=\mathcal{W}_T^{1}+\mathcal{W}_T^{2}+\mathcal{W}_T^{3}$ for 
\begin{equation}
\begin{aligned}
\mathcal{W}_T^{1}&=C([0,T];H^{-1})\cap H^{1}(0,T; H^{-1}),\\
\mathcal{W}_T^{2}&=C([0,T]; H^{-1})\cap W^{1,p}(0,T; H^{-1}), \\
\mathcal{W}_T^{3}&=C([0,T]; H^{-1})\cap W^{s,4}(0,T; H^{-1}),
\end{aligned}
\end{equation}
normed with 
\begin{align*}
||B ||_{\mathcal{W}_T}=\inf\left\{||B^{1}||_{\mathcal{W}_T^{1}}+||B^{2}||_{\mathcal{W}_T^{2}}+||B^{3}||_{\mathcal{W}_T^{3}}\ \middle|\ B=B^{1}+B^{2}+B^{3},\ B^{i}\in \mathcal{W}_T^{i},\ i=1,2,3\   \right\}.
\end{align*}
We will deduce the tightness of lifted invariant measures $\{\boldsymbol{\mu}_{\kappa}\} \subset \mathcal{P}(\mathcal{X}_{T}^{-\varepsilon})$ from Prokhorov theorem by showing the compactness and the boundedness (1.20). 

We first show the compactness (1.20$)_1$ by Lions--Aubin--Simon theorem \cite[Corollary 9]{Simon}; see also \cite[Theorem 4.4]{Latocca}. 

\begin{prop}[Lions--Aubin--Simon theorem]
Let $X_0$ and $X_1$ be Banach spaces such that $X_1\Subset X_0$. Let $X$ be an intermediate space such that $X_1\subset X\subset X_0$ and  
\begin{align*}
||f||_{X}\leq C||f||_{X_1}^{1-\theta}||f||_{X_0}^{\theta},\quad f\in X_1,
\end{align*}
with some constants $\theta\in (0,1)$ and $C>0$. For $s_0,s_1\geq 0$ and $1\leq p_0,p_1\leq \infty$, set 
\begin{align*}
s_{\theta}=(1-\theta)s_1+\theta s_0,\quad \frac{1}{p_\theta}=\frac{1-\theta}{p_1}+\frac{\theta}{p_0},\quad s_{*}=s_{\theta}-\frac{1}{p_{\theta}}.
\end{align*}
Then, the following holds:

\noindent
(i) If $s_*\leq 0$, $W^{s_0,p_0}(0,T; X_0)\cap W^{s_1,p_1}(0,T; X_1)\Subset L^{r}(0,T; X)$ for $r<r_*=-1/s_*$. \\
(ii) If $s_*>0$, $W^{s_0,p_0}(0,T; X_0) \cap W^{s_1,p_1}(0,T; X_1)\Subset C([0,T]; X)$.
\end{prop}

\begin{prop}
For $\varepsilon>0$, 
\begin{align}
W^{s,p}(0,T; H^{-1})\cap L^{2}(0,T; H^{1})&\Subset L^{2}(0,T; H^{1-\varepsilon}), \\
C^{\alpha}([0,T]; H^{-1})\cap L^{2}(0,T; H^{1})&\Subset C([0,T]; H^{-\varepsilon}), \quad 0<\alpha<\frac{1}{2}.
\end{align}
In particular, 
\begin{equation}
\begin{aligned}
C^{\alpha}([0,T]; H^{-1})\cap W^{s,p}(0,T; H^{-1})\cap L^{2}(0,T; H^{1})\Subset \mathcal{X}^{-\varepsilon}_{T}.
\end{aligned}
\end{equation}
\end{prop}

\begin{proof}
By the interpolation inequality (2.6), 
\begin{align*}
||f||_{L^{2}(0,T; H^{1-\sigma\varepsilon})}
\leq ||f||_{L^{2}(0,T; H^{1-\varepsilon})}^{\sigma}||f||_{L^{2}(0,T; H^{1})}^{1-\sigma},\quad 0\leq \sigma\leq 1.
\end{align*}
By the continuous embedding $L^{2}(0,T; H^{1-\varepsilon_1})\subset L^{2}(0,T; H^{1-\varepsilon_2})$ for $\varepsilon_1\leq \varepsilon_2$, any sequence bouned in $L^{2}(0,T; H^{1})$ and compact in $L^{2}(0,T; H^{1-\varepsilon_0})$ for some $\varepsilon_0>0$ is compact in $L^{2}(0,T; H^{1-\varepsilon})$ for any $\varepsilon>0$. Similarly, any sequence bouned in $C(0,T; H^{-1})$ and compact in $C(0,T; H^{-\varepsilon_0})$ for some $\varepsilon_0>0$ is compact in $C([0,T]; H^{-\varepsilon})$ for any $\varepsilon>0$. Thus, it suffices to show that (5.9) and (5.10) hold for some $\varepsilon>0$.

We set 
\begin{align*}
X_0=H^{-1},\ X_1=H^{1},\ \textrm{and}\ X=H^{1-2\theta },\quad  0<\theta<1.
\end{align*}
By the compact embedding (2.5) and the interpolation inequality (2.6), $X_1\subset X\subset X_0$ satisfies the condition of Proposition 5.4. 

We take $s_0=s$, $p_0=p$, $s_1=0$, $p_1=2$, and observe that 
\begin{align*}
s_{*}=\theta\left(s+\frac{1}{2}-\frac{1}{p}\right)-\frac{1}{2}.
\end{align*}
By the condition $s+1/2-1/p>0$, we choose $\theta$ such that 
\begin{align*}
0<\theta<\frac{1}{2\left(s+\frac{1}{2}-\frac{1}{p} \right)}
\end{align*}
for which 
\begin{align*}
-\frac{1}{2}<s_{*}=\theta\left(s+\frac{1}{2}-\frac{1}{p} \right)-\frac{1}{2}<0.
\end{align*}
Then, Proposition 5.4 (i) implies (5.9) for $\varepsilon=2\theta$ by $2<-1/s_{*}$.

We take $s_0=\alpha$, $p_0=\infty$, $s_1=0$, $p_1=2$, and observe that 
\begin{align*}
s_{*}=\left(\alpha+\frac{1}{2}\right)\theta-\frac{1}{2}.
\end{align*}
We choose 
\begin{align*}
\frac{1}{2}<\frac{1}{2\alpha+1}<\theta<1,
\end{align*}
so that $s_{*}>0$ and apply Proposition 5.4 (ii) to obtain (5.10) for $\varepsilon=-1+2\theta>0$.
\end{proof}

\begin{lem}[Compact embedding of a Bochner space]
For $0<\alpha<\min\left\{s-1/4, 1-1/p\right\}$, 
\begin{align}
\mathcal{W}_T^{i}\subset C^{\alpha}([0,T]; H^{-1})\cap W^{s,p}(0,T; H^{-1}),\quad i=1,2,3.
\end{align}
Moreover, 
\begin{align}
\mathcal{K}_T = \mathcal{W}_T\cap L^{2}(0,T; H^{1})
\subset C^{\alpha}([0,T]; H^{-1})\cap W^{s,p}(0,T; H^{-1})\cap L^{2}(0,T; H^{1})
\Subset \mathcal{X}_T^{-\varepsilon},\quad \varepsilon>0.
\end{align}
\end{lem}

\begin{proof}
By $\mathcal{W}^{1}_{T}=C([0,T]; H^{-1})\cap H^{1}(0,T; H^{-1})\subset C^{1/2}([0,T]; H^{-1})$ and $H^{1}(0,T; H^{-1})\subset W^{s,p}(0,T; H^{-1})$ for $s<1/2$ and $1\leq p<2$, (5.12) holds for $i=1$. 

By $\mathcal{W}^{2}_{T}=C([0,T]; H^{-1})\cap W^{1,p}(0,T; H^{-1})\subset C^{1-1/p}([0,T]; H^{-1})\subset C^{\alpha}([0,T]; H^{-1})$ for $\alpha\leq 1-1/p$ and $W^{1,p}(0,T; H^{-1})\subset W^{s,p}(0,T; H^{-1})$ for $s<1/2$, (5.12) holds for $i=2$. 

By $\mathcal{W}^{3}_{T}=C([0,T]; H^{-1})\cap W^{s,4}(0,T; H^{-1})\subset C^{s-1/4}([0,T]; H^{-1})\subset C^{\alpha}([0,T]; H^{-1})$ for $\alpha\leq s-1/4$, (5.12) holds for $i=3$. By (5.12), $\mathcal{W}_{T}\subset C^{\alpha}([0,T]; H^{-1})\cap W^{s,p}(0,T; H^{-1})$. The compact embedding in (5.13) follows from (5.11).
\end{proof}

\subsection{Boundedness of statistically stationary solutions}

We show the boundedness (1.20$)_2$ by estimating the statistically stationary solution 
\begin{align}
B_{\kappa}(t)=\left(B_{\kappa}(0)+\kappa\int_{0}^{t}\Delta B_{\kappa}(s)ds\right)+\int_{0}^{t}\nabla \cdot (B_{\kappa}\otimes u_{\kappa}-u_{\kappa}\otimes B_{\kappa})ds+\sqrt{\kappa}\zeta:=B^{1}_{\kappa}+B^{2}_{\kappa}+B^{3}_{\kappa}
\end{align}
on each $\mathcal{W}^{i}$ for $i=1,2,3$. We apply the following estimate for $B^{3}_{\kappa}$ in $\mathcal{W}^{3}_{T}=C([0,T]; H^{-1})\cap W^{s,4}(0,T; H^{-1})$ \cite[p.222]{Kuk12}.

\begin{prop}
\begin{align}
\mathbb{E}\left\| \zeta\right\|_{\mathcal{W}_T^{3}}^{2}\lesssim (1+T^{\frac{1}{2}})T \mathsf{C}_0.
\end{align}
\end{prop}

\begin{proof}
By Proposition 3.3, $\zeta\in L^{2}(\Omega; C([0,T]; H))$ and $\mathbb{E}||\zeta||_{C([0,T]; H)}^{2}\lesssim \mathsf{C}_0T$. By $H \subset H^{-1}$, $\mathbb{E}||\zeta||_{C([0,T]; H^{-1})}^{2} \lesssim \mathsf{C}_0T$. For $0\leq \tau\leq t$,
\begin{align*}
||\zeta(t)-\zeta(\tau)||_{H}^{4}=\left(\sum_{j=1}^{\infty}b_j^{2}|\beta_j(t)-\beta_j(\tau)|^{2} \right)^{2}
=\sum_{i,j=1}^{\infty}b_i^{2}b_j^{2}|\beta_i(t)-\beta_i(\tau)|^{2}|\beta_j(t)-\beta_j(\tau)|^{2}. 
\end{align*}
Since $\beta_j(t)-\beta_j(\tau)$ is normally distributed with mean zero and variance $t-\tau$, the fourth-moment is $\mathbb{E}|\beta_j(t)-\beta_j(\tau)|^{4}=3(t-\tau)^{2}$. Thus,  
\begin{align*}
\mathbb{E}||\zeta(t)-\zeta(\tau)||_{H}^{4}
\leq \sum_{i,j=1}^{\infty}b_i^{2}b_j^{2}\left(\mathbb{E} |\beta_i(t)-\beta_i(\tau)|^{4}\right)^{\frac{1}{2}} \left(\mathbb{E} |\beta_i(t)-\beta_i(\tau)|^{4}\right)^{\frac{1}{2}}
=3(t-\tau)^{2} \mathsf{C}_0^{2}.
\end{align*}
By integrating in time using the condition $1/4<s<1/2$,
\begin{align*}
\mathbb{E}||\zeta||_{L^{4}(0,T; H)}^{4}&\leq T^{3} \mathsf{C}_0^{2},\\
\mathbb{E} \int_{0}^{T}\int_{0}^{T} \frac{||\zeta(t)-\zeta(\tau)||_{H}^{4} }{|t-\tau|^{1+4s}}d\tau dt 
&\lesssim T^{3} \mathsf{C}_0^{2}. 
\end{align*}
Thus, $\mathbb{E}||\zeta||_{W^{s,4}(0,T; H)}^{4}\lesssim T^{3} \mathsf{C}_0^{2}$ and (5.15) follows.
\end{proof}

We estimate $B^{1}_{\kappa}$ in $\mathcal{W}^{1}_{T}=C([0,T]; H^{-1})\cap H^{1}(0,T; H^{-1})$.

\begin{prop}
\begin{align}
\mathbb{E}\left\|B^{1}_{\kappa}\right\|_{\mathcal{W}^{1}_{T}}^{2}\lesssim \left(1+T\right)\mathsf{C}_0+\kappa^{2}(1+T+T^{3})T\mathsf{C}_0.
\end{align}
\end{prop}

\begin{proof}
By $\partial_t B^{1}_{\kappa}=\kappa \Delta B_{\kappa}$, 
\begin{align*}
||\partial_t B^{1}_{\kappa}||_{L^{2}(0,T; H^{-1})}\leq \kappa ||\nabla B_{\kappa}||_{L^{2}(0,T; H)}. 
\end{align*}
By (5.1), 
\begin{align*}
\mathbb{E}\left(\left\|\partial_t B^{1}_{\kappa}\right\|_{L^{2}(0,T; H^{-1})}^{2}\right)
\lesssim \kappa^{2}T\mathsf{C}_0.
\end{align*}
By H\"older's inequality,
\begin{align*}
\left\|B^1_{\kappa}\right\|_{H^{-1}}\leq ||B_{\kappa}(0)||_{H^{-1}}+\kappa  \int_{0}^{t}||\nabla B_{\kappa}||_{H}ds
\leq ||B_{\kappa}(0)||_{H^{1}}+\kappa t^{\frac{1}{2}} ||\nabla B_{\kappa}||_{L^{2}(0,T; H)}.
\end{align*}
By taking the sup-norm and the $L^{2}$-norm, 
\begin{align*}
\left\|B^1_{\kappa}\right\|_{C([0,T]; H^{-1})\cap L^{2}(0,T; H^{-1})}
\lesssim (1+T^{\frac{1}{2}})||B_{\kappa}(0)||_{H^{1}}+\kappa T^{\frac{1}{2}}(1+T) ||\nabla B_{\kappa}||_{L^{2}(0,T; H)}.
\end{align*}
By taking the mean and using (5.1) and (5.6), 
\begin{align*}
\mathbb{E}\left\|B^1_{\kappa}\right\|_{C([0,T]; H^{-1})\cap L^{2}(0,T; H^{-1})}^{2}
\lesssim \left(1+T\right)\mathsf{C}_0+\kappa^{2}T^{2}(1+T^{2})\mathsf{C}_0.
\end{align*}
By combinning this with $\mathbb{E}\left(\left\|\partial_t B^{1}_{\kappa}\right\|_{L^{2}(0,T; H^{-1})}^{2}\right)\lesssim \kappa^{2}T\mathsf{C}_0$, we obtain (5.16).
\end{proof}

We estimate $B^{2}_{\kappa}$ in $\mathcal{W}^{2}_{T}=C([0,T]; H^{-1})\cap W^{1,p}(0,T; H^{-1})$.

\begin{prop}
\begin{align}
\mathbb{E}\left\| B^{2}_{\kappa}\right\|_{\mathcal{W}_T^{2}}^{p}\lesssim \frac{1}{\rho^{p}}(1+T^{p-1}+T^{p})T(\mathsf{C}_0+1)e^{(1-\frac{p}{2})\rho(\mathsf{C}_0+1) },
\end{align}
for the constant $\rho$ satisfying the condition (5.2).
\end{prop}

\begin{proof}
By $\partial_t B_{\kappa}^{2}=\nabla \cdot F$ for $F=B_{\kappa}\otimes u_{\kappa}-u_{\kappa}\otimes B_{\kappa}$, 
\begin{align*}
\left\|\partial_t B_{\kappa}^{2}\right\|_{L^{p}(0,T; H^{-1}) }\leq ||F||_{L^{p}(0,T; H)}.
\end{align*}
By $B^{2}_{\kappa}(0)=0$, for the conjugate exponent $p'$ of $p$, 
\begin{align*}
\left\|B^{2}_{\kappa}\right\|_{H^{-1}}
\leq \int_{0}^{t}\left\|\partial_s B^{2}_{\kappa}\right\|_{H^{-1}}ds
\leq t^{\frac{1}{p'}}||F||_{L^{p}(0,T; H) }.
\end{align*}
By taking the supremum and the $L^{p}$-norm in time,   
\begin{align*}
\left\|B^{2}_{\kappa}\right\|_{C([0,T]; H^{-1})\cap L^{p}(0,T; H^{-1})}\lesssim (T^{\frac{1}{p'}}+T)||F||_{L^{p}(0,T; H)}.
\end{align*}
By combinning this estimate with $\left\|\partial_t B_{\kappa}^{2}\right\|_{L^{p}(0,T; H^{-1}) }\leq ||F||_{L^{p}(0,T; H)}$, we obtain 
\begin{align*}
\left\|B^{2}_{\kappa}\right\|_{W_T^{2}}^{p}\lesssim (1+T^{p-1}+T^{p})\left\|F\right\|_{L^{p}(0,T; H) }^{p}.
\end{align*}
By applying the cubic estimate (2.18) for $K_{\gamma}$ as noted in Remark 2.15, 
\begin{align*}
||F||_{L^{p}(0,T; H) }\lesssim ||B||_{L^{\frac{4p}{2-p}}(0,T; H) }^{2}||B||_{L^{2}(0,T; H^{1}) }.
\end{align*}
By H\"older's inequality,
\begin{align*}
\mathbb{E}||F||_{L^{p}(0,T; H)}^{p}
\lesssim 
\left( \mathbb{E}||B||_{L^{\frac{4p}{2-p}}(0,T; H) }^{\frac{4p}{2-p}} \right)^{1-\frac{p}{2}}
\left( \mathbb{E}||B||_{L^{2}(0,T; H^{1})}^{2} \right)^{\frac{p}{2}}.
\end{align*}
We apply (5.6) and (5.7) for $r=4p/(2-p)$ and obtain (5.17).
\end{proof}

\begin{lem}[Boundedness of statistically stationary solutions]
\begin{align}
\mathbb{E}||B_{\kappa}||_{\mathcal{K}_T}^{p} \leq C,\quad 0<\kappa\leq 1,
\end{align}
for some constant $C=C(\mathsf{C}_0,\rho, T, s,p)>0$.  
\end{lem}

\begin{proof}
The estimate (5.18) follows from (5.15)-(5.17) and (5.6). 
\end{proof}

\begin{lem}[Tightness of lifted invariant measures]
For $\varepsilon>0$, there exists a subsequence such that $\boldsymbol{\mu}_{\kappa}\to\boldsymbol{\mu}_0$ in $\mathcal{P}(\mathcal{X}_T^{-\varepsilon})$ as $\kappa\to0$. The measure $\boldsymbol{\mu}_0$ satisfies $\boldsymbol{\mu}_0(L^{2}(0,T; V))=1$.
\end{lem}

\begin{proof}
Since $B_{\kappa}\in \mathcal{K}_T$ almost surely by (5.18), $\boldsymbol{\mu}_{\kappa}(\mathcal{K}_T)$=1. We denote by $B_{\mathcal{K}_T}(R)$ an open ball with radius $R>0$ in $\mathcal{K}_T$. By the compactness (5.13), $\overline{B_{\mathcal{K}_T}(R)}\Subset \mathcal{X}_T^{-\varepsilon}$. By (5.18),
\begin{align*}
\boldsymbol{\mu}_{\kappa}(\overline{B_{\mathcal{K}_T}(R)}^{c})
=\int_{\{||B||_{\mathcal{K}_T}>R\}}\boldsymbol{\mu}_{\kappa}(dB)
\leq \frac{1}{R^{p}}\int_{L^{2}_{\textrm{loc}}([0,\infty); V)} ||B||_{\mathcal{K}_T}^{p}\boldsymbol{\mu}_{\kappa}(dB)
=\frac{1}{R^{p}} \mathbb{E}||B_{\kappa}||_{\mathcal{K}_T}^{p}\leq \frac{C}{R^{p}}.
\end{align*}
Thus, $\boldsymbol{\mu}_{\kappa}(\overline{B_{\mathcal{K}_T}(R)})\geq 1-C/R^{p}$ and the measures $\{\boldsymbol{\mu}_{\kappa}\}$ are tight. By Prokhorov theorem (Proposition 4.5), the convergence of the measures follows.

By (5.6),
\begin{align*}
\int_{L^{2}_{\textrm{loc}}([0,\infty); V)  }||\nabla B||_{L^{2}(0,T; H)}^{2}\boldsymbol{\mu}_k(d B)\leq \frac{\mathsf{C}_0}{2}T.
\end{align*}
We set the projection $\pi_N: H\to H\cap C^{\infty}$ by 
\begin{align*}
\pi_NB=\sum_{|k|\leq N}\hat{B}_{k}e^{ik\cdot x}.
\end{align*}
The projection $\pi_N: \mathcal{X}_T^{-\varepsilon}\to L^{2}(0,T; H^{1})$ is bounded and 
\begin{align*}
||\nabla \pi_N  B||_{L^{2}(0,T; H)}^{2}=\int_{0}^{T}\sum_{|k|\leq N}|k|^{2} |\hat{B}_k|^{2}dt
\leq \int_{0}^{T}\sum_{k\in \mathbb{Z}_0^{d}}|k|^{2} |\hat{B}_k|^{2}dt=|| \nabla B||_{L^{2}(0,T; H)}^{2}.
\end{align*}
We set a bounded and continuous function $\varphi_R: L^{2}(0,T; H^{1})\to [0,\infty)$ by
\begin{align*}
\varphi_R(B)=\begin{cases}
\quad ||\nabla B||_{L^{2}(0,T; H)}^{2},& \quad ||\nabla B||_{L^{2}(0,T; H)}\leq R,\\
\quad R^{2},&\quad ||\nabla B||_{L^{2}(0,T; H)}> R.
\end{cases}
\end{align*}
By $\varphi_R(B)\leq ||\nabla B||_{L^{2}(0,T; H)}^{2}$,
\begin{align*}
(\varphi_R\circ \pi_N,\boldsymbol{\mu}_k)=\int_{\mathcal{X}_T^{-\varepsilon}}\varphi_R(\pi_N B)\boldsymbol{\mu}_k(d B)\leq \frac{\mathsf{C}_0}{2}T.
\end{align*}\\
Since $\varphi_R\circ \pi_N\in C_b(\mathcal{X}_T^{-\varepsilon})$ and $\boldsymbol{\mu}_k\to \boldsymbol{\mu}_0$ on $\mathcal{P}(\mathcal{X}_T^{-\varepsilon})$, $(\varphi_R\circ \pi_N,\boldsymbol{\mu}_0)\leq \mathsf{C}_0T/2$. By letting $\kappa\to 0$, $R\to\infty$, and Fatou's lemma,
\begin{align*}
\int_{\mathcal{X}_T^{-\varepsilon}}\left(\int_{0}^{T}\sum_{|k|\leq N }|k|^{2} |\hat{B}_k|^{2}dt\right)\boldsymbol{\mu}_0(d B)
=
\int_{\mathcal{X}_T^{-\varepsilon}}||\nabla \pi_N B||_{L^{2}(0,T; H)}^{2}\boldsymbol{\mu}_0(d B)\leq \frac{\mathsf{C}_0}{2}T.
\end{align*}
By letting $N\to\infty$ and Fatou's lemma,
\begin{align*}
\int_{\mathcal{X}_T^{-\varepsilon}}|| \nabla B||_{L^{2}(0,T; H)}^{2}\boldsymbol{\mu}_0(d B)\leq \frac{\mathsf{C}_0}{2}T.
\end{align*}
In a similar manner to the proof of Proposition 5.2, we obtain $\boldsymbol{\mu}_0(L^2 (0,T; V)) = 1$. 
\end{proof}

\subsection{Convergence of statistically stationary solutions}

We show that statistically stationary solutions of (1.3) converge to a random MHS equilibrium. 
\begin{lem}
There exist statistically stationary solutions ${B}_\kappa$ of (1.3) with $\mathcal{D}({B}_\kappa)=\boldsymbol{\mu}_{\kappa}$ such that ${B}_\kappa$ converges to a random variable $B$ with $\mathcal{D}(B)=\boldsymbol{\mu}_0$ in the sense that  
\begin{align}
B_\kappa\to B\quad \textrm{in}\ \mathcal{X}_{T}^{-\varepsilon}\quad \textrm{a.s.},
\end{align}
for some probability space $(\tilde{\Omega},\tilde{\mathcal{F}},\tilde{\mathbb{P}})$.
\end{lem}

\begin{proof}
We apply Skorokhod theorem (Proposition 4.6) for measures $\{\boldsymbol{\mu}_{\kappa}\}\subset \mathcal{P}(\mathcal{X}^{-\varepsilon}_T)$ to take random variables $\tilde{B}_{\kappa}$ and $B$ on a probability space $(\tilde{\Omega},\tilde{\mathcal{F}},\tilde{\mathbb{P}})$ such that $\mathcal{D}(\tilde{B}_{\kappa})=\boldsymbol{\mu}_{\kappa}$, $\mathcal{D}(\tilde{B})=\boldsymbol{\mu}_{0}$ and $\tilde{B}_{\kappa}\to {B}$ in $\mathcal{X}_{T}^{-\varepsilon}$ almost surely.

We set 
\begin{align*}
X(B)=\int_{0}^{T}
\left\|B(t)-B(0)
-\int_{0}^{t}\left(\kappa \Delta B+\nabla \cdot (B\otimes u-u\otimes B)  \right)ds
-\zeta(t)\right\|_{V^{*}}dt,\quad u=K_{\gamma}(B,B),
\end{align*}
and observe that $X(B_\kappa)=0$, $\mathbb{P}$-a.s. for statistically stationary solutions $B_k$ on $(\Omega,\mathcal{F},\mathbb{P})$. By 
\begin{align*}
0=\mathbb{E}\frac{X(B_\kappa)}{1+X(B_\kappa)}=\int_{\Omega}\frac{X(B_\kappa)}{1+X(B_\kappa)}\mathbb{P}(d\omega)
=\int_{\mathcal{X}_T}\frac{X(B)}{1+X(B)}\boldsymbol{\mu}_{\kappa}(d B) 
=\tilde{\mathbb{E}}\frac{X(\tilde{B}_\kappa)}{1+X(\tilde{B}_\kappa)},
\end{align*}
we find that $X(\tilde{B}_\kappa)=0$, $\tilde{\mathbb{P}}$-a.s. and $\tilde{B}_\kappa$ is a statistically stationary solution on $(\tilde{\Omega},\tilde{\mathcal{F}},\tilde{\mathbb{P}})$. 
\end{proof}

In the sequel, we denote $(\tilde{\Omega},\tilde{\mathcal{F}},\tilde{\mathbb{P}})$ by $({\Omega},{\mathcal{F}},{\mathbb{P}})$.

\begin{prop}
The limit $B$ belongs to $L^{2}(0,T; V)$ almost surely and $\mathcal{D}(B(\cdot, t))=\mu_0$ for $t\geq 0$.
\end{prop}

\begin{proof}
By Lemma 5.12, $1=\boldsymbol{\mu}_0(L^{2}(0,T; V))=\mathbb{P}(\omega\in \Omega\ |\ B^{\omega}\in L^{2}(0,T; V))$. For arbitrary $\varphi\in C_b(H^{1-\varepsilon})$, the function 
\begin{align*}
\tilde{\varphi}(B)=\int_{0}^{t}\varphi(B)ds
\end{align*}
belongs to $C_b(\mathcal{X}_T^{-\varepsilon})$ and 
\begin{align*}
\int_{\mathcal{X}_T^{-\varepsilon}}\tilde{\varphi}(B)\boldsymbol{\mu}_{\kappa}(dB)
=\mathbb{E}\tilde{\varphi}(B_{\kappa})
=\int_{0}^{t}\mathbb{E}{\varphi}(B_{\kappa})ds
=\int_{0}^{t}\int_{H^{1-\varepsilon}}{\varphi}(B){\mu}_{\kappa}(dB)ds.
\end{align*}
Since $\boldsymbol{\mu}_{\kappa}\to \boldsymbol{\mu}_{0}$ in $\mathcal{P}(\mathcal{X}_T^{-\varepsilon})$ and $\mu_k\to \mu_0$ in $\mathcal{P}(H^{1-\varepsilon})$, letting $\kappa\to0$ implies that 
\begin{align*}
\int_{\mathcal{X}_T^{-\varepsilon}}\tilde{\varphi}(B)\boldsymbol{\mu}_{0}(dB)
=\int_{0}^{t}\int_{H^{1-\varepsilon}}{\varphi}(B){\mu}_{0}(dB)ds.
\end{align*}
Since the limit $B$ in (5.19) satifies $\mathcal{D}(B)=\boldsymbol{\mu}_{0}$,
\begin{align*}
\int_{\mathcal{X}_T^{-\varepsilon}}\tilde{\varphi}(B)\boldsymbol{\mu}_{0}(dB)
=\mathbb{E}\tilde{\varphi}(B)
=\int_{0}^{t}\mathbb{E}{\varphi}(B)ds.
\end{align*}
By differentiating for $t>0$, 
\begin{align*}
\mathbb{E}{\varphi}(B)(t)
=\int_{H^{1-\varepsilon}}{\varphi}(B){\mu}_{0}(dB),\quad t\geq 0.
\end{align*}
Thus, $\mu_0$ is the law of $B(t)$.
\end{proof}

\begin{prop}
The following convergence holds as $\kappa\to 0$:
\begin{align}
\sqrt{\kappa}\zeta&\to 0\quad \textrm{in}\ L^{2}\left(\Omega; C([0,T]; H)\right),\\
\kappa\Delta B_\kappa&\to 0\quad \textrm{in}\ L^{2}\left(\Omega; L^{2}\left(0,T; H^{-1}\right)\right),\\
u_\kappa&\to 0\quad \textrm{in}\ L^{2}\left(\Omega; L^{2}\left(0,T; H^{\gamma }\right)\right),\\
u_\kappa\otimes B_\kappa,\ B_\kappa \otimes u_\kappa&\to 0\quad \textrm{in}\ L^{1}(\Omega\times \mathbb{T}^{d}\times (0,T)),\\
B_\kappa\otimes B_\kappa&\to B\otimes B\quad \textrm{in}\ L^{1}(\mathbb{T}^{d}\times (0,T)),\quad  \textrm{a.s.}
\end{align}
\end{prop}

\begin{proof}
The convergence (5.20) follows from Proposition 3.3. The convergences (5.21) and (5.22) follow from (5.6). By H\"older's inequality and (5.6),
\begin{align*}
\mathbb{E}||u_\kappa\otimes B_\kappa||_{L^{1}(\mathbb{T}^{d}\times (0,T)) }
\leq \left(\mathbb{E}||u_\kappa||_{L^{2}(0,T; H)}^{2}\right)^{1/2}\left(\mathbb{E}||B_\kappa||_{L^{2}(0,T; H)}^{2}\right)^{1/2}
\leq \left(\mathbb{E}||u_{\kappa}||_{L^{2}(0,T; \dot{H}^{\gamma  })}^{2}\right)^{1/2}\left(\mathbb{E}||B_\kappa||_{L^{2}(0,T; H^{1})}^{2}\right)^{1/2}\to 0.
\end{align*}
Thus, (5.23) holds. By (5.19), $B_\kappa\to B$ in $L^{2}(0,T; H)$ almost surely and  (5.24) holds.
\end{proof}

\begin{proof}[Proof of Theorem 1.1]
Since $B_k$ satisfies the equation (3.13) on $V^{*}$ for $t\geq 0$ almost surely, for arbitrary $\varphi\in H\cap C^{\infty}(\mathbb{T}^{d})$,  
\begin{align*}
(B_\kappa(t),\varphi)=(B_\kappa(0),\varphi)+\int_0^{t}(\kappa\Delta B_\kappa,\varphi)ds 
+\int_0^{t}(\nabla\cdot(B_\kappa\otimes u_\kappa-u_\kappa\otimes B_\kappa),\varphi)ds
+(\sqrt{\kappa}\zeta,\varphi).
\end{align*}
By the convergences (5.20)-(5.24) and $B_{\kappa}\to B$ in $C([0,T]; H^{-\varepsilon})$ a.s., letting $\kappa\to 0$ implies that 
\begin{align*}
(B(t),\varphi)=(B(0),\varphi),\quad t\geq 0,\ \textrm{a.s.}
\end{align*}
Thus, $B=B(x)$ is time independent. Since $B\in L^{2}(0,T; V)$ by Proposition 5.13, $B$ belongs to $V$ almost surely.

By multiplying $\varphi\in H\cap C^{\infty}(\mathbb{T}^{d})$ by the constitutive law $u_\kappa=K_{\gamma}(B_\kappa,B_\kappa)=(-\Delta)^{-\gamma}\Pi\nabla \cdot (B_\kappa\otimes B_\kappa)$ and integration by parts,
\begin{align*}
\int_{0}^{T}(u_\kappa,\varphi)dt=-\int_{0}^{T}\left(B_\kappa\otimes B_\kappa, \nabla (-\Delta)^{-\gamma}\varphi \right)dt.
\end{align*}
The left-hand side vanishes by (5.22) and the right-hand side converges by (5.24) almost surely. Since $B$ is time-independent and $B\in H^{1}$ almost surely, 
\begin{align*}
\left(B \cdot \nabla  B,  (-\Delta)^{-\gamma}\varphi \right)=0.
\end{align*}
By substituting $\varphi=(-\Delta)^{\gamma}\tilde{\varphi}$ for $\tilde{\varphi}\in H\cap C^{\infty}(\mathbb{T}^{d})$ into the above, $(B \cdot \nabla  B,\tilde{\varphi} )=0$. By Proposition 2.1, $B$ satisfies (1.2) for some $p\in W^{1,1}(\mathbb{T}^{d})$ almost surely.
\end{proof}

\begin{thm}
The measure $\mu_0$ satisfies the following:
\begin{align}
\int_{H} || \nabla B||_{H}^{2} \mu_{0}(dB)
&\leq \frac{\mathsf{C}_0}{2},\\
\int_{H}\exp\left(\rho ||B||_{H}^{2}\right)\mu_{0}(dB)
&\leq (\mathsf{C}_0+1)e^{\rho(\mathcal{C}_0+1)},
\quad 0<\rho\leq \frac{1}{2\sup_{j\geq 1}b_j^{2}}, \\
\int_{H}||B ||_{H}^{2}\mu_0 (dB)
&=\frac{\mathsf{C}_{-1}}{2},\quad d=2, \\
\int_{H}\left( \nabla \times  B, B   \right)_H\mu_0 (dB)
&=\frac{\mathcal{C}_{-\frac{1}{2}}}{2},\quad d=3.
\end{align}
\end{thm}

\begin{proof}
At the end of the proof of Lemma 5.11, we obtained 
\begin{align*}
\int_{\mathcal{X}^{-\varepsilon}_{T}}||\nabla B||_{L^{2}(0,T; H) }^{2}\boldsymbol{\mu}_{0}(d B)\leq \frac{\mathsf{C}_0}{2}T.
\end{align*}
Since $\boldsymbol{\mu}_{0}=\mathcal{D}(B)$ and $\mu_0=\mathcal{D}(B)$ for the limit $B$,  
\begin{align*}
\int_{\mathcal{X}^{-\varepsilon}_{T}}||\nabla B||_{L^{2}(0,T; H) }^{2}\boldsymbol{\mu}_{0}(d B)
=\mathbb{E}||\nabla B||_{L^{2}(0,T; H) }^{2}
=T\mathbb{E}||\nabla B||_{H}^{2}
=T\int_{H}||\nabla B||_{H}^{2}\mu_0(dB).
\end{align*}
Thus, (5.25) holds. Similarly, we obtain (5.26) from (5.2). 

We show the equality (5.28). The equality (5.27) follows a similar argument. By Parseval's identity for $B=\sum_{j=1}^{\infty}(B,e_j)_{H}e_j$ and $\nabla \times B=\sum_{j=1}^{\infty}\tau_j (B,e_j)_{H}e_j$ with $\tau_j^{2}=\lambda_j$,
\begin{align*}
(\nabla \times B,B)_{H}&=\sum_{j=1}^{\infty}\tau_j |(B,e_j)_{H}|^{2}
\leq \sum_{j=1}^{\infty}\sqrt{\lambda_j} |(B,e_j)_{H}|^{2}
=((-\Delta)^{\frac{1}{2}}B,B )_{H}
=||B||_{H^{\frac{1}{2}}}^{2}.
\end{align*}
Thus, $|(\nabla \times B,B)_H|\leq ||B||_{H^{1/2}}^{2}$ for $B\in V^{1/2}=H\cap H^{1/2}(\mathbb{T}^{3})$. By $B_{V^{\frac{1}{2}}}(R)$, we denote an open ball with radius $R>0$ in $V^{1/2}$. By the balance relation (5.3),
\begin{align*}
\int_{B_{V^{\frac{1}{2}}}(R)}(\nabla \times B,B)_{H}\mu_{\kappa}(dB)+\int_{B_{V^{\frac{1}{2}}}(R)^{c} }(\nabla \times B,B)_{H}\mu_{\kappa}(dB)=\frac{\mathcal{C}_{-\frac{1}{2}}}{2}.
\end{align*}
We estimate the second term on the left-hand side by 
\begin{align*}
\left|\int_{B_{V^{\frac{1}{2}}}(R)^{c} }(\nabla \times B,B)_{H}\mu_{\kappa}(dB)\right|
\leq \int_{B_{V^{\frac{1}{2}}}(R)^{c} }||B||_{H^{\frac{1}{2}}}^{2}\mu_{\kappa}(dB)
\leq \frac{1}{R}\int_{H }||B||_{H^{\frac{1}{2}}}^{3}\mu_{\kappa}(dB)=\frac{1}{R}\mathbb{E}||B_{\kappa}||_{H^{\frac{1}{2}}}^{3}.
\end{align*}
By the interpolation inequality (2.6) and H\"older's inequality,
\begin{align*}
\mathbb{E}||B_{\kappa}||_{H^{\frac{1}{2}}}^{3}\leq (\mathbb{E}||B_{\kappa}||_{H}^{6})^{\frac{1}{4}}(\mathbb{E}||B_{\kappa}||_{H^{1}}^{2})^{\frac{3}{4}}.
\end{align*}
By (5.1) and (5.2), the right-hand side is bounded by a constant independent of $\kappa$. Thus,
\begin{align*}
\frac{\mathcal{C}_{-\frac{1}{2}}}{2}-\frac{C}{R}\leq \int_{B_{V^{\frac{1}{2}}}(R)}(\nabla \times B,B)_{H}\mu_{\kappa}(d B)\leq \frac{\mathcal{C}_{-\frac{1}{2}}}{2}+\frac{C}{R},
\end{align*}
for some constant $C>0$. By integrating in time,
\begin{align*}
\int_{0}^{T}\int_{B_{V^{\frac{1}{2}}}(R)}(\nabla \times B,B)_{H}\mu_{\kappa}(dB)dt
=\int_{\Omega}\int_{0}^{T}1_{B_{V^{\frac{1}{2}}}(R)}(B_{\kappa})(\nabla \times B_{\kappa},B_{\kappa})_{H}\mathbb{P}(d\omega)dt.
\end{align*}
The integrand is bounded by $R^{2}$. Since $B_{\kappa}\to B$ in $L^{2}(0,T; H^{1-\varepsilon})$ almost surely, the dominated convergence theorem implies that 
\begin{align*}
\lim_{\kappa\to 0}\int_{\Omega}\int_{0}^{T}1_{B_{V^{\frac{1}{2}}}(R)}(B_{\kappa})(\nabla \times B_{\kappa},B_{\kappa})_{H}\mathbb{P}(d\omega)dt
=T\int_{B_{V^{\frac{1}{2}}}(R)}(\nabla \times B,B)_{H}\mu_0(d B).
\end{align*}
Thus, the limit measure $\mu_0$ satisfies 
\begin{align*}
\frac{\mathcal{C}_{-\frac{1}{2}}}{2}-\frac{C}{R}\leq \int_{B_{V^{\frac{1}{2}}}(R)}(\nabla \times B,B)_{H}\mu_0(d B)\leq \frac{\mathcal{C}_{-\frac{1}{2}}}{2}+\frac{C}{R}.
\end{align*}
The equality (5.28) follows by letting $R\to\infty$.
\end{proof}

\subsection{Other settings}

We remark on a general result with hyper-resistivity and the bounded domain case.

\subsubsection{Hyper-resistivity}

We can apply the fluctuation-dissipation method in \S 2-\S 5 also for a general system with hyper-resistivity:
\begin{equation}
\begin{aligned}
\partial_t B+u\cdot \nabla B-B\cdot \nabla u&=-\kappa (-\Delta)^{\alpha} B +\sqrt{\kappa}\partial_t \zeta,\\
\nabla p&=B\cdot \nabla B-(-\Delta)^{\gamma}u,\\
\nabla \cdot u=\nabla \cdot B&=0,
\end{aligned}
\end{equation}
for $\alpha\geq 1$ and $\gamma>d/2$. The path-wise global well-posedness result in \S 3 and the existence of invariant measures in \S 4 are extendable to a system with hyper-resistivity for $\alpha\geq 1$. By the It\^o formula, statistically stationary solutions to (5.29) satisfy the exponential estimate (5.2) and the following balance relations: 

\begin{align*}
\mathbb{E}\left(\kappa || B_{\kappa}||_{\dot{H}^{\alpha}}^{2}+\left\|u_{\kappa}\right\|_{\dot{H}^{\gamma  }}^{2} \right)
&=\kappa\frac{\mathsf{C}_0}{2},\\
\mathbb{E}\left((-\Delta)^{\alpha}\textrm{curl}^{-1}  B_{\kappa}, B_{\kappa}   \right)_H
&=\frac{\mathcal{C}_{-\frac{1}{2}}}{2},\quad d=3, \\
\mathbb{E}\left\|\textrm{curl}^{-1} B_{\kappa}   \right\|_{\dot{H}^{\alpha}}^{2}
&=\frac{\mathsf{C}_{-1}}{2},\quad d=2.
\end{align*}
By modifying the argument in this section, we can show the convergence of lifted invariant measures on $C([0,T]; H^{-\varepsilon})\cap L^{2}(0,T; H^{\alpha-\varepsilon})$ for $T>0$ and construct a random MHS equiblium on $H\cap H^{\alpha}(\mathbb{T}^{d})$ for $\alpha\geq 1$. 

\begin{thm}
Let $d\geq 2$, $\alpha\geq 1$, $\gamma> d/2$, and $\mathsf{C}_0>0$. The system (5.29) admits an invariant measure $\mu_{\kappa}$ on $H^{\alpha}(\mathbb{T}^{d})$. For $\varepsilon>0$, there exists a sequence of invariant measures $\{\mu_{\kappa}\}$ weakly converging to a measure $\mu_0$ on $H^{\alpha-\varepsilon}(\mathbb{T}^{d})$ as $\kappa\to 0$ such that there exist statistically stationary solutions to (5.29), $B_{\kappa}(t): (\Omega, \mathcal{F},\mathbb{P})\to (H, \mathcal{B}(H))$ with law $\mathcal{D}(B_{\kappa})=\mu_{\kappa}$, and a random variable $B: (\Omega, \mathcal{F},\mathbb{P})\to (H, \mathcal{B}(H))$ on some probability space $(\Omega, \mathcal{F},\mathbb{P})$ such that for $\varepsilon>0$,
\begin{align*}
B_{\kappa}(x,t)\to B(x)\quad \textrm{in}\ C([0,\infty); H^{-\varepsilon}(\mathbb{T}^{d}))\cap L^{2}_{\textrm{loc}}([0,\infty); H^{\alpha-\varepsilon}(\mathbb{T}^{d})),\quad a.s.\quad \textrm{as}\ \kappa\to 0.
\end{align*}
The limit $B\in H\cap H^{\alpha}(\mathbb{T}^{d})$ is an MHS equilibrium (1.2) with some pressure function almost surely and $\mathcal{D}(B)=\mu_0$. For $d=2$ and $d=3$, the measure $\mu_0$ satisfies the equalities: 
\begin{align*}
\mathbb{E}||\textrm{curl}^{-1}B ||_{\dot{H}^{\alpha}}^{2}
&=\int_{H}||\textrm{curl}^{-1}B ||_{\dot{H}^{\alpha}}^{2}\mu_0 (dB)=\frac{\mathsf{C}_{-1}}{2}, \\
\mathbb{E}\left((-\Delta)^{\alpha}\textrm{curl}^{-1}  B, B   \right)_H
&=\int_{H}\left( (-\Delta)^{\alpha}\textrm{curl}^{-1}  B, B   \right)_H\mu_0 (dB)
=\frac{\mathcal{C}_{-\frac{1}{2}}}{2}.
\end{align*}
\end{thm}

\subsubsection{The bounded domain case}

For general (possibly multiply-connected) bounded domains $\Omega\subset \mathbb{R}^{3}$, the work \cite{CP22} constructed MHS equilibria in $H^{\alpha}(\Omega)$ for $\alpha\geq 1$ which are not Beltrami fields by long time limits of the Voigt-MHD system. Standard boundary conditions for viscous and resistive MHD are the perfect conductivity condition and the Dirichlet boundary condition, 
\begin{align*}
(\nabla \times B)\times n=0,\quad B\cdot n=0,\quad  u=0\quad \textrm{on}\ \partial\Omega,  
\end{align*}
for the unit outward normal vector field $n$, e.g., \cite[2.2]{GLL}. 

The fluctuation-dissipation method is also available for bounded domains; see \cite[5.2.6]{Kuk12} for the complex Ginzburg--Landau equation. For simply-connected domains, we can construct a complete orthonormal basis $\{e_{j}\}$ (without explicit forms) on a Hilbert space by eigenfunctions of the rotation operator \cite{YG90}, and apply the same fluctuation-dissipation method of this paper for $\alpha=1$. On the other hand, for multiply-connected domains, the rotation operator is not a self-adjoint operator. Its spectra are point spectra and agree with the complex plane \cite{YG90}. 

%See also \cite{MV19} for the definition of magnetic helicity on multiply-connected domains.

\section{Absolute continuity of laws}

We show that the laws of energy, helicity, mean-square potential, and Casimir invariants under $\mu_0$ are absolutely continuous for Lebesgue measures. 

\subsection{The law of energy}

\begin{thm}
Let $d\geq 2$. Let $\mu_0$ be a measure as in Theorem 1.1. Assume that $b_j\neq 0$ for all $j\in \mathbb{N}$. Then, the law of energy $\mathcal{D}_{\mu_0}\mathscr{E}$ is absolutely continuous for the Lebesgue measure on $(0,\infty)$.
\end{thm}

We show absolute continuity of the law of $\tilde{\mathscr{E}} (B)=2{\mathscr{E}} (B)=||B||_{H}^{2}$.

\begin{prop}
The following holds for statistically stationary solutions $B_{\kappa}$ to (1.3) for $\gamma>d/2$ and $f\in C^{2}(\mathbb{R})$:
\begin{align}
\mathbb{E}\left(f'\left(||B_{\kappa}(0)||_{H}^{2}\right)\left(  \frac{\kappa}{2}\mathsf{C}_0-\kappa ||\nabla B_\kappa(0)||_{H}^{2}-||u_\kappa(0)||_{\dot{H}^{\gamma }}^{2}\right)+\kappa f''\left(||B_{\kappa}(0)||_{H}^{2}\right)\sum_{j=1}^{\infty}b_j^{2}|(B_\kappa(0),e_j)_{H}|^{2}   \right)=0.
\end{align}
\end{prop}

\begin{proof}
We apply the It\^o formula (B.5) for the functional $F[B]=f(\tilde{\mathscr{E}}(B))=f(||B||_{H}^{2})$ in Remark B.5. Since the law of $F[B_{\kappa}]$ is time independent, taking the mean to (B.5) implies $\mathbb{E}A(0)=0$ for  
\begin{align*}
\frac{1}{2}A(0)=\left(f'\left(||B_{\kappa}(0)||_{H}^{2}\right)\left(  \frac{\kappa}{2}\mathsf{C}_0-\kappa ||\nabla B_\kappa(0)||_{H}^{2}-||u_\kappa(0)||_{\dot{H}^{\gamma }}^{2}\right)+\kappa f''\left(||B_{\kappa}(0)||_{H}^{2}\right)\sum_{j=1}^{\infty}b_j^{2}|(B_\kappa(0),e_j)_{H}|^{2}   \right).
\end{align*}
\end{proof}

\begin{prop}
Set 
\begin{align}
(\lambda-\partial_x^{2})^{-1}f=\frac{1}{2\sqrt{\lambda}}e^{-\sqrt{\lambda}|x|}*f,\quad f\in C_{c}^{\infty}(\mathbb{R}),\ \lambda>0.
\end{align}
Then, $f_{\lambda}=(\lambda-\partial_x^{2})^{-1}f\in C^{\infty}(\mathbb{R})$ satisfies 
\begin{align}
\lambda ||f_{\lambda}||_{L^{\infty}}&\leq ||f||_{L^{\infty}}, \\
||f_{\lambda}'||_{L^{\infty}}&\leq ||f||_{L^{1}}, \\
\lim_{\lambda \to 0}||f_{\lambda}''+f||_{L^{\infty}}&=0.
\end{align}
\end{prop}

\begin{proof}
By an elementary computation, $G_{\lambda}=(2\sqrt{\lambda})^{-1}e^{-\sqrt{\lambda}|x|}$ satisfies $(\lambda -\partial_{x}^{2})G_{\lambda}=\delta_0$, $||G_{\lambda}||_{L^{\infty}}\leq (2\sqrt{\lambda})^{-1}$, $||G_{\lambda}' ||_{L^{\infty}}\leq 1/2$, and $||G_{\lambda} ||_{L^{1}}=1/\lambda$. Thus, $f_{\lambda}=G_{\lambda}*f$ satisfies $(\lambda -\partial_{x}^{2})f_{\lambda}=f$, (6.3) and, (6.4). By $||f_{\lambda}||_{L^{\infty}}\leq ||G_{\lambda}||_{L^{\infty}}||f||_{L^{1}}\leq (2\sqrt{\lambda})^{-1} ||f||_{L^{1}}$, (6.5) follows.
\end{proof}

\begin{lem}
The inequality 
\begin{align}
\mathbb{E}\sum_{j=1}^{\infty}1_{\Gamma}\left( ||B_{\kappa}(0)||_{H}^{2}  \right)b_j^{2}|(B_{\kappa}(0),e_j)|^{2} 
\leq \mathsf{C}_0 l_1(\Gamma),\quad \Gamma\in \mathcal{B}(0,\infty),
\end{align}
holds for statistically stationary solutions $B_{\kappa}$ to (1.3) for $\gamma>d/2$.
\end{lem}

\begin{proof}
We apply (6.1) for $f_{\lambda}=(\lambda-\partial_x^{2})^{-1}f$ and $f\in C^{2}_{c}(\mathbb{R})$. By (6.4) and (5.1), 
\begin{align*}
\kappa \left|\mathbb{E}\sum_{j=1}^{\infty}f_{\lambda}''\left(||B_{\kappa}(0)||_{H}^{2}\right)b_j^{2}|(B_{\kappa}(0),e_j)_{H}|^{2}\right|
\leq ||f||_{L^{1}}
\mathbb{E} \left|\frac{\kappa}{2}\mathsf{C}_0-\kappa ||\nabla B_{\kappa}(0)||_{H}^{2}-||u_{\kappa}(0)||_{\dot{H}^{\gamma }}^{2}\right| 
\leq ||f||_{L^{1}} \kappa \mathsf{C}_0.
\end{align*}
By (6.5) and the dominated convergence theorem, letting $\lambda\to0$ implies that 
\begin{align*}
\left|\mathbb{E}\sum_{j=1}^{\infty}f\left(||B_{\kappa}(0) ||_{H}^{2}\right)b_j^{2}|(B_{\kappa}(0),e_j)_{H}|^{2}\right| 
\leq  ||f||_{L^{1}} \mathsf{C}_0.
\end{align*}
We approximate the indicator function $1_{\Gamma}$ for a Borel set $\Gamma\in \mathcal{B}(0,\infty)$ by elements of $C^{2}_{c}(\mathbb{R})$ and obtain (6.6).
\end{proof}

\begin{prop}
\begin{align}
\sum_{j=N}^{\infty}|(B,e_j)_H|^{2}&\leq \frac{1}{\lambda_N}||\nabla B||_{H}^{2},\\
\sum_{j=1}^{N-1}|(B,e_j)_H|^{2}&\geq ||B||_{H}^{2}-\frac{1}{\lambda_N}||\nabla B||_{H}^{2},\quad B\in V,\ N\geq 1.
\end{align}
\end{prop}

\begin{proof}
By Parseval's identity for $B=\sum_{j=1}^{\infty}(B,e_j)_{H}e_j$ and $-\Delta B=\sum_{j=1}^{\infty}\lambda_{j} (B,e_j)_{H}e_j$,
\begin{align*}
||\nabla B||_{H}^{2}=(B,-\Delta B)_{H}=\sum_{j=1}^{\infty}\lambda_j |(B,e_j)_{H}|^{2}\geq\lambda_N \sum_{j=N}^{\infty} |(B,e_j)_{H}|^{2}.
\end{align*}
The second inequality (6.8) follows from (6.7).
\end{proof}

\begin{proof}[Proof of Theorem 6.1]
By the assumption $b_j\neq 0$ for all $j\in \mathbb{N}$, $\min_{1\leq j\leq N-1}b_j^{2}>0$ for any $N\geq 2$. For $\delta >0$, we take $N=N_{\delta}$ such that $\delta -1/(\lambda_{N}\delta)>0$. We set $I_{\delta}=\{B\in V\ |\ ||B||_{H}^{2}> \delta,\ ||\nabla B||_{H}^{2}< 1/\delta\ \}$ and apply (6.8) for $B\in I_{\delta}$ to estimate 
\begin{align*}
\sum_{j=1}^{\infty}b_j^{2}|(B,e_j)_H|^{2}
\geq \left(\min_{1\leq j\leq N-1}b_j^{2}\right)\sum_{j=1}^{N-1}|(B,e_j)_H|^{2} 
\geq \left(\min_{1\leq j\leq N-1}b_j^{2}\right)\left(||B||_{H}^{2}-\frac{1}{\lambda_{N}}|| \nabla B||_{H}^{2}  \right)
\geq \left(\min_{1\leq j\leq N-1}b_j^{2}\right)\left( \delta-\frac{1}{\lambda_{N}\delta}  \right).
\end{align*}
We apply (6.6) to estimate 
\begin{align*}
\mathsf{C}_0 l_1(\Gamma) \geq \mathbb{E}\sum_{j=1}^{\infty}1_{\Gamma}\left(||B_{\kappa}(0)||_{H}^{2}\right)b_j^{2}|(B_{\kappa}(0),e_j)_{H}|^{2} 
&\geq \int_{I_{\delta}}\sum_{j=1}^{\infty}1_{\Gamma}\left(||B||_{H}^{2}\right) b_j^{2}|(B,e_j)_{H}|^{2}   \mu_{\kappa}(dB) \\
&\geq \left(\min_{1\leq j\leq N-1}b_j^{2}\right)\left( \delta-\frac{1}{\lambda_{N}\delta}  \right)\mu_{\kappa}\left(\tilde{\mathscr{E}}^{-1}(\Gamma) \cap I_{\delta} \right).
\end{align*}
By (5.1) and Chebyshev's inequality, $\mu_{\kappa}\left(\ \left\{\ ||\nabla B||_{H}^{2}\geq 1/\delta\ \right\} \right)\leq \delta \mathsf{C}_0/2$. We thus obtain 
\begin{align*}
&\tilde{\mathscr{E}}^{-1}(\Gamma)\cap \left\{||B||_{H}^{2}> \delta\ \right\}\subset \left(\tilde{\mathscr{E}}^{-1}(\Gamma)\cap I_{\delta}\right)\cup \left( \left\{\ ||\nabla B||_{H}^{2}\geq 1/\delta\ \right\}\right),\\
&\mu_{\kappa}\left(\tilde{\mathscr{E}}^{-1}(\Gamma)\cap \left\{||B||_{H}^{2}> \delta\ \right\} \right)
\leq  \left(\min_{1\leq j\leq N-1}b_j^{2}\right)^{-1}\left( \delta-\frac{1}{\lambda_{N}\delta}  \right)^{-1}\mathsf{C}_0 l_1(\Gamma)+\frac{\delta}{2}\mathsf{C}_0.
\end{align*}
For an open set $\Gamma\subset (0,\infty)$, the set $\tilde{\mathscr{E}}^{-1}(\Gamma)\cap \left\{||B||_{H}^{2}> \delta\ \right\}$ is open in $H$. By $\mu_{\kappa}\to \mu_0$ in $\mathcal{P}(H^{1-\varepsilon})$ and Portmanteau theorem (Proposition 4.4), 
\begin{align*}
\mu_{0}\left(\tilde{\mathscr{E}}^{-1}(\Gamma)\cap \left\{||B||_{H}^{2}> \delta\ \right\} \right)
\leq \left(\min_{1\leq j\leq N-1}b_j^{2}\right)^{-1}\left( \delta-\frac{1}{\lambda_{N}\delta}  \right)^{-1}\mathsf{C}_0 l_1(\Gamma)+\frac{\delta}{2}\mathsf{C}_0.
\end{align*}
This inequality also holds for all Borel sets $\Gamma\in \mathcal{B} (0,\infty)$ by the approximation $l_1(\Gamma)=\inf\{\ l_1(G)\ |\ \Gamma\subset G,\ G:\textrm{open}\ \}$. If $l_1(\Gamma)=0$, $\lim_{\delta\to 0}\mu_{0}(\tilde{\mathscr{E}}^{-1}(\Gamma)\cap \{||B||_{H}^{2}> \delta\ \} )=0$. By $\tilde{\mathscr{E}}^{-1}(\Gamma)\backslash \{0\}=\cup_{n\geq 1}\tilde{\mathscr{E}}^{-1}(\Gamma)\cap \{||B||_{H}^{2}> 1/n\ \}$,
\begin{align*}
\mu_0(\tilde{\mathscr{E}}^{-1}(\Gamma)\backslash \{0\})=\mu_0\left(\bigcup_{n=1}^{\infty}\tilde{\mathscr{E}}^{-1}(\Gamma)\cap \left\{||B||_{H}^{2}> 1/n\ \right\}\right)=\lim_{n\to\infty}\mu_0(\tilde{\mathscr{E}}^{-1}\left(\Gamma)\cap \left\{||B||_{H}^{2}> 1/n\ \right\}\right)=0.
\end{align*}
By $0\notin \Gamma$, $\tilde{\mathscr{E}}^{-1}(\Gamma)\backslash \{0\}=\tilde{\mathscr{E}}^{-1}(\Gamma)$ and $\mu_0(\tilde{\mathscr{E}}^{-1}(\Gamma))=0$. Thus, $\mathcal{D}_{\mu_0}(\tilde{\mathscr{E}})$ is absolutely continuous with respect to the Lebesgue measure on $(0,\infty)$.
\end{proof}

\subsection{The law of helicity}

We show absolute continuity of the law of $\mathscr{H}$ by a similar argument to that for $\mathscr{E}$.

\begin{thm}
Let $d=3$. Let $\mu_0$ be a measure as in Theorem 1.1. Assume that $b_j\neq 0$ for all $j\in \mathbb{N}$. Then, the law of helicity $\mathcal{D}_{\mu_0}\mathscr{H}$ is absolutely continuous for the Lebesgue measure on $\mathbb{R}\backslash \{0\}$.
\end{thm}

\begin{prop}
The following holds for statistically stationary solutions $B_{\kappa}$ to (1.3) for $\gamma>d/2$ and $f\in C^{2}(\mathbb{R})$:
\begin{equation}
\begin{aligned}
&\mathbb{E}\Bigg(f'\left((\textrm{curl}^{-1}B_{\kappa}(0),B_{\kappa}(0) )_{H}\right)\left(\frac{1}{2}\mathcal{C}_{-\frac{1}{2}}-(B_{\kappa}(0),\nabla \times B_{\kappa}(0))_H\right) \\
&+f''\left((\textrm{curl}^{-1}B_{\kappa}(0),B_{\kappa}(0) )_{H}\right)\sum_{j=1}^{\infty}\frac{1}{\lambda_j}b_j^{2}|(B_{\kappa}(0),e_j)_{H}|^{2}\Bigg)=0.
\end{aligned}
\end{equation}
\end{prop}

\begin{proof}
We apply the It\^o formula (B.5) for the functional $F[B]=f(\mathscr{H}(B))=f\left((\textrm{curl}^{-1}B,B)_H \right)$ in Remark B.5. By taking the mean of (B.5), $\mathbb{E}A(0)=0$ for 
\begin{align*}
\frac{1}{2\kappa}A(0)
&= f'(\textrm{curl}^{-1}B_{\kappa}(0),B_{\kappa}(0))_H\left(- (B_{\kappa}(0),\nabla \times B_{\kappa}(0))_H+\frac{\mathcal{C}_{-\frac{1}{2}}}{2}\right) \\
&+f''(\textrm{curl}^{-1}B_{\kappa}(0),B_{\kappa}(0))_H \sum_{j=1}^{\infty}\frac{1}{\lambda_j}b_j^{2}|(B_{\kappa}(0),e_j)_H|^{2}.  
\end{align*}
\end{proof}

\begin{lem}
The inequality
\begin{align}
\mathbb{E}\sum_{j=1}^{\infty}1_{\Gamma}\left((\textrm{curl}^{-1}B_{\kappa}(0),B_{\kappa}(0) )_H\right)\frac{1}{\lambda_j}b_j^{2}|(B_{\kappa}(0),e_j)|^{2} 
\leq \frac{1}{2}l_1(\Gamma) \left( \mathsf{C}_0+  \left| \mathcal{C}_{-\frac{1}{2}}  \right|\right),\quad \Gamma\in \mathcal{B}(\mathbb{R}),
\end{align}
holds for statistically stationary solutions $B_{\kappa}$ to (1.3) for $\gamma>d/2$.
\end{lem}

\begin{proof}
We apply (6.9) for $f_{\lambda}=(\lambda-\partial_x^{2})^{-1}f$ and $f\in C^{2}_{c}(\mathbb{R})$. By using Poincar\'e inequality $||B||_H\leq ||\nabla B||_{H}$ and (5.1), $\mathbb{E}|(B,\nabla \times B)_H |\leq \mathsf{C}_0/2$. Thus, by (6.4), 
\begin{align*}
\left|\mathbb{E}\sum_{j=1}^{\infty}f_{\lambda}''\left((\textrm{curl}^{-1}B_{\kappa}(0),B_{\kappa}(0))_{H} \right)\frac{1}{\lambda_j}b_j^{2}|(B_{\kappa}(0),e_j)_{H}|^{2} \right|
\leq \frac{1}{2}||f||_{L^{1}} \left( \mathsf{C}_0+  \left| \mathcal{C}_{-\frac{1}{2}}  \right|\right).
\end{align*}
In a similar way to the proof of Lemma 6.4, by letting $\lambda\to0$ and approximating $1_{\Gamma}$ by $f\in C^{2}_{c}(\mathbb{R})$, we obtain (6.10).
\end{proof}

\begin{prop}
\begin{align}
||B||_{\dot{H}^{-\frac{1}{2}}}^{2}&=\sum_{j=1}^{\infty}\frac{1}{\sqrt{\lambda_j}}|(B,e_j)_H|^{2},\\
\sum_{j=1}^{N-1}\frac{1}{\lambda_j} |(B,e_j)_H|^{2}&
\geq \frac{1}{\sqrt{\lambda_N}}\left( ||B||_{H^{-\frac{1}{2}}}^{2}-\frac{1}{\sqrt{\lambda_N}}||B||_{H}^{2}\right),\quad B\in H,\ N\geq 1.
\end{align}
\end{prop}

\vspace{5pt}

\begin{proof}
By $(-\Delta)^{1/2}e_j=\sqrt{\lambda_j}e_j$ and Parseval's identity for $B=\sum_{j=1}^{\infty}(B,e_j)_{H}e_j$ and $(-\Delta)^{-1/2} B=\sum_{j=1}^{\infty}\lambda_j^{-1/2} (B,e_j)_{H}e_j$,
\begin{align*}
||B||_{\dot{H}^{-\frac{1}{2}}}^{2}=\left\|(-\Delta)^{-\frac{1}{4}}B\right\|_{H}^{2}=\left((-\Delta)^{-\frac{1}{2}}B,B\right)_H=\sum_{j=1}^{\infty}\frac{1}{\sqrt{\lambda_j}}|(B,e_j)_{H}|^{2}.
\end{align*}
By (6.11),
\begin{align*}
\sum_{j=1}^{N-1}\frac{1}{\lambda_j} |(B,e_j)_H|^{2}\geq \frac{1}{\sqrt{\lambda_N}}\sum_{j=1}^{N-1}\frac{1}{\sqrt{\lambda_j}} |(B,e_j)_H|^{2}
&=\frac{1}{\sqrt{\lambda_N}}\left(||B||_{\dot{H}^{-\frac{1}{2}}}^{2}-\sum_{j=N}^{\infty}\frac{1}{\sqrt{\lambda_j}} |(B,e_j)_H|^{2}\right)\\
&\geq \frac{1}{\sqrt{\lambda_N}}\left(||B||_{\dot{H}^{-\frac{1}{2}}}^{2}-\frac{1}{\sqrt{\lambda_N}} ||B||_{H}^{2}\right) 
\geq \frac{1}{\sqrt{\lambda_N}}\left(||B||_{H^{-\frac{1}{2}}}^{2}-\frac{1}{\sqrt{\lambda_N}} ||B||_{H}^{2}\right).
\end{align*}
\end{proof}

\begin{proof}[Proof of Theorem 6.6]
For $\delta >0$, we take $N=N_{\delta}$ such that $\delta -1/(\sqrt{\lambda_{N}}\delta)>0$. We apply (6.12) for $B\in I_{\delta}=\{B\in H\ |\ ||B||_{H^{-1/2}}^{2}>\delta,\  ||B||_{H}^{2}<1/\delta\  \}$ to esimate 
\begin{align*}
\sum_{j=1}^{\infty}\frac{1}{\lambda_j}b_j^{2}|(B,e_j)_{H}|^{2}
\geq \frac{1}{\sqrt{\lambda_N}}\left(\min_{1\leq j\leq N-1}{b_j^{2}}\right)\left(\delta-\frac{1}{\sqrt{\lambda_{N}}\delta} \right).
\end{align*}
For $\mathscr{H}(B)=(\textrm{curl}^{-1}B,B)_{H}$, we apply (6.10) to estimate
\begin{align*}
\frac{1}{2}l_1(\Gamma) \left( \mathsf{C}_0+  \left| \mathcal{C}_{-\frac{1}{2}}  \right|\right)
&\geq \mathbb{E}\sum_{j=1}^{\infty}1_{\Gamma}\left( \mathscr{H}(B_{\kappa}(0)) \right)\frac{1}{\lambda_j}b_j^{2}|(B_{\kappa}(0),e_j)_{H}|^{2} \\
&\geq\int_{I_{\delta}}1_{\Gamma}\left( \mathscr{H}(B) \right)\sum_{j=1}^{\infty} \frac{1}{\lambda_j} b_j^{2}|(B,e_j)_{H}|^{2}\mu_{\kappa}(dB)
\geq \frac{1}{\sqrt{\lambda_N}}\left(\min_{1\leq j\leq N-1}{b_j^{2}}\right)\left(\delta-\frac{1}{\sqrt{\lambda_{N}}\delta} \right)
\mu_{\kappa}\left(  \mathscr{H}^{-1}(\Gamma)\cap I_{\delta}   \right).
\end{align*}\\
By (5.1) and Chebyshev's inequality, $\mu_{\kappa}\left(\left\{ ||B||_{H}^{2}\geq 1/\delta\right\}\right)\leq \delta \mathsf{C}_{0}/2$. Thus, 
\begin{align*}
\mathscr{H}^{-1}(\Gamma)\cap \left\{\ ||B||_{H^{-\frac{1}{2}}}^{2}>\delta \right\} &\subset \left(\mathscr{H}^{-1}(\Gamma)\cap I_{\delta}\right) \cup \left\{ ||B||_{H}^{2}\geq 1/\delta\right\} \\
\mu_{\kappa}\left(\mathscr{H}^{-1}(\Gamma)\cap \left\{\ ||B||_{H^{-\frac{1}{2}}}^{2}>\delta \right\}\right)
&\leq \frac{\sqrt{\lambda_{N}}}{2}\left(\min_{1\leq j\leq N-1}{b_j^{2}}\right)^{-1}\left(\delta-\frac{1}{\sqrt{\lambda_{N}}\delta} \right)^{-1}
 \left( \mathsf{C}_0+  \left| \mathcal{C}_{-\frac{1}{2}}  \right|\right)l_1(\Gamma)+\frac{\delta}{2}\mathsf{C}_0.
\end{align*}
For an open set $\Gamma\subset \mathbb{R}$, the set $\mathscr{H}^{-1}(\Gamma)\cap \left\{\ ||B||_{H^{-\frac{1}{2}}}^{2}>\delta \right\} $ is open in $H^{-1/2}(\mathbb{T}^{3})\cap H$. By letting $\kappa\to0$ and Portmanteau theorem (Proposition 4.4), 
\begin{align*}
\mu_{0}\left(\mathscr{H}^{-1}(\Gamma)\cap \left\{\ ||B||_{H^{-\frac{1}{2}}}^{2}>\delta \right\}\right)
\leq \frac{\sqrt{\lambda_{N}}}{2}\left(\min_{1\leq j\leq N-1}{b_j^{2}}\right)^{-1}\left(\delta-\frac{1}{\lambda_{N}\delta} \right)^{-1}
 \left( \mathsf{C}_0+  \left| \mathcal{C}_{-\frac{1}{2}}  \right|\right)l_1(\Gamma)+\frac{\delta}{2}\mathsf{C}_0.
\end{align*}
This inequality holds for Borel sets $\Gamma\in \mathcal{B}(\mathbb{R})$ by the approximation $l_1(\Gamma)=\inf\{\ l_1(G)\ |\ \Gamma\subset G,\ G:\textrm{open}\ \}$. If $l_1(\Gamma)=0$, $\lim_{\delta\to 0}\mu_0(\mathscr{H}^{-1}(\Gamma)\cap \{\ ||B||_{H^{-\frac{1}{2}}}^{2}>\delta \})=0$. By $\mathscr{H}^{-1}(\Gamma)\backslash \{0\}=\cup_{n\geq 1}\mathscr{H}^{-1}(\Gamma)\cap \{\ ||B||_{H^{-\frac{1}{2}}}^{2}>1/n \}$, 
\begin{align*}
\mu_0\left(\mathscr{H}^{-1}(\Gamma)\backslash \{0\}\right)=\mu_0\left( \bigcup_{n=1}^{\infty}\mathscr{H}^{-1}(\Gamma)\cap \left\{\ ||B||_{H^{-\frac{1}{2}}}^{2}>\frac{1}{n} \right\}\right)
=\lim_{n\to\infty}\mu_0\left(\mathscr{H}^{-1}(\Gamma)\cap \left\{\ ||B||_{H^{-\frac{1}{2}}}^{2}>\frac{1}{n} \right\}\right)=0.
\end{align*}
If $0\notin \Gamma$, $\mathscr{H}^{-1}(\Gamma)\backslash \{0\}=\mathscr{H}^{-1}(\Gamma)$ and $\mu_0(\mathscr{H}^{-1}(\Gamma))=0$. Thus, $\mathcal{D}_{\mu_0}\mathscr{H}$ is absolutely continuous for the Lebesgue measure on $\mathbb{R}\backslash \{0\}$.
\end{proof}

\subsection{The law of mean-square potential}

For $d=2$, we show the absolute continuity of the law of $\mathscr{M}$ on $[0,\infty)$ (including zero) by using the identity for one-dimensional stationary processes (B.28). 

\begin{thm}
Let $d=2$. Let $\mu_0$ be a measure as in Theorem 1.1. Assume that $b_j\neq 0$ for all $j\in \mathbb{N}$. Then, the law of mean-square potential $\mathcal{D}_{\mu_0}\mathscr{M}$ is absolutely continuous for the Lebesgue measure on $[0,\infty)$.
\end{thm}

\begin{lem}
Assume that $b_{j}\neq 0$ for two integers $j$. Then, $\varsigma=\mathsf{C}_{-1}-\sup_{j\in \mathbb{N}}b_j^{2}/\lambda_j>0$ and 
\begin{align}
\mu_{0}( 0 \leq ||B||_{\dot{H}^{-1}}< \delta  )\leq \frac{2}{\varsigma}\sqrt{\mathsf{C}_{0}}\delta,\quad \delta>0.
\end{align}
\end{lem}

\begin{prop}
\begin{align}
||B||_{\dot{H}^{-1}}^{2}=\sum_{j=1}^{\infty}\frac{1}{\lambda_j}|(B,e_j)_{H}|^{2}=||\textrm{curl}^{-1}B||_{L^{2}}^{2},\quad B\in H.
\end{align}
\end{prop}

\begin{proof}
By Parseval's identity for $B=\sum_{j=1}^{\infty}(B,e_j)_{H}e_j$ and $(-\Delta)^{-1} B=\sum_{j=1}^{\infty}\lambda_{j}^{-1} (B,e_j)_{H}e_j$,
\begin{align*}
||B||_{\dot{H}^{-1}}^{2}=\left\|(-\Delta)^{-\frac{1}{2}} B\right\|_{H}^{2} =\left((-\Delta)^{-1} B,B\right)_H =\sum_{j=1}^{\infty}\frac{1}{\lambda_j}|(B,e_j)_H|^{2}.
\end{align*}
Since $\{\sqrt{\lambda_j}d_j\}=\{\sqrt{\lambda_j}\textrm{curl}^{-1}e_j\}$ is an orthonormal basis on $L^{2}_{\textrm{av}}$ (Theorem A.2), applying Parseval's identity for $\textrm{curl}^{-1}B=\sum_{j=1}^{\infty}(B,e_j)_{H}d_j$, 
\begin{align*}
|| \textrm{curl}^{-1}B ||_{L^{2}}^{2}=\sum_{j=1}^{\infty}\frac{1}{\lambda_j}|(B,e_j)_H|^{2}.
\end{align*}
\end{proof}

\begin{prop}
Let $f\in C^{2}(\mathbb{R})$ satisfy $f'\in W^{1,\infty}(\mathbb{R})$. Let $B_{\kappa}$ be a statistically stationary solution to (1.3) for $\gamma>d/2$. Let $\phi_{\kappa}=\textrm{curl}^{-1}\ B_{\kappa}$. Then, $f(||\phi_{\kappa}||_{L^{2}}^{2})$ is a one-dimensional stationary process 
\begin{align}
f\left(||\phi_{\kappa}(t)||_{L^{2}}^{2}\right)=f\left(||\phi_{\kappa}(0)||_{L^{2}}^{2}\right)+\int_{0}^{t}A(s)ds+\sum_{j=1}^{\infty}\int_{0}^{t}B_{j}(s)d\beta_j(s),\quad t \geq 0,\ \textrm{a.s.}
\end{align}
for the stationary processes 
\begin{equation}
\begin{aligned}
A(t)&=2\kappa \left(f'\left(||\phi_{\kappa}||_{L^{2}}^{2} \right)\left( -||\nabla \phi_{\kappa}||_{L^{2}}^{2}+\frac{\mathsf{C}_{-1}}{2} \right)+f''\left(||\phi_{\kappa}||_{L^{2}}^{2} \right)\sum_{j=1}^{\infty}b_j^{2}|(\phi_{\kappa},\textrm{curl}^{-1} e_j)_{L^{2}}|^{2}    \right),\\
B_j(t)&= 2\sqrt{\kappa}f'\left(||\phi_{\kappa}||_{L^{2}}^{2}\right) b_j(\phi_{\kappa},\textrm{curl}^{-1} e_j)_{L^{2}},
\end{aligned}
\end{equation}
satisfying 
\begin{align}
\mathbb{E}\int_{0}^{t}\left( |A(s)|+\sum_{j=1}^{\infty}|B_j(s)|^{2}\right)ds
=t \mathbb{E}\left(|A(0)|+\sum_{j=1}^{\infty}|B_j(0)|^{2} \right)
\leq \kappa \mathsf{C}_{-1}(1+\mathsf{C}_{-1}),\quad  t\geq 0.
\end{align}
\end{prop}

\begin{proof}
By $f'\in W^{1,\infty}(\mathbb{R})$, 
\begin{align*}
\frac{1}{\kappa}\left(|A(0)|+\sum_{j=1}^{\infty}|B_j(0)|^{2}\right)\lesssim \left\|\nabla \phi_{\kappa}(0) \right\|_{L^{2}}^{2}+\mathsf{C}_{-1}+\mathsf{C}_{-1} ||\phi_{\kappa}(0) ||_{L^{2}}^{2}.
\end{align*}
We apply Poincar\'e inequality and take the mean. By using (5.4), (6.17) follows. We apply the It\^o formula (B.5) for the functional $F[B]=f(\mathscr{M}(B))=f(||\textrm{curl}^{-1}B||_{L^{2}}^{2})$ in Remark B.5. By using the boundedness (6.17), we apply the It\^o formula (B.7) and obtain (6.15).   
\end{proof}

\begin{prop}
The following inequality holds:  
\begin{align}
\mathbb{E}\int_{\alpha}^{\beta}1_{[a,\infty)}(||\phi_{\kappa}(0)||_{L^{2}} )A(0)da\leq 0,
\end{align}
for $0<\alpha<\beta$ and 
\begin{align*}
\frac{2}{\kappa}A(0)=\frac{1}{||\phi_{\kappa}(0)||_{L^{2}}^{3} }\left(\mathsf{C}_{-1}||\phi_{\kappa}(0) ||_{L^{2}}^{2}-\sum_{j=1}^{\infty}b_j^{2} \left|(\phi_{\kappa}(0) ,\textrm{curl}^{-1}e_j)_{L^{2}} \right|^{2}   \right)-2\frac{||\nabla \phi_{\kappa}(0)||_{L^{2}}^{2}}{||\phi_{\kappa}(0)||_{L^{2}}}.
\end{align*}
\end{prop}

\begin{proof}
We choose a function $f\in C^{2}(\mathbb{R})$ such that $f(r)=\sqrt{r}$ for $r\geq \alpha$ and $f(r)=0$ for $r\leq 0$ and apply the identity (B.28) for the one-dimensional stationary process $y(t)=f(||\phi_{\kappa}||_{L^{2}}^{2})$ with $x(t)=A(t)$ and $\theta^{j}(t)=B_j(t)$. Then, (6.18) follows.
\end{proof}

\begin{prop}
\begin{align}
\mu_{\kappa}(0<||B||_{\dot{H}^{-1}}<\delta)=\mathbb{P}( 0<||\phi_{\kappa}(0)||_{L^{2}}< \delta  )\leq \frac{2}{\varsigma}\sqrt{\mathsf{C}_{0}}\delta,\quad \delta>0.
\end{align}
\end{prop}

\begin{proof}
The first equality follows from the norm identity (6.14). By (6.18), 
\begin{align*}
\mathbb{E}\int_{\alpha}^{\beta}1_{[a,\infty)}(||\phi_{\kappa}(0)||_{L^{2}} )\frac{1}{||\phi_{\kappa}(0)||_{L^{2}}^{3} } \left(\mathsf{C}_{-1}||\phi_{\kappa}(0) ||_{L^{2}}^{2}-\sum_{j=1}^{\infty}b_j^{2}|(\phi_{\kappa}(0),\textrm{curl}^{-1}\ e_j)_H|^{2} \right)  da\leq 2(\beta-\alpha)\mathbb{E}\left(\frac{||\nabla \phi_{\kappa}(0)||_{L^{2}}^{2} }{||\phi_{\kappa}(0)||_{L^{2}}}\right).
\end{align*}
By the interpolation inequality (2.6) and (5.1),
\begin{align*}
\mathbb{E}\left(\frac{||\nabla \phi_{\kappa}(0)||_{L^{2}}^{2} }{||\phi_{\kappa}(0)||_{L^{2}}}\right)\leq \mathbb{E}||\nabla^{2}\phi_{\kappa}(0)||_{L^{2}}\leq \sqrt{\mathsf{C}_{0}}.
\end{align*}
We obtain 
\begin{align*}
\mathbb{E}\int_{\alpha}^{\beta}1_{[a,\infty)}\left(||\phi_{\kappa}(0)||_{L^{2}}  \right)
\frac{1}{||\phi_{\kappa}(0)||_{L^{2}}^{3}}
 \left(\mathsf{C}_{-1}|| \phi_{\kappa}(0)||_{L^{2}}^{2}-\sum_{j=1}^{\infty}b_j^{2}|(\phi_{\kappa}(0),\textrm{curl}^{-1}\ e_j)_H|^{2}      \right)     
da\leq 2(\beta-\alpha)\sqrt{\mathsf{C}_{0}}.
\end{align*}
For $d_j=\textrm{curl}^{-1} e_j$, $\{\sqrt{\lambda_j}d_j\}$ is an orthonormal basis on $L^{2}_{\textrm{av}}(\mathbb{T}^{2})$ (Theorem A.2). Thus, 
\begin{align*}
\mathsf{C}_{-1}||\phi_{\kappa}(0)||_{L^{2}}^{2}-\sum_{j=1}^{\infty}b_j^{2} |(\phi_{\kappa}(0),\textrm{curl}^{-1}\ e_j)_H|^{2}
&=\mathsf{C}_{-1}||\phi||_{L^{2}}^{2}-\sum_{j=1}^{\infty}\frac{1}{\lambda_j}b_j^{2} |(\phi_{\kappa}(0), \sqrt{\lambda_j}d_j)_H|^{2} \\
&=\mathsf{C}_{-1}||\phi_{\kappa}(0)||_{L^{2}}^{2}-\left(\sup_{j\geq 1}\frac{1}{\lambda_j}b_j^{2}\right) ||\phi_{\kappa}(0)||_{L^{2}}^{2}\geq \varsigma ||\phi_{\kappa}(0)||_{L^{2}}^{2},
\end{align*}\\
and for $\delta>0$,  
\begin{align*}
2(\beta-\alpha) \sqrt{\mathsf{C}_{0}}
\geq 
\varsigma \mathbb{E}\int_{\alpha}^{\beta}1_{[a,\infty)}\left(||\phi_{\kappa}(0)||_{L^{2}}  \right)
\frac{1}{||\phi_{\kappa}(0)||_{L^{2}}} 
da
&\geq 
\varsigma \mathbb{E}\int_{\alpha}^{\beta}1_{(a,\delta]}\left(||\phi_{\kappa}(0)||_{L^{2}}   \right)
\frac{1}{||\phi_{\kappa}(0)||_{L^{2}} } 
da \\
&\geq 
\frac{\varsigma}{\delta}\int_{\alpha}^{\beta}
\mathbb{P}(a< ||\phi_{\kappa}(0)||_{L^{2}}  \leq \delta)da.
\end{align*}
By letting $\alpha\to 0$, and then dividing both sides by $\beta$ and letting $\beta\to 0$, we obtain (6.19).
\end{proof}

\begin{proof}[Proof of Lemma 6.11]
We show that $\mu_{\kappa}(B=0)=0$. We take $j_0\in \mathbb{N}$ such that $b_{j_0}\neq 0$. Then, $y(t)=(B_{\kappa},e_{j_0})_{H}$ is a one-dimensional stationary stochastic process (B.24) for $x(t)=(\kappa \Delta B_{\kappa}-u_{\kappa}\cdot \nabla B_{\kappa}+B_{\kappa}\cdot \nabla u_{\kappa},e_{j_0})_{H}$ and $\theta^{j}(t)=\delta_{j_0,j}\sqrt{\kappa}b_j$. By the identity for the stationary process (B.28),
\begin{align*}
\mathbb{E}\left(1_{[-\varepsilon,\varepsilon]}(y(0))\sum_{j=1}^{\infty}|\theta^{j}(0)|^{2} \right)
=-2\mathbb{E}\int_{-\varepsilon}^{\varepsilon}1_{[a,\infty)}(y(0))x(0)da.
\end{align*}
By $-\Delta d_j=\lambda_jd_j$ and $e_j=\nabla^{\perp}d_j$,
\begin{align*}
x(t)=(\nabla^{\perp}(\kappa \Delta \phi_{\kappa}-u_{\kappa}\cdot \nabla \phi_{\kappa}),e_{j_0} )_H
=-(\kappa \Delta \phi_{\kappa}-u_{\kappa}\cdot \nabla \phi_{\kappa},\nabla^{\perp}\cdot e_{j_0})_{L^{2}}
&=\lambda_{j_0} (\kappa \Delta \phi_{\kappa}-u_{\kappa}\cdot \nabla \phi_{\kappa}, d_{j_0})_{L^{2}} \\
&=-\lambda_{j_0} (\kappa \nabla \phi_{\kappa}-u_{\kappa} \phi_{\kappa}, \nabla d_{j_0})_{H}.
\end{align*}
By H\"older's inequality, $||u_{\kappa}\phi_{\kappa}||_{L^{2}}\leq ||u_{\kappa}||_{L^{4}}||\phi_{\kappa}||_{L^{4}}\lesssim ||u_{\kappa}||_{H^{1}}||\phi_{\kappa}||_{H^{1}}$. By (5.1), $\mathbb{E}||u_{\kappa}||_{H^{1}}||\phi_{\kappa}||_{H^{1}}\lesssim  \sqrt{\kappa}\sqrt{\mathsf{C}_{0}\mathsf{C}_{-1}}$ and 
\begin{align*}
\mathbb{E}|x(0)|\lesssim  \lambda_{j_0}\left(\kappa \sqrt{ \mathsf{C}_{-1}}+\sqrt{ \kappa \mathsf{C}_{0}\mathsf{C}_{-1}} \right).
\end{align*}
We obtain 
\begin{align*}
b_{j_0}^{2}\mathbb{P}(|y(0)|\leq \varepsilon)\lesssim \varepsilon \lambda_{j_0}\left(\sqrt{\mathsf{C}_{-1}}+\sqrt{\frac{\mathsf{C}_{0}\mathsf{C}_{-1}}{\kappa}} \right).
\end{align*}
Thus, $\mathbb{P}((B_{\kappa},e_{j_0})=0)=0$. By the norm identity (6.14), $\lambda_{j_0}^{-1} |(B_{\kappa},e_{j_0})_{H}|^{2}\leq ||\phi_{\kappa}||_{L^{2}}^{2}$ and $\mu_{\kappa}(B=0)=\mathbb{P}(\phi_{\kappa}=0)\leq \mathbb{P}((B_{\kappa},e_{j_0})_{H}=0)=0$. 

By (6.19), 
\begin{align}
\mu_{\kappa}(0\leq ||B||_{\dot{H}^{-1}}<\delta)=\mathbb{P}( 0\leq ||\phi_{\kappa}(0)||_{L^{2}}< \delta  )\leq \frac{2}{\varsigma}\sqrt{\mathsf{C}_{0}}\delta,\quad \delta>0.
\end{align}
Since $\mu_{\kappa}\to \mu_{0}$ in $\mathcal{P}(H^{-1})$, (6.13) follows from Portmanteau theorem (Proposition 4.4).
\end{proof}

\begin{proof}[Proof of Theorem 6.10]
The absolute continuity of the law $\mathcal{D}_{\mu_0}\mathscr{M}$ for the Lebesgue measure on $(0,\infty)$ follows from a similar argument as the proof of Theorem 6.1 and the It\^o formula for the functional $f(\mathscr{M}(B))$  for $f\in C^{2}(\mathbb{R})$. The inequality (6.13) implies the absolute continuity of the law $\mathcal{D}_{\mu_0}\mathscr{M}$ on $[0,\infty)$. 
\end{proof}

\subsection{The law of Casimir invariants}

We consider $f=(f_1,\cdots,f_n)\in C^{2}(\mathbb{R})$ such that $f_1',\cdots,f_n'$ are independent modulo constant. Namely, if there exist constants $a=(a_1,\cdots,a_n)\in \mathbb{R}^{n}$, $C\in \mathbb{R}$, $z_0\in \mathbb{R}$, and $\delta_0>0$ such that if 
\begin{align*}
h_a(z):=a\cdot f'(z)=a_1f'_1(z)+\cdots+a_nf'_n(z)=C,\quad z\in (z_0-\delta,z_0+\delta),
\end{align*}
then $a=0$ and $C=0$. The level sets $h^{-1}_a(C)$ are discreate for any $a\in \mathbb{R}^{n}$ if $f_1',\cdots,f_n'$ are independent modulo constant. 

\begin{thm}
Let $d=2$. Let $\mu_0$ be a measure as in Theorem 1.1. Let $f=(f_1,\cdots,f_n)\in C^{2}(\mathbb{R})$ be a function such that $f_1',\cdots,f_n'$ are independent modulo constant and $f'\in W^{1,\infty}(\mathbb{R})$. Set $\mathscr{C}(B)=(\mathscr{C}_k(B))_{1\leq k\leq n}$ by $\mathscr{C}_k(B)=\int_{\mathbb{T}^{2}}f_k(\textrm{curl}^{-1}B)dx$. Assume that $b_j\neq 0$ for all $j\in \mathbb{N}$. Then, the law $\mathcal{D}_{\mu_0}(\mathscr{C})$ is absolutely continuous for the $n$-dimensional Lebesgue measure in $\mathbb{R}^{n}$. 
\end{thm}

\begin{proof}[Proof of Theorem 1.3]
For a compact set $K\Subset V$ with a finite Hausdorff dimension $\textrm{dim}_H K<\infty$, we take an integer $n\geq 1$ such that $n>\textrm{dim}_H K$ and a function $f$ satisfying the assumptions in Theorem 6.16. The function $\mathscr{C}: H\to \mathbb{R}^{n}$ is Lipschitz continuous by the Lipschitz continuity of $f$ and 
\begin{align*}
\left|\mathscr{C}_k(B^{1})-\mathscr{C}_k(B^{2})\right|=\left|\int_{\mathbb{T}^{2}}\left(f_k\left(\textrm{curl}^{-1}B^{1}\right)-f_k\left(\textrm{curl}^{-1}B^{2}\right)\right)dx\right|\leq L\left\|\textrm{curl}^{-1}B^{1}-\textrm{curl}^{-1}B^{2}\right\|_{L^{1}}\lesssim \left\|B^{1}-B^{2}\right\|_{H}.
\end{align*}
Since the Hausdorff dimension is non-increasing by a Lipschitz map $\mathscr{C}: V\to \mathbb{R}^{n}$, e.g., \cite[Proposition 2.3]{Fal}, $l_n(\mathscr{C}(K))=0$. By the absolute continuity of the law $\mathcal{D}_{\mu_0}\mathscr{C}$ in Theorem 6.16, $\mu_0(K)=\mathcal{D}_{\mu_0}\mathscr{C}( \mathscr{C}(K) )=0$.
\end{proof}

In the rest of the paper, we show Theorem 6.16. We consider the essential range of a function $\phi: \mathbb{T}^{2}\to \mathbb{R}$,
\begin{align*}
\textrm{ess. im}\ \phi=\{y\in \mathbb{R}\ |\ l_2(\phi^{-1}(U) )>0,\ y\in U,\ U: \textrm{open}\ \}.
\end{align*}
For continuous functions $\phi\in C(\mathbb{T}^{2})$, the essential range agrees with the range of $\phi$, i.e., $\textrm{ess. im}\ \phi=\phi(\mathbb{T}^{2})$. In particular, $\phi\in C(\mathbb{T}^{2})$ must be a constant if $\phi(\mathbb{T}^{2})$ is a discrete set. We use a stronger property for Sobolev functions $\phi\in H^{r}(\mathbb{T}^{2})$ and $1/2<r\leq 1$ \cite[Lemma A.15.1]{Kuk12}: $\phi\in H^{r}(\mathbb{T}^{2})$ musy be a constant if $\textrm{ess. im}\ \phi$ is a discrete set. 

\begin{lem}
Set $\sigma=(\sigma^{kl})_{1\leq k,l\leq n}:L^{2}(\mathbb{T}^{2}) \to \mathbb{R}$ by 
\begin{align*}
\sigma^{kl}(g)=\sum_{j=1}^{\infty}b_j^{2}(g_k,d_j)_{L^{2}}(g_l,d_j)_{L^{2}},\quad g=(g_1,\cdots,g_n)\in L^{2}(\mathbb{T}^{2}).
\end{align*}
Then, $\textrm{det}\ \sigma(f'(\phi))$ is positive and continuous for $\phi\in H^{r}(\mathbb{T}^{2})\cap L^{2}_{\textrm{av}}(\mathbb{T}^{2})\backslash \{0\}$ and $1/2<r\leq 1$. 
\end{lem}

\begin{proof}
For $g\in L^{2}(\mathbb{T}^{2})$, we set the operator $T(g): L^{2}(\mathbb{T}^{2})\to \mathbb{R}^{n}$ by 
\begin{align*}
T(g)\xi=\left(\left(\sum_{j=1}^{\infty}\sqrt{\lambda_j}b_j (g_k,d_j)_{L^{2}}d_j ,\xi\right)_{L^{2}}\right)_{1\leq k\leq n}
=\left( \sum_{j=1}^{\infty}\sqrt{\lambda_j}b_j (g_k,d_j)_{L^{2}}\left(d_j ,\xi\right)_{L^{2}}\right)_{1\leq k\leq n},\quad \xi\in L^{2}(\mathbb{T}^{2}).
\end{align*}\\
By Parseval's identity and $(d_i,d_j)_{L^{2}}=\lambda_j^{-1}\delta_{i,j}$,
\begin{align*}
\left\|\sum_{j=1}^{\infty}\sqrt{\lambda_j}b_j(g_k,d_j)_{L^{2}}d_j\right\|_{L^{2}}^{2}
=\sum_{j=1}^{\infty}b_j^{2}|(g_k,d_j)_{L^{2}}|^{2}\leq \mathsf{C}_{-1}||g_k||_{L^{2}}^{2}.
\end{align*}
Thus, $T(g)$ is a bounded operator. We set the adjoint operator $T(g)^{*}: \mathbb{R}^{n}\to L^{2}(\mathbb{T}^{2})$ by 
\begin{align*}
T(g)\xi\cdot a&=\sum_{k=1}^{n}\sum_{j=1}^{\infty}\sqrt{\lambda_j}b_j(g_k,d_j)_{L^{2}}(d_j,\xi)_{L^{2}}a_{k} \\
&=\left(\xi,\sum_{k=1}^{n}a_k\sum_{j=1}^{\infty}\sqrt{\lambda_j}b_j(g_k,d_j)_{L^{2}}d_j \right)_{L^{2}}
=(\xi, T(g)^{*}a)_{L^{2}},\quad \xi\in L^{2}(\mathbb{T}^{2}),\ a\in \mathbb{R}^{n}.
\end{align*}
By an elementary computation, $T(g)T(g)^{*}=\sigma(g)$. The matrix $\sigma(g)$ is symmetric, and its quadratic form is non-negative by 
\begin{align*}
a\cdot \sigma(g)a=a\cdot T(g)T(g)^{*}a=|T(g)^{*}a|^{2}\geq 0,\quad a\in \mathbb{R}^{n}.
\end{align*}
The functions $\textrm{det}\ \sigma(\cdot)$: $L^{2}(\mathbb{T}^{2})\to \mathbb{R}$ and $f'(\cdot):L^{2}(\mathbb{T}^{2})\to L^{2}(\mathbb{T}^{2})$ are continuous. Thus, the composition $\textrm{det}\ \sigma( f'(\cdot) ):L^{2}(\mathbb{T}^{2})\to L^{2}(\mathbb{T}^{2})$ is also continuous. For $\phi\in H^{r}(\mathbb{T}^{2})\cap L^{2}_{\textrm{av}}(\mathbb{T}^{2})$ such that $T(f'(\phi))^{*}a=0$ for some $a\neq 0$,  
\begin{align*}
\sum_{j=1}^{\infty}\sqrt{\lambda_j}b_j(a\cdot f'(\phi), d_j)_{L^{2}}d_j=0.
\end{align*}
Since $b_j\neq 0$ for all $j\in \mathbb{N}$ and $\{\sqrt{\lambda_j}d_j\}_{j=1}^{\infty}$ is an orthonormal basis on $L^{2}_{\textrm{av}}(\mathbb{T}^{2})$ (see Thereom A.2), there exists a constant $C$ such that 
\begin{align*}
a\cdot f'(\phi(x))=C,\quad \textrm{a.e.}\ x\in \mathbb{T}^{2}.
\end{align*}
Thus, $\phi(x)\in h^{-1}_{a}(C)$ for almost every $x\in \mathbb{T}^{2}$ for the function $h_a(z)=a\cdot f'(z)$. Since the level set $h^{-1}_{a}(C)$ is discrete by the assumption of $f_1',\cdots, f_n'$, the essential range of $\phi\in H^{r}(\mathbb{T}^{2})$ is discrete. Thus, $\phi$ must be a constant. By the mean zero condition $\phi\in L^{2}_{\textrm{av}}(\mathbb{T}^{2})$, $\phi\equiv 0$. We conclude that $\textrm{det}\ \sigma(f'(\phi))$ is positive for $\phi\in H^{r}(\mathbb{T}^{2})\cap L^{2}_{\textrm{av}}(\mathbb{T}^{2})\backslash \{0\}$.
\end{proof}

\begin{prop}
The stochastic process $B(t)$ of (1.3) for $\gamma>d/2$ satisfies 
\begin{align}
\mathscr{C}_k(B(t))=\mathscr{C}_k(B(0))+\int_{0}^{t}A_k(s)ds+\sum_{j=1}^{\infty}\int_{0}^{t}B_{k,j}(s)d\beta_j(s),\quad t\geq 0,\ \textrm{a.s.},
\end{align}
for the processes
\begin{align*}
\frac{1}{\kappa}A_k(t)&= -(f_k''(\phi)\nabla \phi, \nabla \phi)_H+\frac{1}{2}\sum_{j=1}^{\infty}b_j^{2}\left( (\nabla^{\perp}(-\Delta)^{-1}f_{k}''(\phi))d_j,e_j\right)_{H},\\
B_{k,j}(t)&=\sqrt{\kappa}b_j (f_k'(\phi),d_j )_{L^{2}},
\end{align*}
where $\phi=\textrm{curl}^{-1}B$. Moreover, 
\begin{align}
\mathbb{E}\int_0^{t}\left(|A_k(s)|+\sum_{j=1}^{\infty}|B_{k,j}(s)|^{2} \right)ds\lesssim \mathbb{E}||\phi_0||_{L^{2}}^{2}+\kappa \mathsf{C}_{-1}t,\quad \textrm{for any}\ t\geq 0.
\end{align}
\end{prop}

\begin{proof}
By $e_{j}=\nabla^{\perp}d_j$ and integration by parts, $2\left( (\nabla^{\perp}(-\Delta)^{-1}f_{k}''(\phi))d_j,e_j\right)_{H}=(f_{k}''(\phi)d_j,d_j)_{L^{2}}$. By $f_k'\in W^{1,\infty}(\mathbb{R})$,
\begin{align}
|A_k(s)|+\sum_{j=1}^{\infty}|B_{k,j}(s)|^{2}\lesssim \kappa \left(||\nabla \phi||_{L^{2}}^{2}+\mathsf{C}_{-1} \right).
\end{align}
By the mean-square potential balance (4.4), (6.22) follows. We apply the It\^o formula (B.5) for the functional $F[B]=||f(\textrm{curl}^{-1}B)||_{L^{2}}^{2}$ in Remark B.6. By the condition (6.22), we apply the It\^o formula (B.7) and obtain (6.21). 
\end{proof}

\begin{proof}[Proof of Theorem 6.16]
By (6.21), the $n$-dimensional It\^o process $\mathscr{C}(B)=(\mathscr{C}_k(B))_{1\leq k\leq n}$ satisfies  
\begin{align*}
\mathscr{C}(B(t))=\mathscr{C}(B(0))+\int_{0}^{t}A(s)ds+\sum_{j=1}^{\infty}\int_{0}^{t}B_{j}(s)d\beta_j(s),\quad t\geq 0,\ \textrm{a.s.}
\end{align*}
The processes $A=(A_{k})_{1\leq k\leq n}$ and $B_j=(B_{k,j})_{1\leq k\leq n}$ satisfy the condition (B.25) by (6.22). For a statistically stationary solution $B_{\kappa}(t)$, $\mathscr{C}(B_{\kappa}(t))$ and  $A(t)$ are stationary processes. We set the matrix 
\begin{align*}
\sum_{j=1}^{\infty}B_{k,j}B_{l,j}
=\kappa \sum_{j=1}^{\infty}b_j^{2}(f_k'(\phi),d_j)_{L^{2}}(f_l'(\phi),d_j)_{L^{2}}
=\kappa \sigma^{kl}(f'(\phi)),
\end{align*}
and apply Krylov's estimate (B.29) to obtain 
\begin{align*}
\kappa \mathbb{E}\int_{0}^{1}\left(\textrm{det}\ \sigma (f'(\phi_{\kappa}(s) ))\right)^{\frac{1}{n}}1_{\Gamma}(\mathscr{C}(B_{\kappa}(s)))ds \leq C_n l_n(\Gamma)^{\frac{1}{n}}\mathbb{E}|A(0)|,\quad \Gamma\in \mathcal{B}(\mathbb{R}^{n}).
\end{align*}
By (6.23) and (5.4), $\mathbb{E}|A(0)|\lesssim \kappa \mathsf{C}_{-1}$ and 
\begin{align}
\mathbb{E}\int_{0}^{1}\left(\textrm{det}\ \sigma\left(f'(\phi_{\kappa}(s) )\right)^{\frac{1}{n}}1_{\Gamma}\left(\mathscr{C}\left(B_{\kappa}(s) \right) \right) \right)ds
\lesssim  \mathsf{C}_{-1} l_n(\Gamma)^{\frac{1}{n}}.
\end{align}
The function $\textrm{det}\ \sigma(f'(\phi))$ is positive and continuous for $\phi \in H^{r}\cap L^{2}_{\textrm{av}}(\mathbb{T}^{2})\backslash \{0\}$ and $1/2<r<1$ by Lemma 6.17. Thus, for $\varepsilon>0$ and a compact set $I_{\varepsilon}=\{\phi\in H^{1}\cap L^{2}_{\textrm{av}}(\mathbb{T}^{2})\ |\ ||\phi||_{L^{2}}\geq \varepsilon,\ ||\nabla \phi||_{L^{2}}\leq 1/\varepsilon\  \}\Subset H^{r}\cap L^{2}_{\textrm{av}}(\mathbb{T}^{2})\backslash \{0\}$,  
\begin{align*}
\inf_{\phi\in I_{\varepsilon}}\textrm{det}\ \sigma(f'(\phi))=:c_{\varepsilon}>0.
\end{align*}
It follows from (6.24) that 
\begin{align*}
\mu_{\kappa}(\mathscr{C}(B)\in \Gamma,\ \phi\in I_{\varepsilon})
&=\int_{\{\mathscr{C}(B)\in \Gamma,\ \phi\in I_{\varepsilon}\}}\mu_{\kappa}(dB) \\
&=\int_{0}^{1}\int_{\{\mathscr{C}(B_{\kappa}(s) )\in \Gamma,\ \phi_{\kappa}(s)\in I_{\varepsilon}\}}\mathbb{P}(d\omega)ds \\
&\leq c_{\varepsilon}^{-\frac{1}{n}}\int_{0}^{1}\int_{\{\mathscr{C}(B_{\kappa}(s) )\in \Gamma,\ \phi_{\kappa}(s)\in I_{\varepsilon}\}} \left(\textrm{det}\ \sigma\left(f'(\phi^{\omega}_{\kappa}(s)) \right) \right)^{\frac{1}{n}}  \mathbb{P}(d\omega)ds\\
& \leq c_{\varepsilon}^{-\frac{1}{n}}\int_{0}^{1}\int_{\Omega} (\textrm{det}\ \sigma(f'(\phi_{\kappa}(s))))^{\frac{1}{n}} 1_{\Gamma}\left(\mathscr{C}\left(B_{\kappa}(s) \right) \right) \mathbb{P}(d\omega)ds  \\
&=c_{\varepsilon}^{-\frac{1}{n}}\mathbb{E}\int_{0}^{1}\left(\textrm{det}\ \sigma\left(f'(\phi_{\kappa}(s) )\right)^{\frac{1}{n}}1_{\Gamma}\left(\mathscr{C}\left(B_{\kappa}(s) \right) \right) \right)ds
\lesssim c_{\varepsilon}^{-\frac{1}{n}}\mathsf{C}_{-1}l_n(\Gamma)^{\frac{1}{n}}.
\end{align*}
By (5.4) and (6.19),
\begin{align*}
\mu_{\kappa}(\mathscr{C}(B)\in \Gamma,\ \phi\in I_{\varepsilon}^{c})
\leq \mu_{\kappa}(||\nabla \phi||_{L^{2}}>1/\varepsilon )
+\mu_{\kappa}( ||\phi||_{L^{2}}< \varepsilon) 
\leq \frac{\varepsilon^{2}}{2}\mathsf{C}_{-1}+\frac{2}{\varsigma}\sqrt{\mathsf{C}_{0}}\varepsilon.
\end{align*}
We obtain
\begin{align*}
\mu_{\kappa}(\mathscr{C}(B)\in \Gamma)
=\mu_{\kappa}(\mathscr{C}(B)\in \Gamma,\ \phi \in I_{\varepsilon})
+\mu_{\kappa}(\mathscr{C}(B)\in \Gamma,\ \phi \in I_{\varepsilon}^{c}) 
\lesssim c_{\varepsilon}^{-\frac{1}{n}}\mathsf{C}_{-1}l_n(\Gamma)^{\frac{1}{n}}
+\frac{\varepsilon^{2}}{2}\mathsf{C}_{-1}+\frac{2}{\varsigma}\sqrt{\mathsf{C}_{0}}\varepsilon.
\end{align*}\\
For an open set $G\subset \mathbb{R}^{n}$, $\mathscr{C}^{-1}(G)\subset H$ is open by Lipschitz continuity of $\mathscr{C}: H\to \mathbb{R}^{n}$. By $\mu_{\kappa}\to \mu_0$ in $\mathcal{P}(H)$ and Portmanteau theorem (Proposition 4.4), the above inequality also holds for the measure $\mu_0$ and open sets $G\subset \mathbb{R}^{n}$. Since $l_n(\Gamma)=\inf\{\ l_n(G)\ |\ \Gamma\subset G,\ G: \textrm{open}\ \}$ for $\Gamma\in \mathcal{B}(\mathbb{R}^{n})$, 
\begin{align*}
\mu_{0}(\mathscr{C}(B)\in \Gamma)
\lesssim c_{\varepsilon}^{-\frac{1}{n}}\mathsf{C}_{-1}l_n(\Gamma)^{\frac{1}{n}}
+\frac{\varepsilon^{2}}{2}\mathsf{C}_{-1}+\frac{2}{\varsigma}\sqrt{\mathsf{C}_{0}}\varepsilon.
\end{align*}
Thus, the law $\mathcal{D}_{\mu_0}(\mathscr{C})$ is absolutely continuous for the $n$-dimensional Lebesgue measure in $\mathbb{R}^{n}$. 
\end{proof}

\begin{rems}
(i) For $d\geq 4$, $\textrm{spt}\ \mu_0$ is an uncountable set unless $\mu_0$ is a Dirac mass $\delta_0$. In fact, 
\begin{align*}
\mathcal{D}_{\mu_0}(\mathscr{E})(\Gamma)=\mu_0(\mathscr{E}^{-1}(\Gamma))=\int_{H}1_{\Gamma}(\mathscr{E}(B))\mu_0(d B),\quad \Gamma\in \mathcal{B}(0,\infty).
\end{align*}
If $\textrm{spt}\ \mu_0=\{B_j\}_{j=1}^{\infty}\subset  V$, $\mu_0=\sum_{j=1}^{\infty}c_j\delta_{B_j}$ for some $\{c_j\}_{j=1}^{\infty}\subset (0,\infty)$ satisfying $\sum_{j=1}^{\infty}c_j=1$. For $\Gamma=\{\mathscr{E}(B_j)\}_{j=1}^{\infty}\subset (0,\infty)$, the left-hand side is zero by the absolute continuity of the law $\mathcal{D}_{\mu_0} (\mathscr{E})$ and $l_1(\Gamma)=0$. On the other hand, the right-hand side is one. 

\noindent
(ii) For $d=3$ and $\mathcal{C}_{-1/2}\neq 0$, $\textrm{spt}\ \mu_0$ is an uncountable set. Moreover, the set $\mathscr{H}(\textrm{spt}\ \mu_0)$ is uncountable. In fact, 
\begin{align*}
\mathcal{D}_{\mu_0}(\mathscr{H})(\Gamma)=
\mu_0(\mathscr{H}^{-1}(\Gamma) )
=\int_{H}1_{\Gamma}(\mathscr{H}(B))\mu_0(d B),\quad \Gamma\in \mathcal{B}(\mathbb{R}).
\end{align*}
If $\Gamma=\mathscr{H}(\textrm{spt}\ \mu_0)\backslash \{0\}$ is countable, $l_1(\Gamma)=0$ and the left-hand side is zero by the absolute continuity of the law $\mathcal{D}_{\mu_0} (\mathscr{H})$. On the other hand, the right-hand side is one. 
\end{rems}

%

%%%%%%%%%%%%%%%%%%%Appendix

\appendix

\section{Orthonormal systems}

We construct a complete orthonormal system on real-valued $H$ by eigenfunctions of the Stokes operator for $d\geq 2$, cf. \cite[2.1.5]{Kuk12}, \cite[4.6]{BV22}. For $d=3$, we build such eigenfunctions from eigenfunctions of the rotation operator.

\subsection{Eigenfunctions of the Stokes operator}

\begin{thm}
Let $d\geq 2$. There exists a complete orthonormal system $\{e_j\}_{j=1}^{\infty}$ on $H$ consisting of the eigenfunctions of the Stokes operator with eigenvalues $0<\lambda_1\leq \lambda_2\leq \cdots\leq \lambda_j\to \infty$.
\end{thm}

\begin{proof}
We first consider a complex-valued $H$. Let $a_k^{1},\cdots, a_{k}^{d-1}\in \mathbb{R}^{d}$ 
be an orthonormal basis on the orthogonal complement of the one-dimensional subspace spanned by $k\in \mathbb{Z}^{d}_{0}$ in $\mathbb{R}^{d}$. By the Fourier expansion of $B\in H$,
\begin{align*}
B=\sum_{k\in \mathbb{Z}^{d}_0}\hat{B}_ke^{ik\cdot x},\quad \hat{B}_k\cdot k=0.
\end{align*}
By $\hat{B}_k=\sum_{l=1}^{d-1}(\hat{B}_k\cdot a_k^{l})a_k^{l}$,
\begin{align*}
B=\sum_{k\in \mathbb{Z}^{d}_0}\sum_{l=1}^{d-1}(\hat{B}_k\cdot a_k^{l})a_k^{l} e^{ik\cdot x}=\sum_{k\in \mathbb{Z}^{d}_0}\sum_{l=1}^{d-1} \left(B, \frac{1}{(2\pi)^{\frac{d}{2}}}a^{l}_{k}e^{ik\cdot x} \right)_{H}\frac{1}{(2\pi)^{\frac{d}{2}}} a^{l}_{k}e^{ik\cdot x}.
\end{align*}
Thus, the vector fields 
\begin{align*}
\left\{\frac{1}{(2\pi)^{\frac{d}{2}}} a^{l}_{k}e^{ik\cdot x}\ \middle|\ k\in \mathbb{Z}_0^{d},\ 1\leq l\leq d-1\  \right\}
\end{align*}
are eigenfunctions of the Stokes operator with the eigenvalue $|k|^{2}$ and make a complete orthonormal system on the complex-valued $H$.

For a real-valued $H$, we take a symmetric subset $K\subset \mathbb{Z}^{d}_{0}$ such that $K\cup -K=\mathbb{Z}^{d}_{0}$. By an elemenrary caluculation, $\overline{\hat{B}_k}=\hat{B}_{-k}$ for $B\in H$ and 
\begin{align*}
B
&=2\sum_{k\in K}\textrm{Re} \left(\hat{B}_ke^{ik\cdot x}\right) \\
&=\sum_{k\in K}\sum_{l=1}^{d-1}\left(B, \frac{1}{2^{\frac{d-1}{2}}\pi^{\frac{d}{2}}}a^{l}_{k}\cos(k\cdot x) \right)_{H}\frac{1}{2^{\frac{d-1}{2}}\pi^{\frac{d}{2}}}a^{l}_{k}\cos(k\cdot x)
+\sum_{k\in -K}\sum_{l=1}^{d-1}\left(B, \frac{1}{2^{\frac{d-1}{2}}\pi^{\frac{d}{2}}}a^{l}_{k}\sin(k\cdot x) \right)_{H}\frac{1}{2^{\frac{d-1}{2}}\pi^{\frac{d}{2}}}a^{l}_{k}\sin(k\cdot x).
\end{align*}
Thus, the vector fields 
\begin{align*}
\left\{\frac{1}{2^{\frac{d-1}{2}}\pi^{\frac{d}{2}}}a^{l}_{k}\cos(k\cdot x)\ \middle|\ k\in K,\ 1\leq l\leq d-1\ \right\}\cup
\left\{\frac{1}{2^{\frac{d-1}{2}}\pi^{\frac{d}{2}}}a^{l}_{k}\sin(k\cdot x)\ \middle|\ k\in -K,\ 1\leq l\leq d-1\ \right\}
\end{align*}
make a complete orthonormal system on the real-valued $H$.
\end{proof}

\begin{thm}
For $d=2$, $\{\sqrt{\lambda_j} \textrm{curl}^{-1}e_j \}$ is a complete orthonormal system on $L^{2}_{\textrm{av}}(\mathbb{T}^{2})$.
\end{thm}

\begin{proof}
By $-\Delta e_j=\lambda_j e_j$ and $e_j=\nabla^{\perp}d_j$ for $d_j=\textrm{curl}^{-1}e_j$,  
\begin{align*}
(d_i,d_j)_{L^{2}}=\frac{1}{\lambda_j}(-\Delta d_i,d_j)_{L^{2}}=\frac{1}{\lambda_j}(e_i,e_j)_{H}
=\frac{1}{\lambda_j} \delta_{i,j}.
\end{align*}
For an arbitrary $\phi\in L^{2}_{\textrm{av}}(\mathbb{T}^{2})$, 
\begin{align*}
\nabla^{\perp}\phi=\sum_{j=1}^{\infty}(\nabla^{\perp}\phi, e_j)_{H}e_j=\sum_{j=1}^{\infty}(\phi,-\nabla^{\perp}\cdot e_j)_{L^{2}}e_j
=\nabla^{\perp}\sum_{j=1}^{\infty}(\phi,\sqrt{\lambda_j}d_j)_{L^{2}}\sqrt{\lambda_j} d_j.
\end{align*}
By the average-zero conditions for $\phi$ and $d_j$,
\begin{align*}
\phi=\sum_{j=1}^{\infty}(\phi,\sqrt{\lambda_j}d_j)_{L^{2}}\sqrt{\lambda_j} d_j.
\end{align*}
Thus, $\{\sqrt{\lambda_j} d_j \}$ is a complete orthonormal system on $L^{2}_{\textrm{av}}(\mathbb{T}^{2})$.
\end{proof}

\subsection{Eigenfunctions of the rotation operator: Beltrami waves}

We construct a complete orthonormal system on real-valued $H$ consisting of eigenfunctions of the rotation operator by using complex plane Beltrami waves \cite[Proposition 3.1]{DeS13}, \cite[Proposition 5.5]{BV19b}.

\begin{lem}[Complex plane Beltrami waves]
Let $a_k^{1}$, $a_{k}^{2}\in \mathbb{R}^{3}$ be an orthonormal basis on the orthogonal complement of the one-dimensional subspace spanned by $k\in \mathbb{Z}^{3}_{0}$ in $\mathbb{R}^{3}$ such that $a_{k}^{l}=a_{-k}^{l}$ for $l=1,2$. For $\xi=k/|k|$, the vectors 
\begin{align}
b^{l}_{k}=\frac{1}{\sqrt{2}}\left( a^{l}_{k}+i\xi\times  a^{l}_k\right), \quad l=1,2,
\end{align}
satisfy 
\begin{align}
|b^{l}_{k}|=1,\quad b_k^{l}\cdot \xi=0, \quad i\xi \times b^{l}_{k}=b^{l}_{k},\quad \overline{b^{l}_{k}}=b^{l}_{-k}.
\end{align}
The vector fields $b^{l}_{k}e^{ik\cdot x}$ and $b^{l}_{k}e^{-ik\cdot x}$ are eigenfunctions of the rotation operator with the eigenvalues $|k|$ and $-|k|$, respectively.
\end{lem}

\begin{thm}
Let $d=3$. There exists a complete orthonormal system $\{e_j\}_{j=1}^{\infty}$ on $H$ consisting of the eigenfunctions of the rotation operator with eigenvalues $\tau_j\in \mathbb{R}$ such that $\tau_j^{2}=\lambda_j$.
\end{thm}

\begin{proof}
We first consider a complex-valued $H$. We fix $l=1$ or $2$. By (A.2$)_4$ and (A.2$)_3$, $b_{k}^{l}\cdot b_{-k}^{l}=1$ and $b_{k}^{l}\cdot b_{k}^{l}=0$. Thus, 
\begin{align*}
(b^{l}_{k}e^{ik\cdot x}, b^{l}_{k'}e^{ik'\cdot x})_{H}
&=b^{l}_{k}\cdot \overline{b^{l}_{k'}} (2\pi)^{3}\delta_{k,k'}
=b^{l}_{k}\cdot b^{l}_{-k'}(2\pi)^{3}\delta_{k,k'}
=(2\pi)^{3}\delta_{k,k'},\\
(b^{l}_{k}e^{ik\cdot x}, b^{l}_{k'}e^{-ik'\cdot x})_{H} &=b^{l}_{k}\cdot \overline{b^{l}_{k'}} (2\pi)^{3}\delta_{k,-k'}=b^{l}_{k}\cdot b^{l}_{-k'}(2\pi)^{3}\delta_{k,-k'}=0.
\end{align*}
For $B\in H$,
\begin{align*}
\sum_{k\in \mathbb{Z}^{3}_{0}}(B,b_k^{l}e^{ik\cdot x} )_{H}b_k^{l}e^{ik\cdot x}
+\sum_{k\in \mathbb{Z}^{3}_{0}}(B,b_k^{l}e^{-ik\cdot x} )_{H}b_k^{l}e^{-ik\cdot x}
=\sum_{k\in \mathbb{Z}^{3}_{0}}\left((B,b_k^{l}e^{ik\cdot x} )_{H}b_k^{l}+(B,b_{-k}^{l}e^{ik\cdot x} )_{H}b_{-k}^{l}  \right)e^{ik\cdot x}.
\end{align*}
By using (A.1) and (A.2),
\begin{align*}
(B,b_k^{l}e^{ik\cdot x} )_{H}b_k^{l}+(B,b_{-k}^{l}e^{ik\cdot x} )_{H}b_{-k}^{l}
&=\frac{1}{\sqrt{2}}\left( (B, b^{1}_{k}e^{ik\cdot x})_{H}(a_k^{l}+i\xi\times a^{l}_{k})
+(B, b^{1}_{-k}e^{ik\cdot x})_{H}(a_k^{l}-i\xi\times a^{l}_{k})    \right) \\
&=(B, a_k^{l}e^{ik\cdot x})_{H}a_k^{l}+(B, i\xi\times a^{l}_{k}e^{ik\cdot x})_{H}i\xi\times a^{l}_{k}.
\end{align*}
The last term is $(B, a^{2}_{k})_{H}a^{2}_{k}$ for $l=1$ and $(B, a^{1}_{k})_{H}a^{1}_{k}$ for $l=2$. Thus, 
\begin{align*}
\sum_{k\in \mathbb{Z}^{3}_{0}}(B,b_k^{l}e^{ik\cdot x} )_{H}b_k^{l}e^{ik\cdot x}
+\sum_{k\in \mathbb{Z}^{3}_{0}}(B,b_k^{l}e^{-ik\cdot x} )_{H}b_k^{l}e^{-ik\cdot x}
&=\sum_{k\in \mathbb{Z}^{3}_{0}}(B, a_k^{1}e^{ik\cdot x})_{H}a_k^{1}e^{ik\cdot x}
+\sum_{k\in \mathbb{Z}^{3}_{0}}(B, a_k^{2}e^{ik\cdot x})_{H}a_k^{2}e^{ik\cdot x} \\
&=(2\pi)^{3}B(x),
\end{align*}
and the vector fields  
\begin{align*}
\left\{\frac{1}{(2\pi)^{\frac{3}{2}}}b_k^{l}e^{\pm ik\cdot x} \ \middle|\ k\in \mathbb{Z}_0^{3}\ \right\},
\end{align*}
are a complete orthonormal system on the complex-valued $H$.

We consider a real-valued $H$. We set the real-valued eigenfunctions of the rotation operator with the eigenvalue $|k|$ by 
\begin{align*}
\sqrt{2}b_{k}e^{ik\cdot x}
&=(a_k+i\xi\times a_k)(\cos{k\cdot x}+i\sin{k\cdot x}) \\
&=a_k\cos{k\cdot x}-\xi\times a_k\sin{k\cdot x}+i(a_k\sin{k\cdot x}+\xi\times a_k\cos{k\cdot x}) \\
&=p_k+iq_k,
\end{align*}
where the integer $l$ is suppressed. For $B\in H$, 
\begin{align*}
2(B,b_{k}e^{ik\cdot x})_Hb_{k}e^{ik\cdot x}
&=(B,p_k+iq_k)_{H}(p_k+iq_k) \\
&=(B,p_k)_Hp_k+(B,q_k)_Hq_k+i\left((B,p_k)_Hq_k-(B,q_k)_Hp_k \right).
\end{align*}
We take a symmetric subset $K\subset \mathbb{Z}^{3}_{0}$ such that $\mathbb{Z}^{3}_{0}=K\cup -K$. By $p_k=p_{-k}$ and $q_k=-q_{-k}$,
\begin{align*}
\textrm{Re}\sum_{k\in \mathbb{Z}^{3}_{0}}(B,b_ke^{ik\cdot x} )_{H}b_ke^{ik\cdot x}
=\frac{1}{2}\sum_{k\in \mathbb{Z}^{3}_{0}}(B,p_k)_Hp_k+(B,q_k)_Hq_k
=\sum_{k\in K}(B,p_k)_Hp_k+(B,q_k)_Hq_k.
\end{align*}
Similarly,  
\begin{align*}
\textrm{Re}\sum_{k\in \mathbb{Z}^{3}_{0}}(B,b_ke^{-ik\cdot x} )_{H}b_ke^{-ik\cdot x}
=\sum_{k\in K}(B,r_k)_Hr_k+(B,s_k)_Hs_k,\quad \sqrt{2}b_{k}e^{-ik\cdot x}=r_k-is_k.
\end{align*}
We thus obtain 
\begin{align*}
(2\pi)^{3}B(x)=\sum_{k\in K}(B,p_k)_Hp_k+(B,q_k)_Hq_k+(B,r_k)_Hr_k+(B,s_k)_Hs_k,
\end{align*}
for 
\begin{align*}
p_k&=a_k\cos{k\cdot x}-\xi\times a_k\sin{k\cdot x},\\
q_k&=a_k\sin{k\cdot x}+\xi\times a_k\cos{k\cdot x}, \\
r_k&=a_k\cos{k\cdot x}+\xi\times a_k\sin{k\cdot x}, \\
s_k&=a_k\sin{k\cdot x}-\xi\times a_k\cos{k\cdot x}.
\end{align*}
By $p_k\cdot p_k=1$, $p_k\cdot q_k=0$, $p_k\cdot r_k=\cos(2k\cdot x)$, and $p_k\cdot s_k=-\sin(2k\cdot x)$, $(p_k,p_k)_H=(2\pi)^{3}$ and $(p_k,q_k)_H=(p_k,r_k)_H=(p_k,s_k)_H=0$. The inner product of $p_k$ on $H$ with functions $p_{k'}$, $q_{k'}$, $r_{k'}$, $s_{k'}$ are zero for $k\neq k'$. Thus, 
\begin{align*}
(p_k,p_{k'})_H=(2\pi)^{3}\delta_{k,k'},\quad (p_k,q_{k'})_H=(p_k,r_{k'})_H=(p_k,s_{k'})_H=0,\quad k,k'\in \mathbb{Z}^{3}_{0}.
\end{align*}
Similarly, $q_k$, $r_k$, and $s_k$ are orthogonal to other vector fields on $H$, respectively. Thus, the normalized eigenfunctions 
\begin{align*}
\frac{1}{(2\pi)^{\frac{3}{2}}} \left\{ p_k, q_k, r_k, s_k\ \middle|\ k\in K  \right\},
\end{align*}
are complete orthonormal system on the real-valued $H$ (The functions $p_k$ and $q_k$ (resp. $r_k$ and $s_k$) are eigenfunctions of the rotation operator with the eigenvalue $|k|$ (resp. $-|k|$)).  
\end{proof}

\begin{rem}
The least positive (resp. largest negative) eigenvalue of the rotation operator is one (resp. minus one), and their multiplicities are six. Indeed, we choose a symmetric subset $K$ by  
\begin{align*}
K=\{(k_1,k_2,k_3)\in \mathbb{Z}^{3}_{0}\ |\ k_3>0\}\cup \{(k_1,k_2,0)\in \mathbb{Z}^{3}_{0}\ |\ k_2>0\}\cup \{(k_1,0,0)\in \mathbb{Z}^{3}_{0}\ |\ k_1>0\}.
\end{align*}
The vectors $k\in K$ satisfying $|k|=1$ are $k=e_1$, $e_2$, $e_3$. Its eigenfunctions are superpositions of the following:
\begin{align*}
p_{e_1}&=(0,\cos x_1,-\sin x_1), \quad
q_{e_1}=(0,\sin x_1,\cos x_1),\\
p_{e_2}&=(-\sin x_2, 0,\cos x_2),\quad
q_{e_2}=(\cos x_2,0,\sin x_2),\\
p_{e_3}&=(\cos x_3, -\sin x_3, 0),\quad
q_{e_3}=(\sin x_3,\cos x_3,0).
\end{align*}
The functions $-q_{e_i}$ are $\pi/2$ translation of $p_{e_i}$, i.e., $p_{e_i}(x_i+\pi/2)=-q_{e_i}(x_i)$ for $i=1,2,3$. The ABC flow is spanned by $\{q_{e_i}\}_{i=1}^{3}$, i.e., $\mathsf{A}q_{e_3}+\mathsf{B}q_{e_1}+\mathsf{C}q_{e_2}$ for $\mathsf{A},\mathsf{B},\mathsf{C}\in \mathbb{R}$. Eigenfunctions for the largest negative eigenvalue are superpositions of $r_{e_i}, s_{e_i}$ for $i=1,2,3$.   
\end{rem}

\section{It\^o formulas}

We provide It\^o formulas for $\mathscr{E}$, $\mathscr{H}$, $\mathscr{M}$, and $\mathscr{C}$, and state an identity for one-dimensional stationary stochastic processes and Krylov's estimate for $n$-dimensional stationary stochastic processes \cite[A.7, A.9]{Kuk12}. 

\subsection{It\^o formulas for It\^o processes in $H$}

We denote the space of all bounded linear operators from $H$ to a Banach space $X$ by $\mathcal{L}(H; X)$. We denote the Fr\'{e}chet derivative of the functional $F: H\to \mathbb{R}$ by $F': H\to \mathcal{L}(H; \mathbb{R})=H$ and the second derivative of the functional by $F'': H\to \mathcal{L}(H; H)=\mathcal{L}(H)$. By $C^{2}(H; \mathbb{R})$, we denote the space of all twice continuously differentiable functionals on $H$. 
\begin{thm}[It\^o formulas for It\^o processes in $H$]
Let $F\in C^{2}(H; \mathbb{R})$. Let $\{u_t\}$ be an It\^o process in $H$ with constant diffusion. Then, 
\begin{align}
F[u_{t}]=F[u_0]+\int_{0}^{t}A(s)ds+\sum_{j=1}^{\infty}\int_{0}^{t}B_j(s)d\beta_j(s),\quad t\geq 0,
\end{align}
with the constants 
\begin{equation}
\begin{aligned}
A(t)&=\left(F'[u_t],f_t\right)_H+\frac{1}{2}\sum_{j=1}^{\infty} \left(F''[u_t]g_j,g_j\right)_H,\\
B_j(t)&=(F'[u_t],g_j)_{H}.
\end{aligned}
\end{equation}
\end{thm}

\subsection{It\^o formulas for It\^o processes in $V^{*}$}

We consider functionals $F\in C^{2}(H; \mathbb{R})$ satisfying the following conditions:

\noindent
(i) There exists a positive continuous function $K\in C[0,\infty)$ such that 
\begin{align}
|(F'[u],v)_{H}|\leq K(||u||_{H})||u||_{V}||v||_{V^{*}},\quad u\in V,\ v\in V^{*}.
\end{align}
(ii) The operator $F': H\to H$ maps a strongly convergent sequence in $V$ to a weakly convergent sequence in $V$. Namely, for any sequence $\{u_k\}\subset V$ such that $u_k\to u$ in $V$ and $v\in V^{*}$,
\begin{align}
(F'[u_k],v)_H\to (F'[u],v)_H.
\end{align}

\begin{thm}[It\^o formulas for It\^o processes in $V^{*}$]
Let $F\in C^{2}(H; \mathbb{R})$ satisfy the conditions (B.3) and (B.4). Let $\{u_t\}$ be an It\^o process in $V^{*}$ with constant diffusion which belongs to $C([0,\infty); H)\cap L^{2}_{\textrm{loc}}([0,\infty); V) $ almost surely. Let $A(s)$ and $B_j(s)$ be as in (B.2). Then, 
\begin{align}
F[u_{t\wedge\tau_n}]=F[u_0]+\int_{0}^{t\wedge\tau_n}A(s)ds+M_{t\wedge\tau_n},\quad t\geq 0,
\end{align}
for $M_{t\wedge\tau_n}=\sum_{j=1}^{\infty}\int_{0}^{t\wedge\tau_n}B_j(s)d\beta_j(s)$ and the stopping time $\tau_n=\inf\{t\geq 0\ |\ ||u_t||_{H}>n\ \}$, $n\in \mathbb{N}$. Assume additionally that  
\begin{align}
\sum_{j=1}^{\infty}\mathbb{E}\int_{0}^{t}|B_j(s)|^{2}ds<\infty,\quad \textrm{for any}\ t\geq 0.
\end{align}
Then, $M_t$ is a square-integrable martingale with almost surely continuous trajectory and  
\begin{align}
F[u_t]=F[u_0]+\int_{0}^{t}A(s)ds+M_t,\quad t\geq 0,\ \textrm{a.s.}
\end{align}
\end{thm}

\subsection{Applications of the It\^o formula}

We apply the It\^o formula (B.5) for the following functionals: 
\begin{align}
F_1[u]&=\tilde{\mathscr{E}}(u)=||u||_{H}^{2},\\
F_2[u]&=\exp \left(\delta \tilde{\mathscr{E}}(u) \right)=\exp \left(\delta ||u||_{H}^{2} \right),\quad \delta>0,\\
F_3[u]&=\mathscr{H}(u)=(\textrm{curl}^{-1}u,u)_{H},\quad d=3,\\
F_4[u]&=\mathscr{M}(u)=||\textrm{curl}^{-1} u||_{L^{2}}^{2},\quad d=2.
\end{align}

\begin{prop}
The functionals $F_i$ ($1\leq i\leq 4$) belong to $C^{2}(H; \mathbb{R})$ and satisfy the conditions (B.3) and (B.4) with their derivatives 
\begin{align}
F'_1[u]&=2u,\quad F''_1[u]=2,\\
F_2'[u]&=2\delta F_1[u]u,\quad F_2''[u]v=2\delta F_2[u](v+2\delta (u,v)_H u ),\quad v\in H,\\
F'_3[u]&=2\textrm{curl}^{-1} u,\quad F''_3[u]=2\textrm{curl}^{-1},\\
F_4'[u]&=2(-\Delta)^{-1}u,\quad F_4''[u]=2(-\Delta)^{-1}.
\end{align}
\end{prop}

\begin{proof}
We show that $F_i\in C^{2}(H; \mathbb{R})$ ($1\leq i\leq 4$) and (B.12)-(B.15). It is not difficult to check the conditions (B.3) and (B.4) by using (B.12)-(B.15). 

The functional $F_1[u]=||u||_{H}^{2}$ is continuous and locally bounded on $H$. For $u,v\in H$ and $\varepsilon>0$,
\begin{align*}
\frac{F_1[u+\varepsilon v]-F_1[u] }{\varepsilon}=2(u,v)_{H}+\varepsilon ||v||_{H}^{2}.
\end{align*}
Thus G\^{a}teaux derivative $D_GF_1: H\to H^{*}$ exists and $<D_GF_1(u),v>=2(u,v)_{H}$ with the pairing $<\cdot,\cdot>$ for $H^{*}$ and $H$. By identifying $H^{*}$ and $H$, $D_GF_1[u]=2u$. Since $D_GF_1$ is continuous and locally bounded in $H$, Fr\'{e}chet derivative $D_GF_1=F_1'$ exists and $F_1\in C^{1}(H; \mathbb{R})$. Similarly, for $w\in H$ we have
\begin{align*}
\left<\frac{F_1'[u+\varepsilon v]-F_1'[u] }{\varepsilon},w   \right>=2(v,w)_{H}.
\end{align*}
By identifying $H^{*}$ and $H$,
\begin{align*}
\frac{F_1'[u+\varepsilon v]-F_1'[u] }{\varepsilon}=2v.
\end{align*}
By letting $\varepsilon\to 0$, $D_GF_1'[u]v=2v$ and $D_GF_1': H\to \mathcal{L}(H)$ is constant. Thus, the Fr\'{e}chet derivative $D_GF_1'=F_1''$ exists and $F_1\in C^{2}(H; \mathbb{R})$. We showed (B.12). Similarly, $F_2 \in C^{2}(H; \mathbb{R})$ and (B.13) holds.

The functional $F_3$ is continuous and locally bounded on $H$ by the Poincar\'{e} inequality $||\Psi||_{H}\lesssim ||\nabla \Psi||_{H}\lesssim ||u||_{H}$ for $\Psi=\textrm{curl}^{-1}u$ and $u\in H$. For $\Phi=\textrm{curl}^{-1}v$, $v\in H$, and $\varepsilon>0$, 
\begin{align*}
\frac{F_3[u+\varepsilon v]-F_3[u] }{\varepsilon}=2(\Psi,v)_{H}+\varepsilon(\Phi,v)_{H}.
\end{align*}
By letting $\varepsilon\to 0$, $<D_G F_3[u], v>=2(\Psi,v)_{H}$. By identifying $H^{*}$ and $H$, $D_GF_3[u]=2\Psi=2\textrm{curl}^{-1} u$. Since $D_GF_3$ is continuous and locally bounded in $H$, Fr\'{e}chet derivative $D_GF_3=F_3'$ exists and $F_3\in C^{1}(H; \mathbb{R})$. Similarly, for $w\in H$ we have
\begin{align*}
\left<\frac{F_3'[u+\varepsilon v]-F_3'[u] }{\varepsilon},w   \right>=2(\Phi,w)_{H}.
\end{align*}
By identifying $H^{*}$ and $H$,
\begin{align*}
\frac{F_3'[u+\varepsilon v]-F_3'[u] }{\varepsilon}=2\Phi.
\end{align*}
By letting $\varepsilon\to 0$, $D_GF_3'[u]v=2\Phi=2\textrm{curl}^{-1} v$. Thus $D_GF_3': H\to \mathcal{L}(H)$ is bounded and continuous and the Fr\'{e}chet derivative $D_GF_3'=F_3''$ exists and $F_3\in C^{2}(H; \mathbb{R})$. We showed (B.14). Similarly, $F_4\in C^{2}(H; \mathbb{R})$ and (B.15) holds. 
\end{proof}

\begin{lem}
The following holds for It\^o processes in $V^{*}$ with constant diffusion $\{u_t\}$ which belong to $C([0,\infty); H)\cap L^{2}_{\textrm{loc}}([0,\infty); V) $ almost surely:
\begin{align}
\tilde{\mathscr{E}}(u)(t\wedge \tau_n)
=\tilde{\mathscr{E}}(u)(0)+2\int_{0}^{t\wedge \tau_n}(u_s,f_s)_Hds
+t\wedge \tau_n\sum_{j=1}^{\infty}||g_j||_{H}^{2}
+2\sum_{j=1}^{\infty}\int_{0}^{t\wedge \tau_n}(u_s,g_j)_Hd\beta_j(s),
\end{align}
\begin{equation}
\begin{aligned}
\exp\left(\delta \tilde{\mathscr{E}}(u) \right)(t\wedge \tau_n)
=&\exp\left(\delta \tilde{\mathscr{E}}(u) \right)(0)+2\delta \int_{0}^{t\wedge \tau_n}\exp\left(\delta \tilde{\mathscr{E}}(u) \right)(u_s,f_s)ds\\
&+\delta \int_{0}^{t\wedge \tau_n} \exp\left(\delta \tilde{\mathscr{E}}(u) \right) \sum_{j=1}^{\infty}\left(||g_j||_{H}^{2}+2\delta |(u_s,g_j)_H|^{2} \right)ds  \\
&+2\delta \sum_{j=1}^{\infty}\int_{0}^{t\wedge \tau_n} \exp\left(\delta \tilde{\mathscr{E}}(u) \right) (u_s,g_j)_{H}d\beta_j(s),
\end{aligned}
\end{equation}
\begin{equation}
\begin{aligned}
\mathscr{H}(u)(t\wedge \tau_n)
&=\mathscr{H}(u)(0)+2\int_{0}^{t\wedge \tau_n}(u_s,\textrm{curl}^{-1}f_s)_Hds 
+t\wedge \tau_n\sum_{j=1}^{\infty}\left(\textrm{curl}^{-1} g_j,g_j\right)_{H} \\
&+2\sum_{j=1}^{\infty}\int_{0}^{t\wedge \tau_n}\left(u_s,\textrm{curl}^{-1}g_j\right)_Hd\beta_j(s),
\end{aligned}
\end{equation}
\begin{equation}
\begin{aligned}
\mathscr{M}(u)(t\wedge \tau_n)
&=\mathscr{M}(u)(0)+2\int_{0}^{t\wedge \tau_n}\left(\textrm{curl}^{-1} u_s,\textrm{curl}^{-1} f_s\right)_Hds 
+t\wedge \tau_n\sum_{j=1}^{\infty}\left\|(-\Delta)^{-\frac{1}{2}}g_j\right\|_{H}^{2}\\
&+2\sum_{j=1}^{\infty}\int_{0}^{t\wedge \tau_n}\left(\textrm{curl}^{-1} u_s,\textrm{curl}^{-1}\  g_j\right)_Hd\beta_j(s).
\end{aligned}
\end{equation}
\end{lem}

\begin{proof}
We substitute (B.12)-(B.15) into (B.2) and obtain (B.16)-(B.19) from (B.5).
\end{proof}

\begin{rem}
For more general functionals,     
\begin{align}
\tilde{F}_1[u]&=f(\tilde{\mathscr{E}}(u))=f(||u||_{H}^{2}), \\
\tilde{F}_3[u]&=f(\mathscr{H}(u))=f\left(  (\textrm{curl}^{-1}u,u)_H   \right),\\
\tilde{F}_4[u]&=f(\mathscr{M}(u))=f\left(||\textrm{curl}^{-1}u||_{L^{2}}^{2}\right),
\end{align}
for $f\in C^{2}(\mathbb{R})$, $\tilde{F}_i\in C^{2}(H; \mathbb{R})$ for $i=1,3,4$ and 
\begin{align*}
\tilde{F}_1'[u]&=2f'(\tilde{\mathscr{E}}(u) )u,\quad \tilde{F}_1''[u]v=4f''(\tilde{\mathscr{E}}(u) )(u,v)_{H}u+2f'(\tilde{\mathscr{E}}(u) )v,\quad v\in H, \\
\tilde{F}_3'[u]&=2f'\left( \mathscr{H}(u) \right)\textrm{curl}^{-1}u,\quad \tilde{F}_3''[u]v=4f''\left( \mathscr{H}(u) \right)(\textrm{curl}^{-1}u,v)_{H}\textrm{curl}^{-1}u+2f'\left( \mathscr{H}(u) \right)\textrm{curl}^{-1}v,\quad v\in H,  \\
\tilde{F}_4'[u]&=2f'(\mathscr{M}(u))(-\Delta)^{-1}u,\quad 
\tilde{F}_4''[u]v=4f''\left(\mathscr{M}(u)\right)\left((-\Delta)^{-1}u,v\right)_{L^{2}}(-\Delta)^{-1}u+2f'\left(\mathscr{M}(u)\right)(-\Delta)^{-1}v,\quad v\in H.
\end{align*}
The It\^o formula (B.5) also holds for (B.20)-(B.22) with the following functions:  
\begin{align*}
A^{1}(s)&=2f'\left(\tilde{\mathscr{E}}(u)\right) \left( \left(u, f_s  \right)_H+\frac{1}{2}\sum_{j=1}^{\infty}\left\|g_j\right\|_{H}^{2} \right)+2f''\left(\tilde{\mathscr{E}}(u)\right)\sum_{j=1}^{\infty} \left|\left( u,  g_j\right)	_{H}\right|^{2}    ,\\
B_j^{1}(s)&= 2 f'\left(\tilde{\mathscr{E}}(u)\right) \left(u, g_j\right)_H,\\
A^{3}(s)&=2 f'\left( \mathscr{H}(u) \right)
\left( \left(u, \textrm{curl}^{-1} f_s  \right)_H +\frac{1}{2}\sum_{j=1}^{\infty}\left( \textrm{curl}^{-1} g_j,  g_j\right)_{H}  \right)
+2f''\left( \mathscr{H}(u) \right)\sum_{j=1}^{\infty}|(u,  \textrm{curl}^{-1} g_j)_H|^{2},\\
B_j^{3}(s)&= 2 f'\left( \mathscr{H}(u) \right) \left(u, \textrm{curl}^{-1} g_j\right)_H,\\
A^{4}(s)&=2f'\left(\mathscr{M}(u)\right) \left( \left(\textrm{curl}^{-1}u_s,\textrm{curl}^{-1} f_s  \right)_{L^{2}}+\frac{1}{2}\sum_{j=1}^{\infty}\left\|(-\Delta)^{-\frac{1}{2}}g_j\right\|_{L^{2}}^{2} \right)+f''\left(\mathscr{M}(u)\right)\sum_{j=1}^{\infty} \left|\left( \textrm{curl}^{-1}u_s,  \textrm{curl}^{-1}\ g_j\right)_{L^{2}}\right|^{2}    ,\\
B_j^{4}(s)&= 2 f'\left(\mathscr{M}(u)\right) \left(\textrm{curl}^{-1}u,\textrm{curl}^{-1}\ g_j\right)_{L^{2}}.
\end{align*}
\end{rem}

\begin{rem}[Casimir invariants]
The functional $F[u]=\mathscr{C}(u)=||f(\textrm{curl}^{-1}u)||_{L^{2}}^{2}$ for $f\in C^{2}(\mathbb{R})$ belongs to $C^{2}(H; \mathbb{R})$ and satisfies the conditions (B.3) and (B.4) with derivatives 
\begin{align*}
\mathscr{C}'[u]=\nabla^{\perp}(-\Delta)^{-1} f'(\textrm{curl}^{-1}u),\quad \mathscr{C}''[u]=\left(\nabla^{\perp}(-\Delta)^{-1}f''(\textrm{curl}^{-1}u)\right)\textrm{curl}^{-1}.
\end{align*}
The It\^o formula (B.5) forms 
\begin{align}
\mathscr{C}(u)(t\wedge \tau_n)=\mathscr{C}(u)(0)+\int_{0}^{t\wedge \tau_n}A(s)ds+\sum_{j=1}^{\infty}\int_{0}^{t\wedge \tau_n}B_{j}(s)d\beta_j(s),
\end{align}
for the processes 
\begin{align*}
A(s)&=(f'(\textrm{curl}^{-1}u),\textrm{curl}^{-1}f_s)_{L^{2}}+\frac{1}{2}\sum_{j=1}^{\infty}\left( (\nabla^{\perp}(-\Delta)^{-1} f''(\textrm{curl}^{-1}u)) \textrm{curl}^{-1}g_j,\textrm{curl}^{-1}g_j \right)_{L^{2}}    ,\\
B_j(s)&= (f'(\textrm{curl}^{-1}u), \textrm{curl}^{-1}g_j)_{L^{2}}.
\end{align*}
\end{rem}

\subsection{Local time and Krylov's estimate}

We consider $\mathbb{R}$-valued stochastic processes $y_t$ satisfying 
\begin{align}
y_t=y_0+\int_{0}^{t}x_sds+\sum_{j=1}^{\infty}\int_{0}^{t}\theta_{s}^{j}d\beta_j(s),\quad t\geq 0,\ a.s.,
\end{align}
for $\mathcal{G}_t$-adapted processes $x_t$, and $\theta^{j}_t$ satisfying 
\begin{align}
\mathbb{E}\int_{0}^{t}\left(|x_s|+\sum_{j=1}^{\infty}|\theta^{j}_s|^{2}\right)ds<\infty\quad \textrm{for any}\ t>0.
\end{align}
For such a process $y_t$, there exists a local time $\Lambda_t(a,\omega)$ satisfying the identity 
\begin{align}
\int_{0}^{t}1_{[\alpha,\beta]}(y_s)\left(\sum_{j=1}^{\infty}|\theta^{j}_s|^{2} \right)ds=2\int_{\alpha}^{\beta}\Lambda_t(a,\omega)da,\quad t\geq 0,\ a.s.,
\end{align}
and the change of the variable formula \cite[A.8]{Kuk12},
\begin{align}
(y_t-a)_+=(y_0-a)_+\sum_{j=1}^{\infty}\int_{0}^{t}1_{[a,\infty)}(y_s)\theta^{j}_s d\beta_j(s)+\int_{0}^{t}1_{[a,\infty)}(y_s)x_s ds+\Lambda_t (a,\omega),\  t\geq 0,\ a\in \mathbb{R},\ \textrm{a.s.}
\end{align}

\begin{thm}
Assume that $y_t$ and $x_t$ are stationary processes. Then,  
\begin{align}
\mathbb{E}\left(1_{[\alpha,\beta]}(y_s)\sum_{j=1}^{\infty}|\theta^{j}_s|^{2} \right)
=-2\mathbb{E}\int_{\alpha}^{\beta}1_{[a,\infty)}(y_s)x_sda,\quad t\geq 0,\ \alpha<\beta.
\end{align}
\end{thm}

\begin{proof}
By taking the mean of (B.26),
\begin{align*}
\mathbb{E}\int_{0}^{t}1_{[\alpha,\beta]}(y_s)\left(\sum_{j=1}^{\infty}|\theta^{j}_s|^{2} \right)ds=2\mathbb{E}\int_{\alpha}^{\beta}\Lambda_t(a,\omega)da.
\end{align*}
By taking the mean of (B.27) and integrating it in $[\alpha,\beta]$ for $a$,
\begin{align*}
0=\int_{0}^{t}\mathbb{E}\int_{\alpha}^{\beta}1_{[a,\infty)}(y_s)x_s dsda+\int_{\alpha}^{\beta}\mathbb{E}\Lambda_t (a,\omega)da.
\end{align*}
We eliminate the local time term by substituting this into the first equality. By differentiating the obtained equality with respect to time, we obtain (B.28). 
\end{proof}

For an $\mathbb{R}^{n}$-valued stochastic process $y_t=(y^{k}_t)_{1\leq k\leq n}$ satisfying (B.24) for $x_t=(x^{k}_t)_{1\leq k\leq n}$ and $\theta_{s}^{j}=(\theta_{s}^{jk})_{1\leq k\leq n}$, we set a non-negative symmetric matrix $\sigma=(\sigma^{kl})_{1\leq k,l\leq n}$ by 
\begin{align*}
\sigma^{kl}=\sum_{j=1}^{\infty}\theta_{s}^{jk}\theta_{s}^{jl},\quad 1\leq k,l\leq n.
\end{align*}
The following Krylov's estimate holds \cite[A.8]{Kuk12}.

\begin{thm}[Krylov's estimate]
Assume that $y_t$ and $x_t$ are stationary processes. Then, there exists a constant $C_n$ such that 
\begin{align}
\mathbb{E}\int_{0}^{1}\left(1_{\Gamma}(y_t)(\textrm{det}\ \sigma)^{\frac{1}{n}} \right)dt\leq C_nl_n(\Gamma)^{\frac{1}{n}} \mathbb{E}|x_0|,\quad \Gamma \in \mathcal{B}(\mathbb{R}^{n}),
\end{align}
where $l_n$ denotes the $n$-dimensional Lebesgue measure.
\end{thm}

\section{Gr\"{o}nwall's inequality}

\begin{prop}[Gr\"{o}nwall's inequality]
Assume that $\varphi\in C[0,\infty)$ satisfies 
\begin{align}
\varphi(t)+m \int_{0}^{t}\varphi(s)ds\leq \varphi(0)+mCt,\quad t\geq 0,
\end{align}
with some constant $m>0$. Then,
\begin{align}
\varphi(t)\leq e^{-mt}\varphi(0)+C,\quad t\geq 0.
\end{align}
\end{prop}

\begin{proof}
The function $\phi(t)=e^{-mt}\varphi(0)+C(1-e^{-mt})$ satisfies
\begin{align*}
\phi(t)+m \int_{0}^{t}\phi(s)ds= \varphi(0)+mCt,\quad t\geq 0.
\end{align*}
We show that $\varphi(t)\leq \phi(t)$ for $t\geq 0$. By (C.1), $h(t)=\varphi(t)-\phi(t)$ satifies 
\begin{align*}
h(t)+m \int_{0}^{t}h(s)ds\leq 0,
\end{align*}
and hence $\int_{0}^{t}h(s)ds\leq 0$. Thus, there exists $t_0>0$ such that $h(t)\leq 0$ for $0\leq t\leq t_0$. By repeating this argument, we conclude that $h(t)\leq 0$ for all $t\geq 0$.
\end{proof}

\bibliographystyle{alpha}
\bibliography{ref}
\end{document}